\documentclass{amsart}
\usepackage{amsmath, amsthm,amssymb}
\usepackage[all]{xy}
\xyoption{all}
\usepackage{hyperref}
\hypersetup{
    colorlinks,
    citecolor=blue,
    filecolor=black,
    linkcolor=blue,
    urlcolor=blue
}

\newtheorem{theorem}{Theorem}
\newtheorem{lemma}{Lemma}
\newtheorem{corollary}{Corollary}
\newtheorem{definition}{Definition}
\newtheorem{proposition}{Proposition}
\newtheorem{convention}{Convention}

\newcommand{\isom}{{\,\cong\,}}
\newcommand{\tensor}{\otimes}
\newcommand{\btensor}{\hat{ \otimes} }
\newcommand{\homotopic}{\simeq}
\newcommand{\bHom}{\operatorname{Hom}^{bdd}}
\newcommand{\Hom}{\operatorname{Hom}}
\newcommand{\Ext}{\operatorname{Ext}}
\newcommand{\Tor}{\operatorname{Tor}}
\newcommand{\im}{\operatorname{im}}
\newcommand{\coker}{\operatorname{coker}}
\newcommand{\N}{\mathbb{N}}
\newcommand{\C}{\mathbb{C}}
\newcommand{\R}{\mathbb{R}}
\newcommand{\Q}{\mathbb{Q}}
\newcommand{\Z}{\mathbb{Z}}
\newcommand{\B}{\mathcal{B}}
\newcommand{\I}{{\mathcal{I}}}
\newcommand{\cR}{{\mathcal{R}}}

\newcommand{\BGH}{{\mathcal{H}_{\B,w}\left(G/H\right)}}
\newcommand{\BGHi}{{\mathcal{H}_{\B,w}\left(G/H_\alpha\right)}}
\newcommand{\BGRH}{{\mathcal{H}_{\B,w}\left(G/\mathcal{H}\right)}}
\newcommand{\BD}{{\Delta_\mathcal{B}}}
\newcommand{\BG}{{\mathcal{H}_{\B,L}\left(G\right)}}
\newcommand{\BGone}{{\mathcal{H}_{\B,L_1}\left(G_1\right)}}
\newcommand{\BGtwo}{{\mathcal{H}_{\B,L_2}\left(G_2\right)}}
\newcommand{\BGthree}{{\mathcal{H}_{\B,L_3}\left(G_3\right)}}
\newcommand{\BH}{{\mathcal{H}_{\B,L}\left(H\right)}}
\newcommand{\BHi}{{\mathcal{H}_{\B,L}\left(H_\alpha\right)}}
\newcommand{\BR}{{\mathcal{H}_{\B,L}\left(R\right)}}

\newcommand{\CGRH}{{\mathbb{C}[G/{\mathcal{H}}]}}

\newcommand{\ttensor}{{\hat\otimes}}

\newcommand{\RH}{\mathcal{H}}

\allowdisplaybreaks[4]

\begin{document}

\title{{$\B$}-Bounded Cohomology and Applications}

%
\author{RONGHUI JI} 
\address{Department of Mathematical Sciences, IUPUI, 402 N. Blackford St.\\
         Indianapolis, IN 46202 USA
}
\email{ronji@math.iupui.edu}
%
\author{CRICHTON OGLE} 
\address{Department of Mathematics,The Ohio State University, 231 W. 18th Ave.\\
         Columbus, OH 43210 USA
}
\email{ogle@math.ohio-state.edu}
%
\author{BOBBY W. RAMSEY} 
\address{Department of Mathematics,The Ohio State University, 231 W. 18th Ave.\\
         Columbus, OH 43210 USA
}
\email{ramsey.313@math.osu.edu}

\begin{abstract} 
A discrete group with word-length $(G,L)$ is $\B$-isocohomological for a bounding classes $\B$ if the comparison map 
from $\B$-bounded cohomology to ordinary cohomology (with coefficients in $\C$) is an isomorphism; it is strongly 
$\B$-isocohomological  if the same is true with arbitrary coefficients. In this paper we establish some basic conditions 
guaranteeing strong $\B$-isocohomologicality. In particular, we show strong $\B$-isocohomologicality for an $FP^{\infty}$ 
group $G$ if all of the weighted $G$-sensitive Dehn functions are $\B$-bounded. Such groups include all $\B$-asynchronously 
combable groups; moreover, the class of such groups is closed under constructions arising from groups acting on an acyclic 
complex. We also provide examples where the comparison map fails to be injective, as well as surjective, and give an example 
of a solvable group with quadratic first Dehn function, but exponential second Dehn function.  Finally, a relative theory of 
$\B$-bounded cohomology of groups with respect to subgroups is introduced. Relative isocohomologicality is determined in terms 
of a new notion of {\it relative} Dehn functions and a {\it relative} $FP^\infty$ property for groups with respect to a 
collection of subgroups.  Applications for computing $\B$-bounded cohomology of groups are given in the context of relatively 
hyperbolic groups and developable complexes of groups.
\end{abstract}

\maketitle



\tableofcontents


\section{Introduction}

To a bounding class of functions $\B$, a group with length $(G,L)$ and a weighted $G$-complex $(X,w)$ 
one can associate the $\B$-bounded, $G$-equivariant cohomology of $X$ with coefficients in a $\BG$-module $V$:
\[
{\B}H^*_G(X;V)
\]
The construction is a variant on that used in the non-bounded case. Starting with $(G,L)$ and $\B$, 
one constructs a bornological algebra $\BG$ - a completion of the group algebra 
${\C}[G]$ -  and a bornological $\BG$-module ${\B}C_*(EG\times X)$
 - a similar completion of the the complex of singular chains $C_*(EG\times X)$. Then given a bornological 
 $\BG$-module $V$, one forms the (co)complex
\[
\bHom_{\BG}({\B}C_*(EG\times X),V)
\]
of bounded $\BG$-module homomorphisms. The cohomology groups ${\B}H^*_G(X;V)$ 
are then defined as the (algebraic) cohomology groups of this complex. There is a natural transformation of functors
\[
{\B}H^*_{_-}(_-;_-)\to H^*_{_-}(_-;_-)
\]
which, for given values, is referred to as the \underline{comparison map}
\[
\Phi_{\B}^* = \Phi_{{\B},G}^*(X;V) : {\B}H^*_G(X;V)\to H^*_G(X;V)
\]

One wants to know the properties of this map; when it is injective, surjective, and what structures are preserved 
under it. In some special cases, quite a bit is known. For example, taking 
${\B}={\B}_{\min} = \{\text{constant functions}\}$ yields $\RH_{{\B}_{\min},L}(G) = \ell^1(G)$, 
and the resulting cohomology theory is simply the equivariant bounded cohomology of $X$ with coefficients in the Banach 
$\ell^1(G)$-module $V$. At the other extreme, when 
${\B}= {\B}_{max} = \{f:\mathbb R_+\to \mathbb R^+\ |\ f\,\,\text{non-decreasing}\}$, the comparison 
map becomes an isomorphism under appropriate finiteness conditions: $G$ is an $FP^{\infty}$ group and $X$ is a $G$-complex 
with finitely many orbits in each dimension. More interesting, 
and also more subtle, is the case when the bounding class $\B$ lies between these two extremes, because it is 
in this range that the weight function on $X$ and word-length function on $G$ have the potential for influencing the  
$\B$-bounded cohomology groups in a non-trivial way. To illustrate why this is of interest, we consider two applications.

\begin{itemize} 
\item The topological $K$-theory of $\ell^1(G)$. Here the bounding class ${\B} = {\mathcal{P}} = $ 
	\{non-decreasing polynomials\} is of particular interest, as $\RH_{{\mathcal{P}},L}(G)$ is a smooth subalgebra 
	of $\ell^1(G)$. Taking $X = pt$ and $V = {\C}$, the image of $\Phi^*_{\B}$ consists precisely 
	of those cohomology classes in $H^*(G) = H^*(G;\C )$ which are polynomially bounded with respect to the 
	word-length function on $G$. Such classes therefore satisfy the $\ell^1$-analogue of the Novikov Conjecture, and 
	there is reason to believe they also satisfy the Strong Novikov Conjecture (without any additional Rapid Decay 
	condition on the group). For this application, one would like the comparison map to be \underline{surjective}.

\item The Strong $\ell^1$-Bass Conjecture. Here one is interested in determining the image of the Chern character 
	$ch_*: K_*^t(\ell^1(G))\to HC_*^{top}(\RH_{{{\mathcal{P}}},L}(G))$ \cite{JOR1}. In this case, the problem of showing 
	that the image lies in the ``elliptic summand" amounts to proving the \underline{injectivity} of the comparison map (for suitable choice of $X$). 
\end{itemize}

Of course, both properties hold when $\Phi^*_{\B}$ is bijective, i.e., when $G$ or more precisely $(G,L)$ is 
\underline{isocohomological} \cite{M1, M2}. In fact, up until this point, all proofs of either injectivity or surjectivity for a 
given bounding class  have arisen by a verification of this stronger isocohomological property. The first type of result 
in this direction is due to the first author, who showed in \cite{J1} that 
$HC_*^{top}(H_{{{\mathcal{P}}},L}(G))\isom HC_*(\C[G])$ for groups of polynomial growth. Subsequently, it was 
determined independently by the second author in \cite{O1} and R. Meyer in \cite{M1} that $(G,L)$ is ${\mathcal{P}}$-isocohomological 
in the case $G$ admits a synchronous combing. Moreover, in \cite{O1} it was shown that isocohomologicality with arbitrary 
coefficients, or \underline{strong $\B$-isocohomologicality},(at least for $\B = \mathcal{P}$) followed for $HF^{\infty}$ groups whenever all of the 
\underline{Dehn functions} (as defined in \cite{O1}) were polynomially bounded. This last result was significantly strengthened 
by the first and third authors in \cite{JR1}, where it was established (again for $\B = \mathcal{P}$), that strong 
$\mathcal{P}$-isocohomologicality for an $FP^{\infty}$ group was actually equivalent to the existence of polynomially 
bounded Dehn functions in each degree. From this the authors were able to conclude that the comparison map (with $X=pt$), 
fails to be surjective for rather simple groups when one allows non-trivial coefficients. However, still unknown for the 
standard word-length function on $G$ and the polynomial bounding class ${\mathcal{P}}$, or more generally any bounding 
class $\B$ containing the linear polynomials $\mathcal{L}$, were answers to the following questions:

\begin{itemize}
\item[Q1.] Is the comparison map $\Phi_{\B}^* = \Phi^*_{\B}(G): {\B}H^*(G)\to H^*(G)$ always surjective?

\item[Q2.] Is $\Phi^*_{\B}(G)$ always injective?

\item[Q3.] If $G$ is an $HF^{\infty}$ group, are the higher Dehn functions of $G$ ${\mathcal{P}}$-equivalent to the first Dehn function of $G$?
\end{itemize}

In fact, it was conjectured by the second author in \cite{O1} that the answer to the third question was ``yes", given that all known 
examples at the time suggested this to be the case. Nevertheless, one of the consequences of the results of this paper is

{\bf\underline{Theorem A}} The answer to each of these questions is ``no".

Precisely, we show

\begin{itemize}
\item[A1.] (following Gromov) There exists a compact, closed, orientable $3$-dimensional solvmanifold $M_1^3\homotopic K(G_1,1)$ 
	with a $2$-dimensional class $t_2\in H^2(M_1) = H^2(G_1)$ which is not $\B$-bounded for any 
	${\B}\prec {{\mathcal{E}}}$, the bounding class of simple exponential functions.

\item[A2.] There exists a compact, closed, orientable $5$-dimensional solvmanifold $M_2^5\homotopic K(G_2,1)$, where the first 
	Dehn function of $G_2$ is quadratic, but the second Dehn function is at least simple exponential.

\item[A3.] If $\B$ is a bounding class with ${\B}\succeq {{\mathcal{L}}}$, and $G$ is a finitely-presented $FL$ group (meaning $BG$ 
	is homotopy-equivalent to a finite complex) for which the comparison map $\Phi^*_{\B}$ is not surjective with respect to the standard 
	word-length function (as in (A1.)), there is another discrete group ${\mathfrak F}(G)$ depending on $G$ ``up to homotopy" for which 
	the comparison map is not injective with respect to the standard word-length function on ${\mathfrak F}(G)$. Moreover, if 
	${\B} \succeq {\mathcal P}$, then ${\mathfrak F}(G)$ can be taken to be also of type $FL$.
\end{itemize}

Somewhat surprising is the sharpness of these results. For (A1.), this is the simplest type of finitely-presented group and smallest 
cohomological dimension in which surjectivity with trivial coefficients can fail. In the case of (A2.), we note that a linear first 
Dehn function implies the group is hyperbolic, in turn implying that all of the higher Dehn functions are also linear. Moreover, there 
is an isoperimetric gap that occurs between degree one and two, so that if the first Dehn function is not linear, it must be at least 
quadratic - the smallest degree for which the second Dehn function could be non-polynomial, or even non-linear. Finally, (A3.) provides, 
for each bounding class $\B \succeq \mathcal{P}$ an injection from the set of (isomorphism classes of) finitely-presented $FL$ groups 
with non-surjective comparison map to the set of (isomorphism classes of) finitely-presented $FL$ groups with non-injective comparison map.

Up until now, the main problem in studying either $\B$-isocohomologicality or strong $\B$-isocohomologicality has been the absence of appropriate computational tools.  The difficulty in extending classical techniques lies in the analysis of the group structures and the geometry of the associated Cayley graph. The primary goal of 
this paper is to develop systematic methods for calculating $\B$-bounded cohomology, and to establish isocohomologicality for a good
class of groups. 
One such class consists of group extensions where the normal and quotient groups are isocohomological with respect to the restricted and quotient
length functions induced by the length function on the middle group.  The other 
main class of groups considered are those associated with developable complexes of groups; this includes the class
of relatively hyperbolic groups.  The main computational techniques introduced are the Hochschild-Serre spectral sequence in $\B$-bounded cohomology
associated to an extension of groups equipped with length functions, and the Serre spectral sequence in $\B$-bounded 
cohomology for developable complexes of groups.  As has been noted by Meyer in \cite{M2}, the category in which one does
homological algebra in the bornological framework is almost never abelian, which makes the extension of homological techniques from the non-bounded to the bounded setting problematic.  Among the results included below, we have

{\bf\underline{Theorem B}} 
Let $G$ be a finitely presented group acting cocompactly on a contractible simplicial complex $X$ without inversion, with finitely generated stabilizers 
$G_\sigma$ for each vertex $\sigma$ in $X$, and with finite edge stabilizers. Suppose also that $X$ is equipped with the 
$1$-skeleton weighting, and all of its higher 
weighted Dehn functions are $\B$-bounded.  Then if each $G_\sigma$ is strongly $\B$-isocohomological, $G$ is as well.

In \cite{Os} and \cite{BC} the notion of the first unweighted `relative Dehn function' is introduced for a group relative to a family of subgroups.  
This relative Dehn function is well-defined for developable complexes of groups with finite edge stabilizers; in particular for relatively 
hyperbolic groups.  Intuitively, the unweighted relative Dehn functions bound `relative fillings' of relative cycles in a 
`relatively contractible space'.  Thus, one expects that the comparison map from $\B$-bounded relative cohomology to non-bounded relative cohomology
is an isomorphism when the relative Dehn functions are appropriately bounded; i.e. the group is \underline{relatively isocohomological} with
respect to the family of subgroups.  

Now it should be noted that the existence of Dehn functions, even in the absolute setting, is not guaranteed by the existence of a 
nice resolution. In general, one needs to work with weighted Dehn functions where the weighting degreewise is either equivalent to a 
weighted $\ell^1$-norm associated to a proper weight function on an additive set of generators, or is one over which there exists 
some geometric control. Using the 1-skeleton weighting, we show

{\bf\underline{Theorem C}} 
Suppose the finitely presented group $G$ is $FP^\infty$ relative to a finite family of 
finitely generated subgroups $\RH$.  Then the following are equivalent.
\begin{enumerate}
	\item[(1)] The weighted relative Dehn functions of $EG$ relative to $E\RH$ are $\B$-bounded.
	\item[(2)] $G$ is strongly $\B$-isocohomological relative to $\RH$.
	\item[(3)] The comparison map $\B H^*(G,\RH;A) \to H^*(G,\RH; A)$ is surjective for all bornological
					$\BG$-modules $A$.
\end{enumerate}

We should also remark that, as in the case of relative group cohomology \cite{Au,BE2}, there is a long-exact sequence in $\B$-bounded cohomology 
relating the bounded cohomologies of the subgroups and the group with the $\B$-bounded relative cohomology of the pair.  

An outline of the paper is as follows. 

In section 2, we recall from \cite{JOR1} some basic terminology regarding bounding classes, and the setup for defining the $G$-equivariant 
$\B$-bounded cohomology of bornological algebras and weighted complexes. We also define what we mean by a \underline{Dehn function} in this 
paper. For $FP^\infty$ groups, 
there are a number of different ways for defining Dehn functions, the most natural being algebraically defined Dehn functions which take into account the action of the group $G$. Also in this section we construct 
some basic pairing operations between $\B$-bounded homology and cohomology that are used later on. 

In section 3, we show\footnote{This result has also recently been obtained in \cite{BRS}, using more geometric techniques.} that asynchronously 
combable groups are type $HF^{\infty}$, via an explicit coning argument that allows us in section 3.2 to show that if the combing lengths are 
$\B$-bounded, so are all of the Dehn functions of $G$. We also extend the main result of \cite{JR1} to arbitrary bounding classes $\B$.

In section 4, we generalize the constructions and results of section 3 to the relative context. We begin by establishing a relative version of 
the Brown and Bieri-Eckmann conditions used to establish homological or homotopic finiteness through a given degree. In section 4.2, we introduce 
the higher dimensional relative Dehn functions in several different, but equivalent ways.  In section 4.3 the notion of relative $\B$-bounded cohomology is
developed and shown to fit into a long-exact sequence similar to the long-exact sequence in non-bounded relative group cohomology.  The notion of 
relative  $\B$-isocohomologicality is introduced and related to the higher relative Dehn functions.  This relationship
is then examined in the case of relatively hyperbolic groups and groups acting on complexes.

In section 5, we construct in ${\B}H^*( \cdot )$ the i) Hochschild-Serre spectral sequence associated to an extension 
of groups with word-length, and ii) the spectral sequence associated to a group acting on a complex. These spectral sequences closely resemble 
their non-bounded counter-parts, and, as we mentioned above, provide the main tools for computing $\B$-bounded cohomology.

Finally, in section 6, we examine specific examples, beginning with duality groups (in the sense of Bieri-Eckmann). A striking fact, proved in 
section 6.2, is that $\B$-isocohomologicality for an oriented $n$-dimensional Poincare Duality group $G$ is guaranteed by the $\B$-boundedness 
of a single cohomology class in $H^n(G\times G)$ coming from the restriction of the Thom class for the diagonal embedding, i.e., the ``dual" 
fundamental class of $G$. For fundamental groups of compact oriented manifolds with connected boundary, $\B$-isocohomologicality is guaranteed 
by a $\B$-bound on two separate cohomology classes.  In section 6.3, we introduce the notion of a $\B$-duality group, and show that when the 
fundamental homology class in $\B$-bounded homology is in the image of the homology comparison map, the cohomology comparison map is injective 
for all coefficients. Finally, in sections 6.4 and 6.5, we discuss the examples mentioned in (A1.) - (A3.) above.

The second author would like to thank Ian Leary for an illuminating remark regarding \cite{L}.  The first and 
third authors are grateful to Denis Osin for his communications relating to complexes of groups, relative Dehn
functions, and the Arzhantseva-Osin example of an exponential growth solvable group that has quadratic
first Dehn function, \cite{AO}.  The authors would also like to thank the referee for their careful reading of
this paper and for their helpful suggestions.


\section{Preliminaries}
We discuss some constructions and terminology that will be used throughout the paper.

\subsection{Bounding Classes} 

Let ${\mathcal{S}}$ denote the set of non-decreasing functions
$\{f:\R_+\to \R^+\}$. Suppose $\phi:{{\mathcal{S}}}^n\to
{{\mathcal{S}}}$ is a function of sets, and ${\B}\subset {{\mathcal{S}}}$.
We will say that $\B$ is \underline{weakly closed} under the
operation $\phi$ if for each $n$-tuple $(f_1,\dots,f_n)\in {\B}^n$, 
there exists an $f\in {\B}$ with $\phi(f_1,f_2,\dots,f_n)
<  f$.  A \underline{bounding class} then is a subset of ${\B}
\subset {{\mathcal{S}}}$ satisfying
\begin{itemize}
\item[(${\B}C$1)] it contains the constant function $1$,
\item[(${\B}C$2)] it is weakly closed under the operations of taking positive
rational linear combinations
\item[(${\B}C$3)] it is weakly closed under the operation $(f,g)\mapsto
f\circ g$ for $g\in{\mathcal{L}}$, where ${\mathcal{L}}$ denotes the linear
bounding class $\{f(x) = ax+b\ |\ a,b\in \Q_+\}$.
\end{itemize}
Naturally occurring classes besides ${\mathcal{L}}$ are ${\B}_{min} =
\{\Q_+\}$, ${\mathcal{P}} =$ the set of polynomials with non-negative
rational coefficients, the set ${{\mathcal{E}}} = \{e^f\ ,\ f\in
{\mathcal{L}}\}$, and ${\B}_{max} = {\mathcal{S}}$. A
bounding class is \underline{multiplicative} if it is weakly closed
under multiplication, and \underline{composable} if it is weakly
closed under composition. More generally, given bounding classes $\mathcal B$ and ${\mathcal B}'$, we say 
$\mathcal B$ is a \underline{left resp.\ right ${\mathcal B}'$-class} if $\mathcal B$ is weakly closed under 
left resp.\ right composition with elements of ${\mathcal B}'$ (thus, for example, all bounding classes are 
right $\mathcal L$-classes by (${\B}C$3)).

Basic properties of bounding classes were discussed in \cite{JOR1}; for technical reasons only composable 
bounding classes were considered in that paper, however, all of the results of  of [JOR1, \S 1.2] apply for 
this larger collection of classes. We write ${\B}'\preceq {\B}$ if every $f\in{\B}'$ is bounded above by 
some element $f\in{\B}$, with equivalence ${\B}'\sim {\B}$ if ${\B}'\preceq {\B}$ and ${\B}\preceq {\B}'$. 
Finally, ${\B}'\prec {\B}$ if ${\B}'\preceq {\B}$ but ${\B}'$ is not equivalent to $\B$.

A bounding class $\B$ is \underline{weakly countable} if there is a countable bounding class $\B'$
with $\B \sim \B'$.  
\begin{convention}
All bounding classes considered will be assumed weakly countable.
\end{convention}
Since equivalent bounding classes yield isomorphic results in what follows,
this amounts to working with countable bounding classes.


Given a weighted set
$(X,w)$ and $f\in{\mathcal{S}}$, the seminorm $|\,\, |_f$ on
$\Hom(X,\C )$ is given by $|\phi |_f := \sum_{x\in X}
|\phi(x)|f(w(x))$. Then $H_{{\B},w}(X) = \{f:X\to\C \ |
\ |\phi |_f < \infty\,\, \forall f\in{\B}\}$. The most important feature of $H_{{\B},w}(X)$ is that it is an
algebra whenever $X$ is a semi-group and $w$ is sub-additive with
respect to the multiplication on $X$. If $X$ has a unit, then so
does $H_{{\B},w}(X)$. We will mainly be concerned with the
case $(X,w) = (G,L)$ is a discrete group equipped with a length
function $L$ (meaning a function $L:G\to\R_+$ subadditive
with respect to the multiplication on $G$, and invariant under the
involution $g\mapsto g^{-1}$). The length function $L$ is
called a \underline{word-length function} (with respect to a
generating set $S$) if $L(1) = 1$ and there is a function
$\phi:S\to \R^+$ with
\[
L(g) = min\left\{\sum_{i=1}^n \phi(x_i)\ |\ x_i\in S,\, x_1x_2\dots x_n = g\right\}
\]
When the generating set $S$ is finite, taking $\phi = 1$ produces
the standard word-length function on $G$.


\subsection{The $FP^{\alpha}$ and $HF^{\alpha}$ conditions}

In this paper the term \underline{complex} will refer either to a simplicial complex, polyhedral complex, or simplicial 
set. For a complex $X$, we say $X$ is \underline{type $HF^{\alpha}$} ($\alpha\leq \infty$) if $|X|\homotopic |Y|$ where 
$Y$ is a $CW$ complex with finitely many cells through dimension $\alpha$. This notion clearly defines an equivalence 
relation on the appropriate category of complexes, and admits an equivariant formulation: for a discrete group $G$ 
which acts either cellularly or simplicially, a $G$-complex $X$ is \underline{type $G-HF^{\alpha}$} if there is a 
strong $G$-homotopy equivalence $X\homotopic Y$ with $Y$ having finitely many $G$-orbits through dimension $\alpha$. 
A group $G$ is type $HF^{\alpha}$ if its classifying space $BG$ is type $HF^{\alpha}$, or equivalently, if $EG$ is type $G-HF^{\alpha}$. 

When $\alpha$ is finite and $X$ is a simplicial complex resp. simplicial set resp. polyhedral complex, the $HF^{\alpha}$ condition is equivalent to saying $X\homotopic Y$ a simplicial complex resp. simplicial set resp. polyhedral complex with $Y^{(\alpha)}$ finite. When $\alpha = \infty$ and $X$ is a polyhedral complex, then $X$ is type $HF^{\infty}$ iff $X\homotopic Y$ a polyhedral complex with $Y^{(n)}$ finite for all $n < \infty$. However, if $X$ is either a simplicial set or simplicial complex, the $HF^{\infty}$ condition is equivalent to the weaker statement $X\homotopic Y = \varinjlim Y_n$ a direct limit of simplicial sets resp. simplicial subcomplexes, with the inclusion $Y_n\hookrightarrow Y$ inducing an $n$-connected map of spaces $|Y_n|\hookrightarrow |Y|$.

For discrete groups, the condition $HF^{\alpha}$ is equivalent to requiring that $G$ is finitely presented and type 
$FP^{\alpha}$. The standard $FP^{\alpha}$ condition - that $\Z$ admits a resolution over $\Z [G]$ 
which is finitely-generated projective through dimension $\alpha$ - is now known to be strictly weaker than requiring 
$G$ to be $HF^{\alpha}$ when $G$ is not finitely-presented. Because the framework used in this paper for defining Dehn 
functions is algebraic, our primary focus will be on discrete groups of type $FP^{\alpha}$. Again, we remind the reader that $FP^{\alpha}$ is equivalent to $FF^{\alpha}$ - the condition that $\Z$ admits a resolution over $\Z [G]$ by free $\Z [G]$-modules which
which are finitely-generated through dimension $\alpha$


\subsection{$\B$-homology and $\B$-cohomology of algebras}

There are a number of different settings in which one can develop
the theory of $\B$-bounded cohomology for non-discrete
algebras. For arbitrary $\B$, the most natural is the bornological
framework introduced by \cite{M2, M3}. Given a bornological
algebra $A$ and bornological $A$-modules $M$ and $N$, the derived
functors $\Tor^A_*(M,N)$ and $\Ext_A^*(M,N)$ are computed in the
bornological category using standard constructions from homological
algebra, with the constraint that projective or injective
resolutions used are contractible via a bounded linear
contraction. 

Before defining homology and cohomology, we want to point out that
in both cases there is a ``bornologically correct''
reduced theory and an ``algebraically correct'' unreduced
theory. In general, given a chain complex $(C_*,d_*)$, or a
cocomplex $(D^*,d^*)$, of bornological vector spaces with
bounded differential, we distinguish between the algebraic
(co)homology groups
\[
H_*^a(C_*) := \{H_n(C_*) =\ker(d_n)/ \im(d_{n+1})\},\quad
H^*_a(D^*):= \{H^n(D^*) = \ker(d^n)/\im(d^{n-1})\}
\]
and the bornological (co)homology groups
\[
H_*^b(C_*) := \{H_n(C_*) = \ker(d_n)/\overline{\im(d_{n+1})}\},\quad
H^*_b(D^*) := \{H^n(D^*) = \ker(d^n)/\overline{\im(d^{n-1})}\}
\]

Thus given a bornological algebra, $A$, with unit, bornological
$A$-modules $M_1$, $M_2$,  projective resolutions $P^i_{\bullet}$
of $M_i$ over $A$, and an injective resolution $Q_{\bullet}$ of
$M_2$ over $A$
\begin{gather}
\Tor_*^{A,x}(M_1,M_2) :=
H^x_*(P^1_{\bullet}\widehat{\underset{A}\tensor} M_2) =
H^x_*(M_1\widehat{\underset{A}\tensor} P^2_{\bullet}),\,\,x = a,b\\
\Ext^*_{A,x}(M_1,M_2) := H^*_x(\bHom_A(P^1_{\bullet},M_2)) =
H^*_x(\bHom_A(M_1,Q_{\bullet})),\,\, x = a,b
\end{gather}
Here $\bHom_A({_-})$ denotes (in each degree) the
bornological vector space of bounded $A$-module homomorphisms.



\subsection{$\B$-homology and $\B$-cohomology of weighted complexes}

A \underline{weight function} on a set $S$ is a map of sets
$w:S\to\R_+$. Fix a weighted set $(S,w)$, and write
$\C[S]$ for the vector space over $\C$ with basis
$S$. For a bounding class $\B$, we may define seminorms on
$\C[S]$ by
\begin{equation*}
\left|\sum_{s\in S}\alpha_s s\right|_f := \sum_{s\in S} |\alpha_s| f(w(s)),\quad f\in{\B}
\end{equation*}
If $(G,L)$ is a discrete group with length function $L$, a
\underline{weighted $G$-set} is a weighted set $(S,w)$ with a
$G$-action on $S$ satisfying
\begin{equation}\label{eqn:weight}
w(gs)\leq C\cdot L(g) + w(s),\ \,\, \forall g\in G, s\in S
\end{equation}
Let $\mathcal{H}_{{\B},w}(S)$ denote the completion of $\C[S]$
with respect to the seminorms in (3). Then $\mathcal{H}_{{\B},w}(S)$
may be viewed as a bornological vector space, which is Frechet
if there exists a countable bounding class ${\B}'$ with ${\B}
\sim {\B}'$. Note that $\C[S]$ is a module over $\C[G]$ in the usual way: $(\sum\lambda_ig_i)(\sum\beta_j s_j) =
\sum_{i,j}\lambda_i\beta_jg_is_j$.

\begin{proposition} The module structure
of $\C[S]$ over the group algebra $\C[G]$ extends to
a bounded bornological $\mathcal{H}_{{\B},L}(G)$-module structure on
$\mathcal{H}_{{\B},w}(S)$.\end{proposition}
\begin{proof}[Proof] This follows by the same estimates as those used to
show $\mathcal{H}_{{\B},L}(G)$ is an algebra:
\begin{gather*}
\left|\left(\sum\lambda_{g_1} g_1\right)\left(\sum\lambda_{g_2} s_2\right)\right|_f\\
= \sum_{s\in S}\left|\sum_{g_1s_2 = s}\lambda_{g_1}\lambda_{s_2}\right|f(w(s))\\
\leq \sum_{s\in S}\left(\sum_{g_1s_2 = s}\left|\lambda_{g_1}\lambda_{s_2}\right|f(L(g_1)+w(s_2))\right)\\
\leq \sum_{s\in S}\left(\sum_{g_1s_2 = s,L(g_1)\le
w(s_2)}\left|\lambda_{g_1}\lambda_{s_2}\right|f(2w(s_2))\right) +
\sum_{s\in S}\left(\sum_{g_1s_2 = s,w(s_2)\leq L(g_1)}\left|\lambda_{g_1}\lambda_{s_2}\right|f(2L(g_1))\right)\\
\leq \left|\sum\lambda_{g_1} g_1\right|_1\left|\sum\lambda_{s_2} s_2\right|_{f_2} +
\left|\sum\lambda_{g_1} g_1\right|_{f_2}\left|\sum\lambda_{s_2} s_2\right|_1< \infty
\end{gather*}
where $|_-|_1$ denotes the $\ell^1$-norm, and $f_2\in{\B}$ is
any function satisfying $f(2x)\leq f_2(x)\ \forall x$.
\end{proof}

Suppose $w$ and $w'$ are two weight functions on a set $S$.  We write $w \preceq w'$ if there exist positive constants
$A$ and $B$ with $w(s) \leq A w'(s) + B$ for all $s \in S$.  The weight functions are \underline{linearly equivalent}
if $w \preceq w'$ and $w' \preceq w$.  If $w$ and $w'$ are linearly equivalent weight functions on $S$,
then $\mathcal{H}_{{\B},w}(S) = \mathcal{H}_{{\B},w'}(S)$ for any bounding class $\B$.

A \underline{weighted simplicial set}, respectively \underline{weighted
simplicial complex}, $(X,w)$ is a simplicial set, respectively a simplicial complex, $X
= \{X_n\}_{n\ge 0}$ together with weight functions $w_n:X_n\to
\R_+$ such that for each $n$ and $n$-simplex $\sigma$,
$w_{n-1}(\sigma')\leq w_n(\sigma)$ when $\sigma'$ is a face of $\sigma$,
and (in the case of simplicial sets) $w_{n+1}(s_j(\sigma)) = w_n(\sigma)$ where
$s_j$ represents a degeneracy map\footnote{The same definition applies to polyhedral complexes, under
a mild restriction on the number of faces allowed in each
dimension.}. In both cases, we will simply refer to $(X,w)$ as
a \underline{weighted complex}. Given a discrete group with
length $(G,L)$, a \underline{weighted $G$-complex} is a
$G$-complex weighted in such a way that for each $n$, the action
of $G$ on $(X_n,w_n)$ satisfies equation (\ref{eqn:weight}) above. In this case,
completing $C_*(X)$ degreewise produces a bornological chain
complex
\[
{\B}C_*(X) := \{\mathcal{H}_{{\B},w_n}(X_n),d_n\}_{n\ge 0}
\]
When the action of $G$ on $X$ is free, the definition of the $G$-equivariant $\B$-bounded
cohomology of $X$ with coefficients in a bornological $\mathcal{H}_{{\B},L}(G)$-module $A$ is
\begin{gather}
{\B}H^*_{G,x}(X;A) := H^*_x(\Hom_{\mathcal{H}_{{\B},L}(G)}^{cont}({\B}C_*(X)_*,A)),\quad x = a,b\\
{\B}H_*^{G,x}(X;A) := H_*^x({\B}C_*(X)\underset
	{\mathcal{H}_{{\B},L}(G)}{\widehat{\tensor}} A),\quad x = a,b
\end{gather}
Note that
$\bHom_{\mathcal{H}_{{\B},L}(G)}({\B}C_*(X)_*,A) =
\bHom({\B}C_*(X)_*,A)^G$, and ${\B}C_*(X)\underset
{\mathcal{H}_{{\B},L}(G)}{\widehat{\tensor}} A$ identifies with the quotient of ${\B}C_*(X){\widehat{\tensor}} A$ 
by the closure of the image of $\{Id - g\, |\, g\in G\}$ where $g\circ (x\tensor y) = (gx\tensor g^{-1}y)$. In general 
when the action of $G$ is not free, the definition is adjusted in the usual way by first replacing these
$G$-fixed-point and $G$-orbit spaces by the larger equivariant``homotopy fixed-point'' and ``homotopy-orbit'' spaces. 
Let $EG$ denote the homogeneous bar resolution of $G$, with weight function $w(g_0,\dots,g_n) = L(g_0) + \sum_{i=1}^n
L(g_{i-1}^{-1}g_i)$. Then\footnote{the isomorphism of complexes ${\B}C_*(EG)\widehat{\tensor}{\B}C_*(X)_*\isom {\B}C_*(EG\times X)_*$ 
is a special case of a more general equivalence to be established in \cite{JOR3}.}

\begin{gather*}
\bHom({\B}C_*(X)_*,A)^{hG}:=
	\bHom_{\mathcal{H}_{{\B},L}(G)}({\B}C_*(EG),\bHom({\B}C_*(X)_*,A))\\
\isom \bHom_{\mathcal{H}_{{\B},L}(G)}({\B}C_*(EG)\widehat{\tensor}{\B}C_*(X)_*,A)\\
\isom \bHom_{\mathcal{H}_{{\B},L}(G)}({\B}C_*(EG\times X)_*,A);\\
({\B}C_*(X){\widehat{\tensor}} A)_{hG}:= 
	({\B}C_*(EG)\widehat{\tensor}{\B}C_*(X)_*)\underset
	{\mathcal{H}_{{\B},L}(G)}{\widehat{\tensor}} A\\
\isom {\B}C_*(EG\times X)_*\underset{\mathcal{H}_{{\B},L}(G)}{\widehat\tensor}A
\end{gather*}

The $G$-equivariant $\B$-bounded cohomology groups of the weighted complex $(X,w)$
with coefficients in the $\mathcal{H}_{{\B},L}(G)$-module $A$ are given as
\begin{gather*}
{\B}H^*_{G,x}(X;A) := H^*_x(\bHom({\B}C_*(X)_*,A)^{hG}),\quad x = a,b\\
{\B}H_*^{G,x}(X;A) := H_*^x({\B}C_*(X){\widehat{\tensor}} A)_{hG}),\quad x = a,b
\end{gather*}
When the action of $G$ on $X$ is free, these groups agree with those defined above.
They also agree with those given in the previous section in terms of derived functors; they
are simply equalities (1) and (2)  where $M_1$ is the
 (DG) $\mathcal{H}_{{\B},L}(G)$-module ${\B}C_*(X)$. In this context,
\begin{gather}
{\B}H^*_{G,x}(X;A)  = \Ext^*_{\mathcal{H}_{{\B},L}(G),x}({\B}C_*(X),A)\\
{\B}H_*^{G,x}(X;A) = \Tor^{\mathcal{H}_{{\B},L}(G),x}_*({\B}C_*(X),A)
\end{gather}
with ${\B}C_*(EG\times X)$ used as a canonical free resolution of ${\B}C_*(X)$
over $\mathcal{H}_{{\B},L}(G)$ when ${\B}C_*(X)$ is not free over $\mathcal{H}_{{\B},L}(G)$ (i.e.,
when the action of $G$ on $X$ is not free). the inclusion of complexes
\[
C_*(EG\times X)\hookrightarrow {\B}C_*(EG\times X)
\]
induces \underline{comparison maps}
\begin{gather}
\Phi_{\B}^*:{\B}H^*_{G,x}(X;A)\to H^*_G(X;A)\\
\Phi^{\B}_x: H_*^G(X;A)\to {\B}H_*^{G,x}(X;A)
\end{gather}
which are clearly functorial in $X,G$, and $A$.

The $\B$-bounded cohomology groups of $(X,w)$ are computed as
the cohomology of a subcomplex of $C^*(EG\times X;A)^G$ which can be difficult
to describe in general. However, when the action of $G$ is free on
$X$ and trivial on $A$, and $A$ is simply a normed vector space
(e.g., $\C$), then
\[
{\B}H^*_{G,x}(X;A) = {\B}H^*_{x}(X/G;A) = H^*_{x}({\B}C^*(X/G;A))
\]
where
\begin{gather}
{\B}C^*(X/G;A) = \{{\B}C^n(X/G;A)\}_{n\ge 0}\\
{\B}C^n(X/G;A) = \{\phi : (X/G)_n\to A\ |\ \exists f\in{\B}\ s.t.\ |\phi(\overline{x})| <
f(\overline{w}(\overline{x}))\quad\forall \overline{x}\in
(X/G)_n\}
\end{gather}

\begin{convention} Unless otherwise indicated, ${\B}H^*(_-)$ will mean ${\B}H^*_a(_-)$; more generally, 
${\B}H^*_{G}(_-)$ will mean ${\B}H^*_{G,a}(_-)$.
\end{convention}


\subsection{Dehn functions} There are two basic environments in which one can consider (higher) Dehn functions. We discuss both.
\subsubsection{The geometric setting}
Suppose $X$ is a weakly contractible complex, with boundary map $\partial$.  Any loop $\alpha$ in $X^{(1)}$ bounds a disk 
$\beta$ in $X^{(2)}$, $\partial \beta = \alpha$.  Denote the number of $n$-cells in a complex $W$ 
by $\|W\|_n$.  Set $Vol_2( \alpha ) = \min \| \beta \|_2$, where this minimum is taken over all disks $\beta$ in $X^{(2)}$ with 
$\partial \beta = \alpha$.  More generally, if $f:\alpha\to X$ is a mapping of a combinatorial $n$-sphere to $X$, there is a map 
of an $(n+1)$-ball $h:\beta\to X^{(n+1)}$ with $\partial h = f$, and the filling volume of $f$ is $Vol_{n+1}(f) = \min \| \beta \|_{n+1}$, 
where this minimum is taken over all combinatorial maps of $(n+1)$-balls $h:\beta\to X$ with boundary $f$. For each $n$, the 
\underline{$n^{\mathrm th}$ geometric Dehn function of $X$}, $d_X^n : \N \to \N$, is defined via the formula
\[
d_X^n(k) := \max Vol_{n+1}(f)
\]
where the maximum is taken over all combinatorial maps $f$ of $n$-spheres $f:\alpha\to X$ with $\| f \|_n \leq k$. In the case $n=1$, $d_X^1$ 
is often referred to as simply `the' geometric Dehn function of the complex.  These Dehn functions give a measurement of the filling 
volume of cycles in $X$, with $d_X^n$ being the $n^{th}$ \underline{unweighted} Dehn function of $X$.  Of course, these 
functions do not exist if the corresponding maximum values do not exist.  

When $X$ comes equipped with a weight $w$ on its cells, there is definition of geometric Dehn function which takes that 
weight into account.  Given a complex $W$ and combinatorial map $f:W\to X$, denote by $\| f \|_{w,n}$ the sum 
$\sum_{\sigma\in W^{(n)}} w(f(\sigma))$.  For a map $f:\alpha\to X$ of a combinatorial $n$-sphere $\alpha$ to $X$, denote 
the weighted filling volume of $f$ by $Vol_{w,n}(f) = \min \| h \|_{w,n+1}$, where this minimum is taken over all maps 
$h:\beta\to X$ of combinatorial $(n+1)$-balls to $X$ with boundary $f$. For each $n$, the \underline{$n^{th}$ weighted 
geometric Dehn function of $X$}, $d_X^{w,n} : \N \to \N$, is defined via the formula
\[
d_X^{w,n}(k) := \max Vol_{w,n+1}(f)
\]
where the maximum is taken over all combinatorial maps $f:\alpha\to X$ of $n$-spheres to $X$ with $\| f \|_{w,n} \leq k$.  
If the weight of each cell is set to one, then $d_X^n$ and $d_X^{w,n}$ are equal [Note: For certain choices of weights, 
the geometric and weighted geometric Dehn functions may be comparable. In general, however, if $X$  has geometric Dehn 
functions and weighted geometric 
Dehn functions defined in all dimensions, there need not be any particular relation between the two. If the weight function in each 
degree is a proper function on the set of simplicies, the weighted geometric Dehn functions exist in all dimensions. In general, however, 
the weighted geometric Dehn functions may fail to exist if the weight function fails to be proper in one or more dimensions].

Assume a non-weighted weakly contractible complex $X$ admits an action by a finitely generated group $G$  which is proper and 
cocompact on all finite skeleta.  Then $\{d_X^n\}_{n\ge 0}$ are referred to as the geometric Dehn functions of $G$ and denoted 
$\{d_G^n\}_{n\ge 0}$.  There is a natural way to weight $X$ so that the weighted geometric Dehn functions encode information 
about the group action.  Fix a basepoint 
$x_0 \in X^{(0)}$.  For a vertex $v \in X^{(0)}$, set $w_X(v) := d_{X^{(1)}}(x_0,v)$, the distance from $v$ 
to the basepoint in the $1$-skeleton of $X$, where each edge in the $1$-skeleton is assumed to have length $1$.  For an $n$-cell 
$\sigma \in X^{(n)}$ with vertices $(v_0, \ldots, v_k)$, 
let $w_X( \sigma ) := \sum_{i=1}^k w_X(v_i)$, the sum of the weights of the vertices.  Changing basepoints yields
different weight functions, but if for each $n$ there is a bound on the number of vertices an $n$-cell can possess,
then the two weight functions will be linearly equivalent.  We refer to this choice of assigning weights
to cells as the \underline{$1$-skeleton weighting}. For this choice of 
weight on  $X$ there is, for each $n\ge 0$, a constant $K_n$ with
\[
w_X( g \cdot \sigma ) \leq K_n\cdot \ell_G(g) +  w_X(\sigma)
\]
for all 
$n$-cells $\sigma$ (compare to (\ref{eqn:weight}) above), where $\ell_G$ denotes the word-length function of $G$ with respect to some 
fixed finite generating set.  In this case, the $G$-action can be used to compare $d_X^n$ with $d_X^{w,n}$.  Up to 
equivalence, one has \cite[Lemmas 2.3, 2.4]{JR1}
\begin{equation}\label{eqn:compareDehn}
 d_X^{w,n}(k) \leq (d_X^n(k))^2 ,\;\;\;\;\; d_X^n(k) \leq k d_X^{w,n}(k^2) 
\end{equation}

Given a finite subcomplex $b$ of $X$ equipped with a $1$-skeleton weight function $w_X$,
define the weight of $b$ to be $|b|_w := \sum_{\sigma\in b} w_X(\sigma)$. 

\subsubsection{The algebraic setting}
In the literature, it is the geometric Dehn functions that have received the most attention.  Our focus, however, will be on
the homological version of the above constructions.
Suppose $(C_*, d_*)$ is an acyclic chain complex of free $\Z$-modules.  Denote by $\{ v_i \, | \, i \in \mathcal{I} \}$, a basis 
of $C_n$ over $\Z$. Given an element $\alpha = \sum_{i \in \mathcal{I}} \lambda_i v_i$ of $C_n$, set 
$\| \alpha \|_n := \sum_{i \in \mathcal{I}} |\lambda_i|$.
For $\alpha\in C_n$ a cycle, let $Vol_{n+1}( \alpha ) := \min \| \beta \|_{n+1}$, where this minimum is taken
over all $\beta \in C_{n+1}$ with $d_{n+1}( \beta ) = \alpha$.
Define a function $d_C^n : \N \to \N$ by
\[  
	d_C^n ( k ) := \max \left\{ Vol_{n+1}( \alpha ) \, | \, d_n( \alpha ) = 0,  \| \alpha \|_n \leq k \right\}
\]
The function $d_C^n$ is the \underline{$n^{\mathrm th}$ unweighted homological Dehn function of $C_*$}.  These Dehn functions
measure the filling complexity of the chain complex $C_*$.  As before, these Dehn functions do not exist if the corresponding 
maximum values do not exist.  

When the $C_n$ come equipped with a weight function $w$ on its basis, homological Dehn functions can be defined so as to take that weight into account.
For $\alpha = \sum_{i \in \mathcal{I}} \lambda_i v_i$, let $\| \alpha \|_{w,n} := \sum_{i \in \mathcal{I}} |\lambda_i| w( v_i )$.
If $\alpha$ is a cycle, the weighted filling volume of $\alpha$ is $Vol_{w,n+1}( \alpha ) := \min \| \beta \|_{w,n+1}$, where this minimum is taken
over all $\beta \in C_{n+1}$ with $d_{n+1}( \beta ) = \alpha$.
Define the \underline{$n^{\mathrm th}$ weighted Dehn function of $C_*$}, $d_C^{w,n} : \N \to \N$, by
\[  
	d_C^{w,n} ( k ) := \max \left\{ Vol_{w,n+1}( \alpha ) \, | \, d_n( \alpha ) = 0,  \| \alpha \|_{w,n} \leq k \right\}
\]
If the weight of each basis element is set to one, then $d_C^n = d_C^{w,n}$.  For certain choices of weights,
the unweighted and the weighted Dehn functions may be comparable.  In general, however, no relationship needs
exist between the two.

Now suppose $G$ is an $FP^{\infty}$ group equipped with word-length function $L$, and $C_*$ is a  resolution of $\Z$ over $\Z [G]$ 
which in each degree is a finitely generated free module over $\Z[G]$.  In this case, the collection $\{d_C^n\}$ are referred to as 
the \underline{Dehn functions of G}, denoted $d_G^n$.  We will call the resolution $C_*$ \underline{$k$-nice} ($k \le \infty$) if 
\begin{itemize}
\item for each finite $n\le k$, $C_n = \Z [G] [T_n]$ for some finite weighted set $(T_n,w^T_n)$;
\item for each finite $n\ge 0$, $S_n$ ($ = $ the orbit of $T_n$ under the free action of $G$) is equipped  with a proper weight function $w^S_n$ satisfying
\[
C_{1,n}L(g) + w_n^T(t)\le w_n^S(gt)\le C_{2,n}L(g) + w_n^T(t)\qquad \forall g\in G,t\in T_n
\]
for positive constants $C_{1,n}\le C_{2,n}$ depending only on $n$;
\item for each finite $n$, $d^C_n: C_n\to C_{n-1}$ is linearly bounded with respect to the weight functions on $C_n$ and $C_{n-1}$.
\end{itemize}

The term ``nice'' will refer to the case $k=\infty$. The \underline{weighted Dehn functions of G} (through dimension $k$ if $k$ 
is finite) are given by $\{d_G^{w,n} := d_C^{w,n}\}_{n < k}$ where $C_*$ is a $k$-nice resolution of $\Z$ over $\Z[G]$.  Both 
$d_G^n$ and $d_G^{w,n}$ are independent of the particular choice of $k$-nice resolution used in their definition, up to linear 
equivalence. For such resolutions, the fact $T_n$ is finite for each finite $n\le k$ means that the linear equivalence classes 
of the Dehn functions $\{d_G^{w,n}\}$ are independent of the choice of weightings on $\{T_n\}$. Finally, a resolution $D_*$ of 
$\C$ over $\C[G]$ is \underline{$k$-nice} resp.\  \underline{nice} if it is of the form $D_* = C_*\otimes\C$ where $C_*$ is a 
$k$-nice resp.\ nice resolution of $\Z$ over $\Z[G]$.

\begin{lemma}\label{lem:niceExt}
Let $G$ be an $FP^{k}$ group equipped with word-length function $L$, $C_*$ a $k$-nice resolution of  $\Z$ over $\Z[G]$, 
$D_* = C_*\otimes\C$, and $\B$ and $\B'$ bounding classes.  Suppose that the weight $w$ on the weighted set underlying $C_*$ 
takes no value in $(0,1)$.  Denote by $\mathcal{B}D_*$ the corresponding Frechet completion of 
$D_*$ with respect to the bounding class $\B$, as defined above. Further suppose that the weighted Dehn functions $\{ d_C^{w,n} \}$ 
are ${\B}'$-bounded in dimensions $n <k$, that $\B$ is a right ${\B}'$-class, and that $\B\succeq {\mathcal L}$ .  Then there 
exists a bounded chain null-homotopy $\{s_{n+1} : \mathcal{B}D_n \to \mathcal{B}D_{n+1}\}_{k > n\ge 0}$, implying ${\B}D_*$ is 
a continuous resolution of $\C$ over $\BG$ through dimension $k$.
\end{lemma}

\begin{proof} 
Note first that the ``niceness'' of $C_*$ guarantees that  boundary map $d^C_n : C_n \to C_{n-1}$ extends to a 
continuous boundary map $d_n : \mathcal{B}D_n \to \mathcal{B}D_{n-1}$.  We will prove the lemma in three steps.
\begin{itemize}
\item[\underline{Claim 1}] 
For each $k > n\ge 1$, there exists a function $f_n\in {\B}'$ so that for all $\alpha\in \ker(d^C_n)$ and $h\in \B$, there exists 
$\beta_{\alpha}\in C_{n+1}$ with $d^C_n\left(\beta_{\alpha}\right) = \alpha$, and 
$\left|\beta_{\alpha}\right|_h\le (h\circ f_n)\left(\|\alpha\|_{w,n}\right)$.

\noindent {\it Proof}. 
The hypothesis on $\{d_C^{w,n}\}$ implies that for each $n<k$ there exists $f_n\in{\B}'$ with 
\[
Vol_{w,n+1}(\alpha)\le f_n(\|\alpha\|_{w,n})\qquad\forall\alpha\in \ker(d^C_n)
\]
Then for $\alpha = \sum m_{ij} g_i t_j\in \ker(d^C_n)$, we may choose $\beta_{\alpha} = \sum n_{kl}g_k s_l\in C_{n+1}$ with 
$d^C_{n+1}\left(\beta_{\alpha}\right) = \alpha$ and
\[
\left\|\beta_{\alpha}\right\|_{w,n+1} = \sum |n_{kl}|  w_{n+1}(g_k s_l)\le f_n\left(\sum |m_{ij}| w_n(g_i t_j)\right) = f_n\left(\|\alpha\|_{w,n}\right)
\]
Since $\B\succeq {\mathcal L}$, we may assume that $h\in\B$ is super-additive on the interval $[1,\infty)$. One then has
\begin{align*}
&\left|\beta_{\alpha}\right|_h = \left|\sum n_{kl} g_k s_l\right|_h\\
	&:= \sum |n_{kl}| h(w_{n+1}(g_k s_l))\\
	&\le h\left(\sum |n_{kl}|w(g_k s_l)\right)\qquad\text{by the super-additivity of $h$}\\
	&= h\left(\left\|\beta_{\alpha}\right\|_{w,n+1}\right)\\
	&\le (h\circ f_n)\left(\|\alpha\|_{w,n}\right)
\end{align*}
\hfill //

\item[\underline{Claim 2}] 
For each $k >n\ge -1$ there exists a $\B$-bounded linear section $s^C_{n+1}:C_n\to C_{n+1}$ satisfying 
$d^C_{n+1}s^C_{n+1} + s^C_n d^C_n = Id$.

\noindent{\it Proof}. 
The case $n = -1$ is trivial since any basis element of $C_0$ determines a linear injection $\Z = C_{-1}\to C_0$ which 
is bounded. Assume $s^C_n$ has been defined. Let $p_n = (Id - s^C_n d^C_n):C_n\to \ker(d^C_n)$; this projection onto 
$\ker(d^C_n)$ is bounded via the boundedness of $s_n$. Thus we may find an $f'_n\in {\B}'$ with 
$\|p_n(x)\|_{w,n}\le f'_n\left(\|x\|_{w,n}\right)$ for all $x\in C_n$. Let $f''_n = f_n\circ f'_n$. For each basis element 
$g_i t_j\in C_n$, set $s^C_{n+1}(g_i t_j) = \beta_{p_n(g_i t_j)}$, as defined in the above Claim. Then for each super-additive $h\in\B$,
\begin{align*}
&\left|s^C_{n+1}(g_i t_j)\right|_h = \left|\beta_{p_n(g_i t_j)}\right|_h\\
	&\le (h\circ f_n)\left(\|p_n(g_i t_j)\|_{w,n}\right)\\
	&\le (h\circ f_n\circ f'_n)\left\|g_i t_j \right\|_{w,n} = (h\circ f''_n)(w_n(g_i t_j))
\end{align*}
Extending $s^C_{n+1}$ linearly to all of $C_n$ yields the desired result.\hfill //

\item[\underline{Claim 3}] 
For each $k > n\ge -1$, the linear extension of $s^C_{n+1}$ to $D_n$ yields a $\B$-bounded linear map $s^D_{n+1}:D_n\to D_{n+1}$.

{\it Proof}. 
This follows from the sequence of inequalities
\begin{align*}
&\left|s^D_{n+1}\left(\sum \lambda_{ij}g_i t_j\right)\right|_h\\
	&= \sum |\lambda_{ij}|\left|s^C_{n+1}(g_i t_j)\right|_h\\
	&\le \sum |\lambda_{ij}|\left|g_i t_j\right|_{h\circ f''_n}\\
	&= \left|\sum \lambda_{ij} g_i t_j\right|_{h\circ f''_n}
\end{align*}
\hfill //
\end{itemize}
This completes the proof of the lemma. We should note that the definition of weighted Dehn function could be considered for 
more general resolutions of $\C$ over $\C[G]$ which do not arise from tensoring a nice resolution of $\Z$ over $\Z[G]$ with 
$\C$. However, for this more general class, it is likely that the statement of this lemma no longer holds true.
\end{proof}

Although it appears that not allowing your proper weight function $w$ to take values in $(0,1)$ is restrictive,
it is always possible to find a linearly equivalent proper weight function $w'$ which takes values
in $\Z_+$.  While the weight structure may change slightly, the completions arising from using the two 
weights will agree.  In many cases, for example the $1$-skeleton weighting discussed above and in \cite{O2},
the naturally occurring weight function is integral-valued.

Also, in the proof of Claim 1 above, we assumed the existence of a function $h \in \B$ which was superadditive on $[0,1)$.  To this end
suppose $f$ and $g$ are differentiable functions, and consider the following property: 

There exists a $C \geq 0$ such that for all $x \geq C$
\[
f(g(x))\ge g(f(x)).
\]

A sufficient set of conditions to guarantee this is:
\begin{enumerate}
	\item[(1)] $f(g(C))\ge g(f(C))$.
	\item[(2)] $[fg(x))]'\ge [g(f(x))]'$.
\end{enumerate}
Restrict to the case where $g$ is the linear function $g(x) = rx$, for $r\geq 1$
a real number. Then
\begin{eqnarray*}
\left(f(g(x))\right)' & = & \left(f(rx)\right)' = rf'(rx) \\
\left(g(f(x))\right)' & = & rf'(x)
\end{eqnarray*}

In this case condition (2) is just the requirement that $f'(rx)\ge f'(x),
r\geq 1$, i.e., that $f'$ is non-decreasing.  If $f(0)$ is
required to be $>0$, then taking, say, $C = 1$ this condition becomes
$f(r)\geq rf(1)$.  If $\B\succeq {\mathcal L}$,
this shows that $\B$ contains functions which are superadditive
when restricted to $[1,\infty)$.

\begin{corollary}
Let $G$ be a finitely generated group acting properly and cocompactly on a contractible polyhedral complex $X$, with finitely many
orbits in each dimension, and endowed with the $1$-skeleton weighting.  If the integral polyhedral chain complex, $C_n(X;\Z)$,
admits $\B'$ bounded weighted Dehn functions and $\B$ is a right $\B'$-class, then the Frechet completion $\B C_n(X;\C)$
gives a continuous resolution of $\C$ over $\BG$.  In particular if $C_n(X;\Z)$ admits polynomially bounded weighted Dehn functions,
so does $C_n(X;\C)$.
\end{corollary}

\underline{\bf Remarks}
\begin{itemize}
\item It is a result due to Gersten that for finitely-presented groups, the first algebraic and first geometric Dehn functions 
are equivalent. However, in dimensions greater than one, it is not at all clear if such a relation persists even when both types 
are defined. The one case in which one can prove an equivalence is when there is a $G-HF^{\infty}$ model for $EG$ admitting an 
appropriate ``coning'' operation in all dimensions with explicitly computable bounds on the number and weights of the simplices 
used in coning off a simplex of one lower dimension (such is the case when $G$ is asynchronously combable - see below).
\item For finitely generated groups, word-hyperbolicity is equivalent to having $d_G^1$ bounded by a linear function. Hyperbolic 
groups provide interesting phenomena in the context of Dehn functions.  A prime example is the isoperimetric gap.  If $d_G^1$ is 
bounded by a function of the form $n^r$ with $r < 2$, then $d_G^1$ is bounded by a linear function \cite{Gromov, Olsh}.  In particular 
if $G$ is not hyperbolic, then $d_G^1$ must be at least quadratic. On the other hand, it is well known that for a hyperbolic group $G$, 
the functions $d_G^n$ are linearly bounded in every dimension $n$. The geometric characterization of the isocohomological property 
discussed in Section \ref{sect:BIsocoh} below  implies that $d_G^{w,n}$ are all linearly bounded Dehn functions, providing a bounded 
version of the $FP^{\alpha}$ condition described above.
\end{itemize}

This idea of combining boundedness with the $FP^{k}$ condition is made precise by
\begin{definition} 
Given a bounding class $\B$ and a group with word-length $(G,L)$, we say $G$ is of type $\B FP^k$ ($k < \infty$) if it is of type $FP^k$, and there exists a $k$-nice resolution $D_*$ of $\C$ over $\C[G]$ for which the completion ${\B}D_*$  admits a bounded linear chain contraction through dimension $k$.  We say $G$ is $\B FP^\infty$ if it is $\B FP^k$ for all $k$.
\end{definition}



\subsection{Products, coproducts and pairing operations}\label{sect:pairings}

\begin{definition}
Let $(X,w)$ be a weighted set.  A bounding class $\B$ is \underline{nuclear for $(X,w)$} if for every $\lambda \in \B$
there is $\eta \in \B$ such that the following series converges.
\[\sum_{x\in X} \frac{\lambda\left( w(x) \right)}{\eta\left( w(x) \right)} \]
\end{definition}

For example, if $(X,w)$ has polynomial growth then $\mathcal{P}$ is a nuclear bounding class, but if $(X,w)$ has exponential 
growth $\mathcal{P}$ is not nuclear.  The exponential bounding class is nuclear for every finitely generated group with word-length.  
The following lemma motivates this definition.

\begin{lemma}
Let $(G,L)$ be a group with a proper length function, and let $\B$ be a nuclear bounding class for $(G,L)$.  Then
$\BG$ is a nuclear Frechet algebra.  
\end{lemma}

For a bornological space $V$, let $V'$ denote the dual space.  Our interest in nuclearity arises from  
its use in identifying $\left( V \btensor W \right)'$.
\begin{lemma}
Let $V$ and $W$ be Frechet spaces, and let $W$ be nuclear.  Then $\left( V \btensor W \right)' = V' \btensor W'$. 
\end{lemma}
\begin{proof}
By definition, $\left( V \btensor W \right)' = \bHom( V \btensor W, \C)$.  Using the adjointness of the projective tensor
product, this is isomorphic to $\bHom(V, \bHom( W, \C ) ) = \bHom( V, W')$.
As $W$ is a nuclear Frechet algebra, its dual is also nuclear.  Corollary 1.161 of \cite{M4} gives that $\bHom(V, W') \isom V' \btensor W'$.
\end{proof}

This Lemma suggests that any sort of K\"unneth Theorem in $\B$-bounded cohomology would hold only under very restrictive conditions. Nevertheless, the pairing operations used to prove it exist in the $\B$-bounded setting under minimal conditions. We consider them next, as they will be needed later on.
\vskip.2in
Let $(X,w_X)$ and $(Y,w_Y)$ be a weighted 
simplicial sets. Then their product $(X\times Y,w_X\times w_Y)$ is again a weighted simplicial set, where $X\times Y$ is equipped with diagonal simplicial structure. It follows from the definition of a weighted complex that the Alexander-Whitney map
\[
\Delta_{AS}:C_*(X\times Y)\to C_*(X)\tensor C_*(Y)
\] 
is uniformly bounded above in each dimension $n$ by a linear function of $n$. As a result, when $\B$ is multiplicative this induces an exterior \underline{algebraic} tensor product on the cochain level
\begin{equation*}
{\B}\Delta^*:\to {\B}C^*(X)\tensor {\B}C^*(Y)\to {\B}C^*(X\times Y)
\end{equation*}


From the definitions of $H^*(_-)$ and $H_*(_-)$ there is an obvious Kronecker-Delta pairing
\begin{equation*}
{\B}H^*(X)\tensor {\B}H_*(X)\to\C \,\quad (c,d)\mapsto <c,d>,\quad x = a,b
\end{equation*}
More generally, an analysis of the standard cap product operation on the chain and cochain level yields a cap product operation
\begin{equation*}
{\B}H^*(X)\tensor {\B}H_*(X)\to {\B}H_*(X),\quad (c,d)\mapsto c\cap d,\quad x = a,b
\end{equation*}

These operations satisfy the appropriate commuting diagrams with respect to the comparison map $\Phi_{\B}^*$ and 
$\Phi^{\B}_*$, leading to the identities
\begin{gather}
<\Phi_{\B}^*(c),d> = <c,\Phi^{\B}_*(d)>,\quad c\in {\B}H^*(X),\ \ d\in H_*(X)\\
\Phi_{\B}^*(c)\cap d = c\cap \Phi^{\B}_*(d),\quad  c\in {\B}H^*(X),\ \ d\in H_*(X)
\end{gather}


\section{$\B$ cohomology of discrete groups}

\subsection{Combable groups}
Call a function $\sigma : \mathbb{N} \rightarrow \mathbb{N}$ a \underline{reparameterization} if
\begin{itemize}
    \item $\sigma(0) = 0$,
    \item $\sigma(n+1)$ equals either $\sigma(n)$ or $\sigma(n)+1$,
    \item $\underset{n\to\infty}{\lim}\sigma(n) = \infty$
\end{itemize}

\begin{definition} 
Let $(X,*)$ be a discrete metric space with basepoint.
By a \underline{combing} of $X$ we will mean a collection of
functions $\{f_n:X\to X\}_{n\geq 0}$ satisfying
\begin{itemize}
\item[(C1)] $f_0(x) = x\ \ \forall x\in X$
\item[(C2)] There exists a super-additive function $\psi$ such that $\forall x,y\in X$,\,\,
there are reparameterizations $\sigma$ and $\sigma'$ with $d(f_{\sigma(n)}(x),f_{\sigma'(n)}(y)) \le
\psi(d(x,y))$ for all $n\ge 0$.
\item[(C3)] $\exists \lambda$ such that $\forall x\in X, n\in \mathbb N$, $d(f_n(x),f_{n+1}(x))\leq \lambda$
\item[(C4)] $\exists \phi$ such that $f_n(x) = *\,\,\forall n\geq \phi(d(x,*))$
\end{itemize}
\end{definition}
Remarks:
\begin{itemize}
\item As noted in the introduction, the combings above are oriented in the opposite direction than what has been customarily the case.
\item Axiom (C2) allows for what are typically referred to as
asynchronous combings, with synchronous combings corresponding to the case
that the reparameterizations are the identity maps. Note also that
\item the reparameterizations $\sigma,\sigma'$ in {\it (C2)} depend on $x$ and $y$.
\end{itemize}

\begin{definition} 
Given a discrete group $G$ equipped with a (proper) length function $L$,
a \underline{combing of $G$} (with respect to $L$), or $(G,L)$, is a combing of the discrete
metric space $(G,d_L)$, where $d_L(g_1,g_2) := L(g_1^{-1}g_2)$.
\end{definition}

We first show
that reparameterizations can be chosen so as to be compatible on
specific $(n+1)$-tuples.

\begin{lemma}\label{lem:asynchCompatible} 
Suppose $(G,L)$ admits an asynchronous
combing in the above sense. Then for all $(m+1)$-tuples
$(g_0,\dots,g_m)\in G^{(m+1)}$, there exist reparameterizations
$\sigma_{0},\dots, \sigma_{m}$ such that
\begin{equation*}
\forall n\geq 0,\ d(f_{\sigma_{i}(n)}(g_i),f_{\sigma_{i+1}(n)}(g_{i+1})) \leq \psi(d(g_i,g_{i+1}))
\end{equation*}
\end{lemma}

\begin{proof} By definition it is true for $m=1$. Assume then it is true for
fixed $m\geq 1$. Given an $(m+2)$-tuple $(g_0,\dots,g_{m+1})$, we may assume by induction that

\begin{itemize}
\item There exist reparameterizations $\sigma_{0},\dots, \sigma_{m}$ with
$d(f_{\sigma_{i}(n)}(g_i),f_{\sigma_{i+1}(n)}(g_{i+1})) \le
\psi(d(g_i,g_{i+1})$ for all $n\geq 1$, and
\item There exist reparameterizations $\sigma'_m,\sigma'_{m+1}$ with
$d(f_{\sigma'_m(n)}(g_m),f_{\sigma'_{m+1}(n)}(g_{m+1})) \leq \psi(d(g_m,g_{m+1}))$
for all $n\geq 1$
\end{itemize}

We need to show that the reparameterization functions can be
further reparameterized so as to synchronize $\sigma_m$ and
$\sigma'_m$. For this we proceed by induction on $k\in \mathbb N =
$ the domain of the reparameterization functions. 

\begin{itemize}
    \item [$\underline{k = 0}$] By definition, $\sigma_m(0) = \sigma'_m(0) = 0$.
    \item [$\underline{k > 0}$] Suppose $\sigma_m(i) = \sigma'_m(i)$ for $0\leq i\leq k$.
    \begin{itemize}
        \item [\underline{Case 1}] $\sigma_m(k+1) = \sigma'_m(k+1)$.
        In this case there is nothing to do.
        \item [\underline{Case 2}] $\sigma_m(k+1) = \sigma_m(k), \sigma'_m(k+1) =
        \sigma'_m(k) + 1$. In this case we leave $\sigma_m$ alone, and redefine
        $\sigma'_m, \sigma'_{m+1}$:
            \begin{equation*}
            \text{for }l = m,m+1, (\sigma'_l)_{new}(i) =
                \begin{cases}
                (\sigma'_l)_{old}(i)\, &0\leq i\leq k\\
                (\sigma'_l)_{old}(i)-1\, &k+1\leq i
                \end{cases}
            \end{equation*}
        \item [\underline{Case 3}] $\sigma_m(k+1) = \sigma_m(k) + 1, \sigma'_m(k+1) =
        \sigma'_m(k)$. In this case we leave $\sigma'_m$ alone, and redefine
        $\sigma_0,\dots, \sigma_m$:
            \begin{equation*}
            \text{for }0\leq l\leq m, (\sigma_l)_{new}(i) =
                \begin{cases}
                (\sigma_l)_{old}(i)\, &0\leq i\leq k\\
                (\sigma_l)_{old}(i)-1\, &k+1\leq i
                \end{cases}
            \end{equation*}
    \end{itemize}
\end{itemize}
Thus by induction on $k$, we may choose reparameterization functions
$\sigma_0,\dots,\sigma_m,\sigma'_m, \sigma'_{m+1}$ with
\begin{itemize}
    \item [(S1)] $\sigma_0,\dots,\sigma_m$ satisfying the conditions of the Lemma for
    the $(m+1)$-tuple $(g_0,\dots,g_m)$,
    \item [(S2)] $\sigma'_m,\sigma'_{m+1}$ satisfying the conditions of the Lemma for
    the pair $(g_m,g_{m+1})$,
    \item [(S3)] $\sigma_m = \sigma'_m$
\end{itemize}
Setting $\sigma_{m+1} := \sigma'_{m+1}$ then concludes the proof of the initial induction
step, and hence of the Lemma.
\end{proof}

\begin{definition} 
We say that a metric space $(X,d)$ is \underline{quasi-geodesic} if
there exist positive constants $\epsilon$, $S$, and $C$ such that
for any two points $x$,$y \in X$, there is a finite sequence of
points $x_0 = x, x_1, x_2, \ldots, x_k = y$ satisfying:
\begin{enumerate}
	\item $\epsilon \leq d( x_i, x_{i+1}) \leq S$ for $i = 0, 1, \ldots, k-1$.
	\item $d(x_0,x_1) + d(x_1,x_2) + \ldots + d( x_{k-1}, x_k ) \leq C d(x,y)$.
\end{enumerate}
\end{definition}
As we are concerned primarily with connected complexes, all metric spaces we consider
will be assumed to be quasi-geodesic.

\begin{lemma}
Suppose $\left\{f_n : X \to X \right\}_{n\geq 0}$ is an asynchronous combing of a quasi-geodesic
metric space $(X,d)$.  There exists a positive constant $K$ such that for all
$x,y \in X$, there are reparameterizations $\sigma$ and $\sigma'$ such that
for all $n \geq 0$, $d( f_{\sigma(n)}(x), f_{\sigma'(n)}(y) ) \leq K d(x,y)$.
\end{lemma}
\begin{proof}
Let $x_0 = x, x_1, x_2, \ldots, x_k = y$ be given by the quasi-geodesic property.
By lemma \ref{lem:asynchCompatible} there are reparameterizations, $\sigma_i$, such that
$d( f_{\sigma_i(n)}(x_i), f_{\sigma_{i+1}(n)}(x_{i+1}) ) \leq \psi( S )$, for all n.
As $k \leq \frac{C}{\epsilon} d(x,y)$, the triangle inequality yields
$d( f_{\sigma_0(n)}(x), f_{\sigma_{k}(n)}(y) ) \leq \frac{\psi( S )C}{\epsilon} d(x,y)$, for all n.
\end{proof}

The next theorem was originally shown for synchronously combable groups in \cite{Al},  and
asynchronously combable groups through dimension 3 in \cite{Ger}. Our method of proof
actually proves more, as we will see in the following section.

\begin{theorem} 
If $(G,L)$ admits an asynchronous combing in the above sense, then it is type $HF^{\infty}$.
\end{theorem}

\begin{proof} 
Let $EG\hskip-.02in.$ denote the simplicial
homogeneous bar resolution of $G$. Let $G$ act in the usual way on
the left, by $g\cdot (g_0,g_1,\dots,g_n) :=
(gg_0,gg_1,\dots,gg_n)$. Define a $G$-invariant simplicial weight
function on $EG\hskip-.02in.$ by
\begin{equation*}
w_n(g_0,g_1,\dots,g_n) := \sum_{i=0}^{n-1} d(g_i,g_{i+1}) = \sum_{i=0}^{n-1} L(g_i^{-1}g_{i+1})
\end{equation*}
Because $L$ is proper, the orbit $\{(g_0,g_1,\dots,g_n)\, |\, w_n(g_0,g_1,\dots,g_n)\leq N\}/G$ is a finite set 
for each $n$ and $N$. This orbit may alternatively be described as $\pi^{-1}(B_N(BG_n))$, where $BG\hskip-.02in.$ 
is the non-homogeneous bar construction on $G$, $\pi:EG\hskip-.02in.\to BG\hskip-.02in.$ is given by 
$\pi(g_0,g_1,\dots,g_n) = [g_0^{-1}g_1,g_1^{-1}g_2,\dots,g_{n-1}^{-1}g_n]$, and $B_N(_-)$ denotes the $N$-ball 
$B_N(BG_n) := \{[g_1,\dots,g_n]\,|\, \sum_{i=1}^n L(g_i)\leq N\}$.

Recall that given two simplicial functions $h_0,h_1:EG \to EG$, 
there is a homotopy between them represented by the ``sum''
\begin{equation*}
H(h_0,h_1)(g_0,g_1,\dots,g_n) = \sum_{i=0}^n(-1)^i
(h_0(g_0),h_0(g_1),\dots,h_0(g_i),h_1(g_i),
h_1(g_{i+1}),\dots,h_1(g_n))
\end{equation*}

Although we have written this sum algebraically, this should be
viewed as a \underline{geometric} sum which associates to the
$n$-simplex $(g_0,g_1,\dots,g_n)$ the collection of
$(n+1)$-simplices indicated by the right-hand side, with
orientation determined by the coefficient $(-1)^i$. Geometrically,
this collection of $(n+1)$ simplices, all of dimension $(n+1)$,
fit together to form a subset whose geometric realization is
homeomorphic to $\Delta_n\times [0,1]$. In fact, this last
statement is true for more general types of maps which are not
simplicial. In particular, given i) a fixed asynchronous combing
$\{f_n\}$ of $G$, ii) a fixed $n$-simplex $(g_0,\dots,g_n)$ of
$EG\hskip-.02in.$, and iii) a collection of reparameterizations
$\sigma_0,\dots,\sigma_n$ satisfying the condition of Lemma 1 with
respect to i) and ii), we may consider the `homotopy" from
$(f_{\sigma_0(m)}(g_0),f_{\sigma_1(m)}(g_1),\dots,f_{\sigma_n(m)}(g_n))$
to
$(f_{\sigma_0(m+1)}(g_0),f_{\sigma_1(m+1)}(g_1),\dots,f_{\sigma_n(m+1)}(g_n))$
given by the expression
\begin{multline}\label{homotopy}
H^{\underline{\sigma}}(\{f_k\};m,m+1)(g_0,g_1,\dots,g_n)\\
:= \sum_{i=0}^n(-1)^i
(f_{\sigma_0(m)}(g_0),f_{\sigma_1(m)}(g_1),\dots,f_{\sigma_i(m)}(g_i),f_{\sigma_i(m+1)}(g_i),
f_{\sigma_i(m+1)}(g_{i+1}),\dots,f_{\sigma_i(m+1)}(g_n))
\end{multline}

This is not part of a global homotopy, but still yields a collection of oriented
$(n+1)$-simplices whose realization is homeomorphic to $\Delta_n\times[0,1]$.
Moreover, these homotopies may be strung together, as the ``end" of
$H^{\underline{\sigma}}(\{f_k\};m,m+1)(g_0,g_1,\dots,g_n)$ and the ``beginning" of
$H^{\underline{\sigma}}(\{f_k\};m+1,m+2)(g_0,g_1,\dots,g_n)$ match up.

Given a function $f:\R_+\to \R^+$, write
$w_n^f$ for the weight function
\begin{equation*}
w_n^f(g_0,g_1,\dots,g_n) := \sum_{i=0}^{n-1}f(d(g_i,g_{i+1}))
\end{equation*}
By Lemma 1 and property (C3),
\begin{align}\label{weight}
&w_{n+1}\left(H^{\underline{\sigma}}(\{f_k\};m,m+1)(g_0,g_1,\dots,g_n)\right)\\
&< (n+1)\bigg(w_n(f_{\sigma_0(m)}(g_0),f_{\sigma_1(m)}(g_1),\dots,f_{\sigma_n(m)}(g_n))\nonumber\\
&\qquad\qquad\qquad+ w_n(f_{\sigma_0(m+1)}(g_0),f_{\sigma_1(m+1)}(g_1),\dots,f_{\sigma_n(m+1)}(g_n)) + \lambda\bigg)\nonumber\\
&\le(2n+2)\left(w_n^{\psi}(g_0,g_1,\dots,g_n) + \lambda\right)\nonumber\\
&\le(2n+2)\psi\left(w_n(g_0,g_1,\dots,g_n)\right) + (2n+2)\lambda\nonumber
\end{align}

Equation (\ref{weight}) implies that every simplex in $\pi^{-1}(B_N(BG_n))$
can be coned off in $\pi^{-1}(B_{N'}(BG_n))$ where $N' =
(2n+2)(\psi(N) + \lambda))$. Of course, degeneracies preserve the
inequality. In other words, if $(g_0,g_1,\dots,g_n) =
s_I(g'_0,\dots,g'_k)$ for some iterated degeneracy map $s_I$ and
$k$-simplex $(g'_0,\dots,g'_k)$, then the inequality in (\ref{weight}) may be
improved to

\begin{equation*}
w_{n+1}\left(H^{\underline{\sigma}}(\{f_k\};m,m+1)(s_I(g'_0,\dots,g'_k)\right) <
(2k+2)\psi(w_k(g'_0,\dots,g'_k)) + (2k+2)\lambda
\end{equation*}

Let $X(n) := EG\hskip-.02in.^{(n)}$, the simplicial $n$-skeleton
of $EG\hskip-.02in.$ For each integer $N$, let $X(n)_N := X(n)\cap
\pi^{-1}(B_N(BG\hskip-.02in.))$. Then $X(n)$ is an $n$-good
complex for $G$ in the sense of \cite{Br1}, and obviously $X(n) =
\lim_N X(n)_N$. Moreover, equations (\ref{homotopy}) and (\ref{weight}) together imply
\begin{equation}\label{upper}
X(n)_N\hookrightarrow X(n)_{N'}\,\,\text{ is
null-homotopic},\quad N' = (2n+2)(\psi(N) + \lambda)
\end{equation}
By Theorem 2.2 of \cite{Br1}, we conclude that $G$ is of type $FP^n$. Then, as $G$ is
of type $FP^n$ for each $n$, it must be of type $FP^{\infty}$
\cite{Br2}.
\end{proof}

In fact, the explicit estimates in (\ref{weight}) and (\ref{upper}) allow one to conclude a bit more. We will need some terminology.

\begin{definition}
A discrete group with word-length $(G,L)$ is
\underline{$\B$-combable} (i.e., $\B$-asynchronously combable)
if the functions $\psi$ and $\phi$ in (C2)and (C4) are bounded
above by functions in the bounding class $\B$.
\end{definition}

As indicated above, given a bornological ${\mathcal{H}}_{{\B},L}(G)$-module $V$,
one has the subcochain complex ${\B} C^*(G;V)\subset C^*(G;V) = \Hom_G(C_*(EG\hskip-.02in.),V)$
consisting of those cochains which are bounded in the bornology induced by $\B$. The group $G$,
or pair $(G,L)$ is called \underline{$\B$-isocohomological} with respect to $V$ (abbr. $V$-$\B$IC)) 
if the inclusion ${\B} C^*(G;V)\subset C^*(G;V)$ induces an isomorphism of cohomology groups in all degrees.
\begin{equation*}
{\B} H^*(G;V) := H^*({\B} C^*(G;V)) \stackrel{\isom}{\to} H^*(G;V)
\end{equation*}

$\B$-isocohomologicality with respect the trivial module $\C$ is referred to simply as
$\B$-isocohomological ($\B$-IC). The pair
$(G,L)$ is \underline{strongly $\B$-isocohomological} (abbr. $\B$-SIC) if it is
$V$-$\B$IC for all bornological ${\mathcal{H}}_{{\B},L}(G)$-modules
$V$.

\begin{corollary}
Let $G$ be a finitely-presented group equipped with word-length function $L$, and $\B$ a multiplicative bounding 
class. If $(G,L)$ is asynchronously $\B$-combable, then $G$ is $\B$-IC.
\end{corollary}

\begin{proof} 
By the previous theorem, the hypothesis that $G$ is $\B$-asynchronously combable implies by equations (\ref{weight}) 
and (\ref{upper}) that in using the combing to cone off a simplex of weight $m$, both the number of simplices appearing 
in the cone, as well as the weight of each, is bounded above by $f_i(m)$ where $f_i\in\B$. By the multiplicativity of 
the bounding class $\B$, the bornological chain complex  ${\B}C_*(EG)$ is a tempered complex in the sense of Meyer \cite{M1} 
which satisfies the necessary conditions established by Meyer to conclude the result (Meyer's original result was stated only 
for the polynomial, subexponential and simple exponential bounding classes, but the same argument works for arbitrary multiplicative bounding classes).
\end{proof}


\subsection{$\B$-isocohomologicality and type ${\B}FP^{\infty}$ groups}\label{sect:BIsocoh}

It is natural to ask about the relation between the purely homological notion of strong $\B$-isocohomologicality and the more geometric/topological  ${\B}FP^{\infty}$ condition. The following result answers that question; it is a generalization to arbitrary bounding classes of \cite[Thm. 2.6]{JR1}.

\begin{theorem}\label{thm:geocritRest}
Let $G$ be a finitely presented group of type $FP^r$, for some $r \leq \infty$.  For $k < r$, the following are equivalent.  
\begin{itemize}
	\item[(${\B} 1$)] ${\B}H^{*}(G;V)\to H^{*}(G;V)$ is an isomorphism for all bornological 
		${\mathcal{H}}_{{\B},L}(G)$-modules $V$, in all degrees $* \leq k+1$.
	\item[(${\B} 2$)] ${\B}H^*(G;V)\to H^*(G;V)$ is surjective for all bornological 
		${\mathcal{H}}_{{\B},L}(G)$-modules $V$, in all degrees $* \leq k+1$.
	\item[$({\B} 3$)] $G$ is ${\B}FP^k$.
\end{itemize}
\end{theorem}

Before proceeding to the proof of this theorem, we introduce necessary notation and isolate two
lemmas that highlight the importance of the finiteness assumption on $G$.

For a simplicial complex $X$ equipped with the $1$-skeleton weight function $w_X$, and a free $G$ action, consider 
$B^\B_m(X) = \partial_{m+1}(\B C_{m+1}(X)) \subset \B C_m(X)$. 
On $B^\B_m(X)$ one has \underline{filling} seminorms
$\|\,\,\,\|_{f,\lambda}, {\lambda\in{\B}}$, defined as follows:
\[
\| b\|_{f,\lambda} = \inf \left\{ \| a \|_{\lambda} \, | \, \exists a\in C_{m+1}(X)\ \  s.t.\ b = \partial(a) \right\}.
\]
This norm identifies $B^\B_m(X)$ with $\B C_{m+1}(X) / \ker_{m+1}$. 

Given a bornological $\BG$-module $V$ with bornology defined by a
collection of seminorms\footnote{We remark that this is an assumption on $V$.  Not every bornological space is of this form.}
 $\{\eta_i\}_{i\in I}$, an $m$-cochain $c\in C^m(X;V)$ is $\B$-bounded 
(i.e., lies in the subspace ${\B} C^m(X;V)$) if
\begin{equation*}
\forall\eta_i, i\in I\ \ \exists\lambda'\in{\B}\ s.t.\ \forall\sigma\in X_m,
\eta_i(c(\sigma))\leq \lambda'(w_X(\sigma)) = \|\sigma\|_{\lambda'}.
\end{equation*}
Then, as was shown in \cite{JOR1}, ${\B}H^*_G(X;V) := 
\Ext^*_{{\mathcal{H}}_{{\B},L}(G)}(\C, V) =$ the cohomology of the cochain
complex $\{{\B} C^m(X;V))\}_{m\geq 0}$. In particular, $B^\B_m(X)$, equipped with the collection of filling seminorms defined above, is a bornological module over ${\mathcal{H}}_{{\B},L}(G)$, and so an $m$-chain $c\in {\B} C^m(X;B^\B_m(X))$ satisfies
\begin{equation*}
\forall\lambda\in{\B}\ \ \exists\lambda'\in{\B}\ s.t.\ \forall\sigma\in X_m,
\|c(\sigma)\|_{f,\lambda}\leq \lambda'(w_X(\sigma)) = \|\sigma\|_{\lambda'}.
\end{equation*}

For a weighted set $(U,w_U)$, we can (as in [JOR1]) form the weighted vector space $X(U) := \mathbb C[U]$, with weighting given on basis elements by $w_U$. For any bounding class $\B$ and $\lambda\in\B$, we have an associated seminorm
$\|\sum_i\alpha_i x_i\|_{\lambda} := \sum_i|\alpha_i|\lambda(w(x_i))$.
If $(V,w_V)$ is a second weighted set, a linear map $\phi:X(U)\to X(V)$ is $\B$-bounded if for all $\lambda\in\B,\exists\lambda'\in\B$ such that $\|\phi(x)\|_{\lambda}\le \|x\|_{\lambda'} \forall x\in X(U)$.
\vskip.1in

Let $S$ and $T$ be free $G$-sets equipped with weight functions $w_S$ and $w_T$, such that $S/G$ is finite. 
Let $\{s_1,s_2,\dots,s_N\}\subset S$ be a complete set of orbit representatives satisfying $w_S(s_j)\leq w_S(gs_j)$ 
for all $1\leq j\leq N$, $g\in G$. We assume the existence of constants $D_1, D_2, D_3$ satisfying
\begin{enumerate}
	\item $w(gz)\leq D_1 L(g) + w(z)\,\,\quad\forall z\in Z = S,T,\,\,w = w_S,w_T$
	\item $L(g)\leq D_2w_S(gs_j) + D_3\,\quad\forall 1\leq j\leq N, g\in G$
\end{enumerate}

\begin{lemma}\label{bounded} 
Let $S,T$ be as above, and $V$, respectively $W$, denote the weighted vector space over $\C$ with basis $S$, respectively $T$.  
Then any $G$-equivariant linear map $V \to W$ is $\B$-bounded.
\end{lemma}

Let $W'\subset {\B}W$ be a subspace closed in the Frechet bornology ($=$ topology), and set $W''={\B}W/W'$. 

Suppose we are also given a $G$-equivariant linear map $\mathfrak{h}:V\to W''$. Then $\mathfrak{h}$ can be lifted to 
a $G$-equivariant linear map $\mathfrak{f}:V\to {\B}W$. 
By the previous lemma, this lifting is $\B$-bounded. Hence

\begin{lemma}\label{quotient} 
Any $G$-equivariant linear map $V\to W''$ is $\B$-bounded.
\end{lemma}

\begin{proof}[Proof of Theorem \ref{thm:geocritRest}]
(${\B} 1$)  obviously implies (${\B} 2$). The implication $({\B} 3)\Rightarrow ({\B} 1)$
follows by a natural extension of the arguments of \cite{JR1} and \cite{O1}.
Namely,  the $\B FP^k$ condition yields a complex $\B D_*$, which is the completion of
a $k$-nice resolution of $\C$ over $\C[G]$, admitting a bounded linear chain contraction
through dimension $k$.   For degrees less than or equal to $k$, this complex is a
free ${\mathcal{H}}_{{\B},L}(G)$-module on finitely many generators. This implies the result.

The main point is to show (${\B} 2$) implies (${\B} 3$), specifically that there exists a resolution $C_*$ of $\C$ over $\C[G]$
which is finite dimensional in degrees less than $k$,  and whose completion ${\B}C_*$ with respect to the semi-norms induced
by $\B$ yields a bounded resolution of $\C$ over $\mathcal{H}_{\B,L}(G)$ through the appropriate range.   This verification will be carried out in two
parts.  We first use the epicohomological condition (${\B} 2$) to show that for every $\lambda \in \B$ 
there is a $\lambda' \in \B$ such that for all $b \in B_m(X)$, $\| b \|_{f,\lambda} \leq \| b \|_{\lambda'}$. 
We then show inductively that this inequality guarantees the existence of a $\B$-bounded section 
$\tilde{s}_{m+1}:C_m\to C_{m+1}$.  The $\B$-boundedness of the section implies that it extends to 
the $\B$-completions as a bounded linear splitting on ${\B}C_*$.

That $G$ is finitely presented and $FP^r$ implies the existence of a contractible free $G$-simplicial complex $X$ such that in dimensions $n \leq r$ the $G$ action on $X^{(n)}$ is cofinite.  Fix a basepoint $x_0 \in X^{(0)}$, and equip $X^{(r)}$ 
with the $1$-skeleton weighting.  Also fix a family of representatives of the orbits of $G$ in $X^{(0)}$, $\cR = \{ x_0, x_1, \dots, x_l\}$
with each $x_i$ satisfying $d_{X^{(1)}}( x_i, x_0 ) \leq d_{X^{(1)}}( g x_i, x_0 )$ for all $g \in G$.  Let $\I G = \sqcup_{i=0}^{l} G_i$,
where $G_i$ is a copy of $G$.  We denote an element of $\I G$ by $(g,i)$ for $g \in G$ and $0 \leq i \leq l$.
Let $Y$ be the geometric realization of the homogeneous bar resolution on $\I G$.  Thus, $C_*(X)$ and $C_*(Y)$ are
resolutions of $\C$ over $\C[G]$.
We endow $Y$ with a weight function arising from the length function $L$ on the group $G$, given on simplices by
\[
w_Y\left([ (g_0, i_0),(g_1, i_1),\dots,(g_n, i_n)]\right) := \sum_{i=0}^n L(g_i).
\]
The weight function $w_Y$ satisfies the inequality
\begin{equation*}
	w_Y(g\sigma) \leq (n+1)L(g) +w_Y(\sigma)\qquad \forall \sigma\in Y^{(n)}
\end{equation*}
while for $w_X$ there is a constant $C$ related to the quasi-isometric equivalence between $G$ and $X^{(r)}$ with
\begin{equation*}
	w(g\sigma) \leq C (n+1) L(g) +w(\sigma)\qquad \forall \sigma\in X^{(n)}.
\end{equation*}

As in \cite[Th. 2.6]{JR1}, there exist $\C[G]$-module morphisms of chain complexes $\phi_*:C_*(Y)\to C_*(X)$ and 
$\psi_*:C_*(X)\to C_*(Y)$ with both $\phi_*\circ \psi_*$ and $\psi_*\circ\phi_*$ $G$-equivariantly chain-homotopic 
to the identity.  Let $V_{k-1} = B^\B_{k-1}(X)$ equipped with the filling seminorms, $\| \, \|_{f, \lambda}$, $\lambda \in \B$,
and let $u \in C^k_G( X; V_{k-1} )$ be the $k$-cocycle given as the composition
$C_k(X) \overset{\partial_k}{\to} B_{k-1}(X) \overset{\iota_{k-1}}{\hookrightarrow} V_{k-1}$.

The properties of $\phi$ and $\psi$ ensure there is a $G$-equivariant $(k-1)$-cocycle $v\in C^{k-1}_G(X;V_{k-1})$ with
$u = (\psi^k\circ \phi^k)(u) + \delta(v)$. Furthermore, condition (${\B} 2$) guarantees there is a $G$-equivariant $\B$-bounded 
cocycle $u'\in {\B} C^k_G(Y;V_{k-1})$ and a $G$-equivariant $(k-1)$-cochain $v'\in C^{k-1}_G(Y;V_{k-1})$
satisfying the equation $\phi^k(u) = u' + \delta(v')$. By the argument of \cite{Mi}, one has for any $b \in B_{k-1}(X)$
\begin{equation}\label{eq:u'}
	\iota_{k-1}{b} = u'([e,\psi_{k-1}(b)]) + (\psi^{k-1}(v') + v)(b)
\end{equation}
where $[e,\sum \gamma_{[g_0,\dots,g_n]}[g_0,\dots,g_n]] := \sum \gamma_{[g_0,\dots,g_n]}[e,g_0,\dots,g_n]$. 

By (\ref{eq:u'}), one has
\[
\left\|b\right\|_{f,\lambda} \leq \left\|u'([e,\psi_{k-1}(b)])\right\|_{f,\lambda} + \left\|(\psi^{k-1}(v') + v)(b)\right\|_{f,\lambda}.
\]
As $u'$ is $\B$-bounded, there exists an element $\lambda'\in{\B}$ satisfying
\[
\left\|u'([e,\psi_{k-1}(b)]\right\|_{f,\lambda}\leq \left\|[e,\psi_{k-1}(b)]\right\|_{\lambda'} = \left\|\psi_{k-1}(b)\right\|_{\lambda'}.
\]
Taking $(S,w_S) = (X^{(k-1)},w_X)$ and $(T,w_T) = (Y^{(k-1)},w_Y)$ and applying Lemma \ref{bounded} shows that
$\psi_{k-1}$ is $\B$-bounded. By Lemma \ref{quotient}, the map $(\psi^{(k-1)}(v') + v):C_{k-1}(X)\to V_{k-1}$ is also
bounded. Hence the sum must be bounded.
Thus for all $\lambda \in \B$ there is $\lambda'' \in \B$ such that for all $b \in B^\B_{k-1}(X)$
\[ \| b \|_{f,\lambda} \leq \| b \|_{\lambda''}. \]

As $\B$ is weakly countable, $\B C_k(X)$ is a Frechet space.
We define a multi-valued map $F: B^\B_{k-1}(X) \to \B C_k(X)$ by associating to an element $b + \ker \partial_{k} \in B^\B_{k-1}(X)$,
the affine subspace $\left\{ a \, | \, \partial_k(a) = b \right\}$.  This is a translate of the subspace $\ker \partial_k$
in $\B C_k(X)$.  We follow Noskov in \cite{No2} by using a version of Michael's Theorem,  generalized to Frechet spaces.
By \cite[Thm. 3.4]{Cho-Kim}, $F$ has a (not necessarily linear) single-valued continuous selection.  
Thus, for every $b \in B^\B_{k-1}(X)$ there is an $a \in \B C_k(X)$ with $\partial_k(a) = b$
such that for every $\lambda \in \B$, there is a $\lambda' \in \B$ satisfying
\[ \| a \|_{\lambda} \leq \| b \|_{f, \lambda'}. \]

The last two inequalities imply that we have a continuous map $\partial_k \B C_{k}(X) \to \B C_k(X)$, with
$\partial_k \B C_k(X) \subset \B C_{k-1}(X)$ topologized by the countable family of norms, $\left( \| \cdot \|_\lambda \right)_{\lambda \in \B}$.
We use these filling inequalities below to construct a bounded contracting homotopy for $\B C_*(X)$, by modifying the
methods of \cite{Mi}.

We set $C_{-1}(X) = \C$, making $C_*(X)$ an augmented resolution of $C_{-1}(X)$ over $\mathbb C[G]$. Let $\partial$ be 
the boundary map $\partial_i : C_i(X) \to C_{i-1}(X)$.  (We also call the boundary map from
$\B C_i(X)$ to $\B C_{i-1}(X)$ by $\partial_i$.)  Let $\B C_*(X)$ be the $\B$-completion of $C_*(X)$.
We will show that $\B C_*(X)$ admits a bounded linear contraction $\{S_n: \B C_n(X)\to \B C_{n+1}(X)\}$ through dimension $k$.

For a $v \in X^{(0)}$, let $i(v)$ be such that $v \in G \cdot x_{i(v)}$.   The augmented complex $\B C_*( \I G)$ 
is a free bornological $\B$-resolution of $\C = \B C_{-1}(\I G)$, and comes equipped with a bounded $\C$-linear contraction. 
Label it $s_i : \B C_i(\I G) \to \B C_{i+1}(\I G)$, with the boundary map $d_i : \B C_i(\I G) \to \B C_{i-1}(\I G)$.

Define $G$-equivariant $\B$-bounded chain maps $\Psi : \B C_*(X) \to \B C_*(\I G)$ and $\Phi : \B C_*(\I G) \to \B C_*(X)$ as follows.
In dimensions $n\leq -1$, set $\Psi_n : \B C_n(X) \to \B C_n(\I G)$ equal to the identity.  (In degree $-1$, this is just the map
from $\C \to \C$, and for $n < -1$, $\B C_n(X)$ and $\B C_n(\I G)$ are both zero.)  
Let $\Psi_0 : \B C_0(X) \to \B C_0( \I G )$ be the map given on simplices by $[v] \mapsto [ (g, i(v) ) ]$.  
This is a well-defined, $G$-equivariant linear map.  More generally, for $1 \leq n \leq k$
let $\Psi_n : \B C_n(X) \to \B C_n( \I G)$ be defined by $[ v_0, v_1, \dots, v_n ] \mapsto [ (g_0, i(v_0)), (g_1, i(v_1)), \dots, (g_n, i(v_n)) ]$.
We note that these maps are $\B$-bounded by Lemma \ref{bounded}.  

In order to construct the continuous $G$-equivariant chain map $\Phi : \B C_*(\I G) \to \B C_*(X)$, we first 
observe that for each point $v \in X^{(0)}$, there is a unique $g_v \in G$ and $x_i \in \cR$ such that $g_v x_i = v$, with the assignment $v \mapsto g_v$ being $G$-equivariant.  As $G$ is quasi-isometric to $X^{(1)}$, there 
exist positive constants $A$ and $B$ such that for all $g \in G$, and $x_j \in \cR$,
$d_X( g x_j, x_0 ) \leq A L( g ) + B$.

Define the $G$-equivariant linear map $\Phi_0 : C_0(\I G) \to \B C_0(X)$ by
\[ \Phi_0([ (g, i) ]) = [ g x_i ]. \]
For $\lambda$ in the bounding class $\B$, there is an $\lambda' \in \B$ such that
$\lambda( A x + B ) \leq \lambda'(x)$.  Then $\| \Phi_0([(g,i)]) \|_{\lambda} \leq \| [(g,i)] \|_{\lambda'}$.
Therefore $\Phi_0$ extends to a bounded $\B G$-module map $\Phi_0: \B C_0(\I G) \to \B C_0(X)$.

As $\Phi_0( d_1[ (g_0, i_0), (g_1, i_1)] ) = [g_1 x_{i_1}] - [g_0 x_{i_0}]$, this is a 0-boundary in $B_0(X)$.  
By the filling inequalities above, there is a 1-chain $c = c((g_0, i_0), (g_1, i_1)) \in \B C_1(X)$
such that for each $\lambda \in \B$ there is $\lambda' \in \B$ with  
$\| c((g_0, i_0), (g_1, i_1)) \|_{\lambda} \leq \| [ (g_0, i_0), (g_1, i_1)] \|_{\lambda'}$.
There could be many possible choices for $c((g_0, i_0), (g_1, i_1))$.  We make the following assumptions on our choice:
\begin{enumerate}
	\item[(1)] If $\{g_0 x_{i_0}, g_1 x_{i_1}\}$ spans a 1-simplex $[ g_0 x_{i_0}, g_1 x_{i_1}]$ in $X^{(2)}$, 
		then $c( (g_0, i_0), (g_1, i_1)) = [g_0 x_{i_0}, g_1 x_{i_1}]$.
	\item[(2)] $c$ is $G$-equivariant (as the action of $G$ on $X^{(2)}$ is free, this can always be arranged).
\end{enumerate}
Set $\Phi_1( [(g_0, i_0), (g_1, i_1)] ) := c((g_0, i_0), (g_1, i_1))$.  This yields a $G$-equivariant $\B$-bounded linear map
$\Phi_1 : \B C_1( \I G ) \to \B C_1( X )$, satisfying $\partial_1 \Phi_1 = \Phi_0 d_1$.  Note that we may
assume (1) only because $X$ is equipped with the 1-skeleton weighting; thus the weight of a simplex is bounded by
the sum of the weights of the codimension 1 faces.

Inductively, suppose that we have constructed $\Phi_n$ for some $n\leq k-1$, the map having been defined through the use of fillings
$c((g_0, i_0), (g_1, i_1), \dots, (g_n, i_n) )$ of n-boundaries $\Phi_{n-1}( d_{n}( [(g_0, i_0), (g_1, i_1), \dots, (g_n, i_n)  ] ) )$ in $X^{(n-1)}$ satisfying:
\begin{enumerate}
	\item[($P_{1,n}$)] If $g_0 x_{i_0}, g_1 x_{i_1} , \dots,  g_n x_{i_n}$ span an n-simplex $[ g_0 x_{i_0}, g_1 x_{i_1} , \dots, g_n x_{i_n}]$ in $X^{(n)}$, 
		then $c( (g_0, i_0), (g_1, i_1), \dots, (g_n, i_n)  ) = [g_0 x_{i_0}, g_1 x_{i_1} , \dots, g_n x_{i_n}]$.
	\item[($P_{2,n}$)] The choice of $c$ is $G$-equivariant.
	\item[($P_{3,n}$)] For every $\lambda \in \B$ there is $\lambda' \in \B$ such that 
	\[
\| c( (g_0, i_0), (g_1, i_1), \dots, (g_n, i_n)  ) \|_{\lambda} \leq \| [ (g_0, i_0), (g_1, i_1), \dots, (g_n, i_n)  ] \|_{\lambda'}
\]
 for all $[ (g_0, i_0), (g_1, i_1), \dots, (g_n, i_n)  ] \in C_n(\I G)$.
\end{enumerate}

Consider the $(n+1)$-chain $[(g_0, i_0), (g_1, i_1), \dots, (g_{n+1}, i_{n+1}) ]\in C_{n+1}(\I G)$.
\[ d_{n+1} ( [(g_0, i_0), (g_1, i_1), \dots, (g_{n+1}, i_{n+1})] ) = 
	\sum_{j =0}^{n+1} (-1)^j [ (g_0, i_0), \dots, \widehat{ (g_j, i_j) }, \dots, (g_{n+1}, i_{n+1}) ]. \]
Then \[ \Phi_{n} \left( d_{n+1} ( [(g_0, i_0), \dots, (g_{n+1}, i_{n+1})] )  \right) = 
	\sum_{j =0}^{n+1} (-1)^j  c( (g_0, i_0) , \dots, \widehat{ (g_j, i_j) }, \dots, (g_{n+1}, i_{n+1}) ) \]
lies in $B_{n}(X)$.  For $\lambda \in \B$, take $\lambda' \in \B$ such that for any 
$[(g_0, i_0), \dots, \widehat{(g_j, i_j)}, \dots, (g_{n+1}, i_{n+1}) ] \in C_{n}(\I G)$,
\[ \| c( (g_0, i_0), \dots, \widehat{(g_j, i_j)}, \dots, (g_{n+1}, i_{n+1})  ) \|_{\lambda} 
	\leq \| [ (g_0, i_0), \dots, \widehat{(g_j, i_j)}, \dots, (g_{n+1}, i_{n+1})  ] \|_{\lambda'}.\]
Thus 
\[\| \Phi_{n} \left( d_{n+1}( [(g_0, i_0), \dots, (g_{n+1}, i_{n+1}) ]) \right) \|_{\lambda} 
	\leq (n+1) \| [ (g_0, i_0), \dots, (g_{n+1}, i_{n+1}) ] \|_{\lambda'}.\]

By the filling norm inequality, for every 
$[(g_0, i_0), \dots, (g_{n+1}, i_{n+1})] \in C_{n+1}(\I G)$ there is an $(n+1)$-chain $c=c((g_0, i_0), \dots, (g_{n+1}, i_{n+1}))$ with
$\partial_{n+1} c((g_0, i_0), \dots, (g_{n+1}, i_{n+1})) = \Phi_n\left( d_{n+1}( [(g_0, i_0), \dots, (g_{n+1}, i_{n+1})])\right)$,
such that for every $\lambda \in \B$ there is an $\lambda' \in \B$ satisfying
$\| c((g_0, i_0), \dots, (g_{n+1}, i_{n+1})) \|_{\lambda} \leq \| [(g_0, i_0), \dots, (g_{n+1}, i_{n+1}) \|_{\lambda'}$.

We may then pick the fillings $c((g_0, i_0), \dots, (g_{n+1}, i_{n+1}))$ to satisfy the properties ($P_{j,n+1}$), analogous to those listed above, and set $\Phi_{n+1}( (g_0, i_0), \dots, (g_{n+1}, i_{n+1}) ) := c((g_0, i_0), \dots, (g_{n+1}, i_{n+1}))$.
This defines a $G$-equivariant, linear, $\B$-bounded chain map $\Phi_n : \B C_n(\I G) \to \B C_n(X)$, for $0 \leq n \leq k$.

We claim $\Phi_n \Psi_n = Id, 0 \leq n \le  k$. 
First, $\Phi_0 \Psi_0( [v] ) = \Phi_0( [ (g_v, i(v)) ] ) = [ g_v x_{i(v)} ] = [ v ]$.

In general, for $1 \leq n \leq k$, 
\begin{align*}
	\Phi_n \Psi_n( [v_0, \dots, v_n] ) &= \Phi_n( [ (g_0, i(v_0)), \dots, (g_i, i(v_n) )] ) \\
		&= c( (g_0, i(v_0)), \dots, (g_i, i(v_n) ) ) = [v_0,\dots,v_n]
\end{align*}	
as $\{g_0 x_{i(v_0)}=v_0, \dots, g_n x_{i(v_n)} = v_n\}$ are the vertices of the $n$-simplex $[v_0, \dots, v_n]$.

For $0 \leq n < k$, consider the map $S_n:  \B C_n(X) \to \B C_{n+1}(X)$ given by the composition $S_n = \Phi_{n+1} s_n \Psi_n$. A routine calculation shows $\partial_{n_1} S_n + S_{n-1} \partial_n = \Phi_n \Psi_n$.
Thus $\{S_n\}$ is a chain contraction. Moreover, as $\Phi$, $\Psi$, and $s$ are all $\B$-bounded, so is $S$,  yielding the necessary bounded contraction in dimensions $n \leq k$, and completing the verification of ($B$3).
\end{proof}

Applying this argument degreewise, we obtain the following.
\begin{theorem}\label{thm:geocrit} 
Let $G$ be a finitely presented discrete group of type $FP^{\infty}$ and $\B$ a bounding class. The following are equivalent. 
\begin{itemize}
	\item[(${\B} 1'$)] $(G,L)$ is strongly $\B$-isocohomological.
	\item[(${\B} 2'$)] ${\B}H^*(G;V)\to H^*(G;V)$ is
		surjective for all bornological ${\mathcal{H}}_{{\B},L}(G)$-modules $V$.
	\item[$({\B} 3'$)] $G$ is type $\B FP^{\infty}$.
\end{itemize}
\end{theorem}



\section{Relative constructions}
We show how the results in the previous section may be ``relativized''.

\subsection{Relative $HF^{n}$ and the Brown-Bieri-Eckmann condition}\label{sect:RelBBE}

In this subsection, $n$ will denote an arbitrary cardinal $\leq \infty$. Given a family of subgroups $\{H_{\alpha}\}_{\alpha\in\Lambda}$ of $G$, let
\[
\Delta = \Delta(\{H_{\alpha}\}_{\alpha\in\Lambda}) := \ker\left(\underset{\alpha\in\Lambda}
	{\bigoplus}\Z [G/H_{\alpha}]\overset{\varepsilon}\to \Z \right)\label{def:fpn1}
\]
where $\varepsilon$ is the linear extension of the map $gH_{\alpha}\mapsto 1$. Then $\Delta$ is a 
$\Z [G]$-module; given a second $\Z [G]$-module $A$, the homology of $G$ relative 
to the family of subgroups $\{H_{\alpha}\}_{\alpha\in\Lambda}$ with coefficients in $A$ is \cite{BE2}
\[
H_*(G,\{H_{\alpha}\}_{\alpha\in\Lambda};A) := \Tor_{*-1}^{\Z [G]}(\Delta,A)
\]
Algebraically, this makes perfect sense regardless of how the groups intersect. Our object is to establish 
necessary and sufficient conditions for relative finiteness. A natural starting point is

\begin{lemma}\label{lemma:fpn}
Suppose that
\begin{enumerate}
\item the indexing set $\Lambda$ is finite,
\item each subgroup $H_{\alpha}$ is finitely generated,
\item $G$ is finitely presented,
\item $\Tor_*^{\Z [G]}(\Delta,\prod \Z [G]) = 0$ for all $1\leq *\leq (n-1)$ and direct 
	products $\prod \Z [G]$ of copies of $\Z [G]$.
\end{enumerate}
Then $\Delta$ is type $FP^{n}$ over $\Z [G]$.
\end{lemma}

\begin{proof}  Consider the commuting diagram
\begin{center}
\centerline{
\xymatrix{
\txt{\phantom{XX}} 0\txt{\phantom{XX}}\ar[d]\ar@<.75ex>@{>->}[r] & \underset{\alpha\in\Lambda}{\oplus} \Z [G]\ar@<.75ex>[r]^{=}\ar@{->>}[d]^{p_1} & 
 \underset{\alpha\in\Lambda}{\oplus} \Z [G]\ar@{->>}[d]^{p_2}\\
\txt{\phantom{XX}}\Delta\txt{\phantom{XX}}\ar@<.75ex>@{>->}[r]  & \underset{\alpha\in\Lambda}{\oplus}
 \Z [G/H_{\alpha}]\ar@<.75ex>@{->>}[r]^-{\varepsilon} & \Z 
}}
\end{center}
where on each summand $p_1$ is the natural projection $\Z [G]\twoheadrightarrow \Z [G/H_{\alpha}]$ 
and $p_2 = \varepsilon\circ p_1$. Denote $\underset{\alpha\in\Lambda}{\oplus} \Z [G/H_{\alpha}]$ by $E$, 
$\Z $ by $B$, and set $P(E) :=  \underset{\alpha\in\Lambda}{\oplus} \Z [G]$, $\Omega(E) := \ker(p_1)$, 
$\Omega(B) := \ker(p_2)$. By the Snake Lemma, the above diagram yields a short-exact sequence
$\Omega(E)\rightarrowtail \Omega(B)\twoheadrightarrow \Delta$. Now consider the pull-back diagram
\begin{center}
\centerline{
\xymatrix{
\Omega(E) \ar@{=}[r]\ar@{>->}[d] & \Omega(E)\ar@{>->}[d]\\
P(E,\Delta)\ar@{->>}[d]\ar@{>->}[r] & P(E)\ar@{->>}[d]\\
\Delta\ar@{>->}[r] & E
}}
\end{center}
There is a natural isomorphism of $\Z [G]$-modules $P(E,\Delta)\isom \Omega(B)$, from which we conclude the existence of a short-exact sequence
\begin{equation}\label{eq:short}
\Omega(B)\isom P(E,\Delta)\rightarrowtail \Delta\oplus P(E)\twoheadrightarrow E
\end{equation}
Conditions {\it 1.} and {\it 2.} imply $E$ is finitely-presented over $\Z [G]$, and condition {\it 3.} implies $\Omega(B)$ is finitely-presented over 
$\Z [G]$. By (\ref{eq:short}), $\Delta\oplus P(E)$ is finitely-presented over $\Z [G]$. But $P(E)$ is a finitely-generated free module over 
$\Z [G]$ (by {\it 1.}), so $\Delta$ itself must be finitely-presented over $\Z [G]$. The result now follows from Prop. 1.2 of \cite{BE1}.
\end{proof}

\begin{definition} 
The group $G$ is type $FP^{n+1}$ (resp. $FF^{n+1}$) rel.\ $\{H_{\alpha}\}_{\alpha\in\Lambda}$ if the 
$\Z [G]$-module $\Delta$ is type $FP^n$ (resp. $FF^n$) over $\Z [G]$.
\end{definition}

As usual, if we are not concerned with constructing finite resolutions but only ones which are finite-dimensional through the 
given range, the conditions $FF^n$ and $FP^n$ agree.

We consider an alternative definition, which will provide the bridge between the algebraic and topological setting.

\begin{definition}\label{def:relfpn2} 
A \underline{resolution of $\Z $ over $\Z [G]$ relative to a family of subgroups}
\underline{$\{H_{\alpha}\}_{\alpha\in\Lambda}$} is a projective resolution $P_*$ of $\Z $ over $\Z [G]$ satisfying
\[
P_m = \oplus_{\alpha\in\Lambda} Ind^G_{H_{\alpha}}(W^{\alpha}_m)\oplus Q_m
\] 
where for each $\alpha\in\Lambda$, $W_*^{\alpha}$ is a projective resolution of $\Z $ over $\Z [H_{\alpha}]$ and
$Ind^G_{H_{\alpha}}(W^{\alpha}_*)\hookrightarrow P_*$ is an inclusion of complexes. Then $G$ is type $\widetilde{FP}^n$, respectively type 
$\widetilde{FF}^n$, if $Q_m$ is finitely generated projective, respectively free over $\Z [G]$, for all finite $m\leq n$.
\end{definition}

Again, when the resolutions are allowed to be infinite, $\widetilde{FP}^n$ and $\widetilde{FF}^n$ are equivalent conditions. 

\begin{lemma}\label{lemma:equiv} 
If $G$ is type $FP^n$ relative to $\{H_{\alpha}\}_{\alpha\in\Lambda}$, it is type $\widetilde{FP}^n$ relative 
to $\{H_{\alpha}\}_{\alpha\in\Lambda}$. If conditions {\it 1.} - {\it 3.} in Lemma \ref{lemma:fpn} are satisfied, then the converse is true.
\end{lemma}
\begin{proof} 
Let $R_*$ be a resolution of $\Delta$ over $\Z [G]$, and let $S_*$ be the $\Z [G]$-resolution 
$\underset{\alpha\in\Lambda}{\oplus} Ind^G_{H_{\alpha}}(W_*^{\alpha})$ of $\underset{\alpha\in\Lambda}{\oplus} \Z [G/H_{\alpha}]$ where $W^{\alpha}_*$ 
is a $\Z [{H_{\alpha}}]$-resolution of $\Z $ for each $\alpha\in\Lambda$.  Let $f_*:R_*\to S_*$ be a map of $\Z [G]$-resolutions 
covering the inclusion $\Delta\rightarrowtail \underset{\alpha\in\Lambda}{\oplus} \Z [G/H_{\alpha}]$, and let $M(f_*)$ be the algebraic mapping cone of 
$f_*$. Then $M(f_*)$ provides a resolution of $\Z $ over $\Z [G]$ rel.\ $\{H_{\alpha}\}_{\alpha\in\Lambda}$ where $Q_m = R_{m-1}$. Hence type 
$FP^n$ rel.\ $\{H_{\alpha}\}_{\alpha\in\Lambda}$ implies type $\widetilde{FP}^n$ rel.\ $\{H_{\alpha}\}_{\alpha\in\Lambda}$. Now suppose $G$ is type 
$\widetilde{FP}^n$  rel.\ $\{H_{\alpha}\}_{\alpha\in\Lambda}$. As we have seen, the first three conditions of Lemma \ref{lemma:fpn} imply $\Delta$ is 
finitely-presented. Let $T_*$ be a $\Z [G]$-resolution of $G$  rel.\ $\{H_{\alpha}\}_{\alpha\in\Lambda}$. Then $S_*$ is a subcomplex of $T_*$, with the 
inclusion $S_*\hookrightarrow T_*$ covering the projection $\varepsilon:\underset{\alpha\in\Lambda}{\oplus} \Z [G/H_{\alpha}]\twoheadrightarrow 
\Z $. The quotient complex $T_*/S_* = Q_* = \{Q_m\}$ satisfies the property
\[
\Tor_*^{\Z [G]}(\Delta,A) = H_{*+1}^G(Q_*\tensor A) = H_{*+1}(Q_*\underset{\Z [G]}{\tensor} A)
\]
If $A = \prod \Z [G]$ and $Q_m$ is finitely-generated projective for finite $m\leq n$, then
\[
\Tor_k^{\Z [G]}\left(\Delta,\prod \Z [G]\right) = H_{k+1}^G\left(Q_*\tensor \prod \Z [G]\right) = 
H_{k+1}\left(\left\{\prod Q_m\right\}_{m\geq 0}\right) = 0
\]
for all finite $k\leq (n-1)$. By Lemma \ref{lemma:fpn} this completes the proof.
\end{proof}

We now consider the topological analogue. Let $\{A_{\alpha}\}_{\alpha\in\Lambda}$ be a family of subcomplexes of a complex $X$, and let 
$A = \bigcup_{\alpha\in\Lambda} A_{\alpha}\subseteq X$.
We say that $X$ is type \underline{$HF^{n}$ relative to $\{A_{\alpha}\}_{\alpha\in\Lambda}$}  if $X/A \homotopic Y$ 
with $Y$ having finitely many cells (or simplices) in each finite dimension $m \leq n$. 

If $\{H_{\alpha}\}_{\alpha\in\Lambda}$ is a family of subgroups of $G$, then whenever $\Lambda$ contains more 
than one element $\coprod BH_{\alpha}$ will not be a subspace of the standard model for $BG$, as the classifying 
spaces $BH_{\alpha}$ all contain the common basepoint (and more if the intersections are non-trivial). However, 
the $\Z [G]$-module $\Delta$ is modeled on the disjoint union of the classifying spaces $\{BH_{\alpha}\}$. 
To accommodate this, we will need a different model for $BG$. Recall that if $T$ is a discrete set, one may form 
the free simplicial set $S_{\bullet}(T)$ generated by $T$, with 
\begin{gather}
S_m(T) := T^{m+1};\\
\partial_i(t_0,t_1,\dots,t_n) = (t_0,t_1,\dots,\widehat{t_i},\dots,t_n);\\
s_j(t_0,t_1,\dots,t_n) = (t_0,t_1,\dots,t_j,t_j,\dots,t_n)
\end{gather}
In other words, the face and degeneracy maps are given by deletion and repetition. Any element of $T$ can be used 
to define a simplical contraction of $S_{\bullet}(T)$, yielding $S_{\bullet}(T)\homotopic *$ for all sets $T$. 
Moreover, if $T$ is a free $G$-set, then the diagonal action of $G$ makes $S_{\bullet}(T)$ a simplicial free $G$-set, 
hence a simplicial model for a universal contractible $G$-space. The standard homogeneous bar resolution of $G$ - $EG$ - 
arises when one takes $T = G$ with left $G$-action given by multiplication. 

\begin{definition} 
For an indexing set $\Lambda$, let
\[
G(\Lambda) := \underset{\alpha\in\Lambda}{\coprod} G
\]
with $G$-action given by left multiplication on each component. Then
\begin{gather*}
EG(\Lambda) := S_{\bullet}(G(\Lambda))\\
BG(\Lambda) := EG(\Lambda)/G
\end{gather*}
\end{definition}

Again, if $\{T_{\alpha}\}_{\alpha\in\Lambda}$ is a collection of $G$-sets, there is an evident inclusion of simplicial $G$-sets
\begin{equation}\label{eqn:inc}
\underset{\alpha\in\Lambda}{\coprod} S_{\bullet}\left(T_{\alpha}\right)\overset{i_{\Lambda}}{\hookrightarrow} 
S_{\bullet}\left(\underset{\alpha\in\Lambda}{\coprod} T_{\alpha}\right)
\end{equation}
which is both equivariant and functorial.

\begin{definition} A group $G$ is of type  $HF^n$ relative to a family of subgroups $\{H_{\alpha}\}_{\alpha\in\Lambda}$ 
($n\leq \infty$) if $BG(\Lambda)$ is of type $HF^n$ relative to  the image of $\underset{\alpha\in\Lambda}{\coprod} BH_{\alpha}$ 
under the composition
\[
\underset{\alpha\in\Lambda}{\coprod} BH_{\alpha}\hookrightarrow \underset{\alpha\in\Lambda}{\coprod} BG\overset{i_{\Lambda}}{\hookrightarrow} BG(\Lambda)
\]
\end{definition}

\begin{proposition}\label{prop:hftoff}  
Assume the indexing set $\Lambda$ is finite. If $G$ is type $HF^{n}$ relative to 
$\{H_{\alpha}\}_{\alpha\in\Lambda}$, then it is type $\widetilde{FF}^{n}$ relative to 
$\{H_{\alpha}\}_{\alpha\in\Lambda}$.
\end{proposition}

\begin{proof} 
For each $\alpha$, fix a simplicial complex $X_{H_{\alpha}}\homotopic BH_{\alpha}$. The condition 
on $G$ implies we may construct a simplicial model $X_G\homotopic BG$ by adding simplices (i.e., cells) 
to $\underset{\alpha\in\Lambda}{\coprod} X_{H_{\alpha}}$ in such a way that through each finite dimension 
$m$ with $0\leq m\leq n$ the number of simplices attached is finite. Let ${\widetilde X}_G\homotopic EG$ 
be the universal cover of $X_G$, and $p: \widetilde{X}_G\to X_G$ the covering map. Taking 
$P_* = C_*(\widetilde{X}_G)$ gives the desired resolution, with $Ind^G_{H_{\alpha}}(W_*) = C_*(p^{-1}(X_{H_{\alpha}}))$ 
and $Q_n$ the free $\C[G]$-submodule of $P_n$ spanned over $\C$ by those $n$-cells not in 
$p^{-1}\left(\underset{\alpha\in\Lambda}{\coprod} X_{H_{\alpha}}\right)$.
\end{proof}

Putting it all together, we may summarize the situation as
\begin{theorem} If the conditions {\it 1.} - {\it 4.} of Lemma \ref{lemma:fpn} are satisfied, then the following are equivalent
\begin{itemize}
\item  $G$ is type $HF^{n}$ relative to  $\{H_{\alpha}\}_{\alpha\in\Lambda}$,
\item  $G$ is type $\widetilde{FF}^{n}$ relative to  $\{H_{\alpha}\}_{\alpha\in\Lambda}$,
\item  $G$ is type $FF^{n}$ relative to  $\{H_{\alpha}\}_{\alpha\in\Lambda}$.
\end{itemize}
\end{theorem}

\begin{proof}  
The second and third properties are equivalent by Lemma \ref{lemma:equiv}, and the first property implies 
the second by the previous proposition. The converse to Proposition \ref{prop:hftoff}, in the presence of 
conditions {\it 1.} through {\it 4.}, follows by the same method of geometrically realizing the resolution 
as in the classical proof of the Eilenberg-Ganea-Wall theorem (compare Thm. 7.1, Chap. VIII of \cite{Br1}).
\end{proof}

Recall from \cite{Br2} that a direct system of groups $\{A_{\beta}\}$ is said to be essentially trivial if 
for each $\beta_1$ there is a $\beta_2\geq \beta_1$ such that the map $A_{\beta_1}\to A_{\beta_2}$ is trivial. 

\begin{definition} 
A \underline{filtration of $EG(\Lambda)$ of finite $n$-type relative to  $\{H_{\alpha}\}_{\alpha\in\Lambda}$} 
is an increasing filtration of $EG(\Lambda)$ by a direct system of subcomplexes $\{X_{\beta}\}$ satisfying
\begin{itemize}
\item $\underset{\beta}{\varinjlim} X_{\beta} = EG(\Lambda)$,
\item $X_{\emptyset} := p^{-1}\left(\underset{\alpha\in\Lambda}{\coprod} BH_{\alpha}\right)\subseteq X_{\beta}$ for all $\beta$,
\item For each $\beta$, $X_{\beta}/X_{\emptyset}$ contains finitely many $G$-orbits in finite dimensions $\leq n$.
\end{itemize}
\end{definition}

\begin{theorem} 
Suppose there exists a filtration $\{X_{\beta}\}$ of finite $n$-type of $EG$ relative to 
$\{H_{\alpha}\}_{\alpha\in\Lambda}$. Assume conditions 1.\ - 4. of Lemma \ref{lemma:fpn}. 
If the directed system $\{H_*(X_{\beta})\}$ is essentially trivial for all finite $*\leq n$, 
then $\Delta$ is type $FP^{n-1}$ over $\Z [G]$.
\end{theorem}

\begin{proof} 
Let $R = \Z [G]$. Then for all finite $m$ with $1\leq m\leq (n-1)$
\begin{gather*}
\Tor^{R}_m\left(\Delta,\prod R\right)\isom \Tor_m^{R}\left(R,\Delta\tensor \prod R\right)\\
= H_m^G\left(EG(\Lambda);\Delta\tensor\prod R\right)\\
\isom H^G_{m+1}\left(EG,X_{\emptyset};\prod R\right)\\
\isom \underset{\beta}{\varinjlim} H^G_{m+1}\left(X_{\beta}/X_{\emptyset};\prod R\right)\\
\isom \underset{\beta}{\varinjlim} \prod H^G_{m+1}\left(X_{\beta}/X_{\emptyset};R\right)\\
=  \underset{\beta}{\varinjlim} \prod H_{m+1}(X_{\beta}/X_{\emptyset}) =0
\end{gather*}
with the last equality following by Lemma 2.1 of \cite{Br2}.
\end{proof}
Note that unlike Brown's condition in the absolute case (where we are not working relative to a 
family of subgroups), the quotient space $EG(\Lambda)/X_{\emptyset}$ has the homotopy-type of a 
wedge $\bigvee S^1$, with $H_1(\bigvee S^1) = \Delta$. This non-contractibility in dimension $1$ 
makes verifying the essential triviality of relative filtrations problematic.


\subsection{Relative Dehn functions}
There is a natural algebraic way to define the relative Dehn functions of $G$
with respect to some finite family of subgroups.  As before, we write $\Delta$ for the kernel of the augmentation map 
$\epsilon : \bigoplus_{\alpha \in \Lambda} \Z[ G/H_\alpha] \to \Z$, where
$\mathcal{H} = \{ H_\alpha \, | \, \alpha \in \Lambda\}$.  A typical element of
$\bigoplus_{\alpha \in \Lambda} \Z[ G/H_\alpha]$
has the form $\sum_{g \in \sqcup G/H_\alpha} \lambda_g g[H_{\alpha_g}]$.  Thus to
define the weighted structure
on $\Delta \subset \bigoplus_{\alpha \in \Lambda} \Z[ G/H_\alpha]$, it is enough to
define the weight of a generator, $g[H_\alpha]$.
Set $w( g[H_\alpha] ) = \min_{h \in H_\alpha} L( gh )$, where $L$ is the length
function equipped on $G$.

If $G$ is of type $FF^\infty$ relative to $\mathcal{H}$, there is a free resolution
of $\Delta$ over $\Z[G]$:
\[ \ldots \to R_2 \to R_1 \to R_0 \to  R_{-1} :=\Delta \to 0 \] 
which is finitely generated $\Z [G]$-module in each dimension on generating set $S_n = \{s_{1,n},\dots,s_{k_n,n}\}$.  
Fixing the weight of each generator to be $1$, we extend this to a weight function on $R_n$ in the usual way by
\[
w\left(\sum \lambda_i g_i s_{j_i,n}\right) := \sum |\lambda_i|\left(L(g_i) + w(s_{j_i,n})\right) = \sum |\lambda_i|(L(g_i) + 1)
\]

Up to linear equivalence, this definition is independent of the initial weighting given to the generating set. Now
choose a $\Z$-linear splitting of this resolution $\{s_n : R_n \to R_{n+1}\}_{n \geq -1}$; associated to this 
contraction are its Dehn functions, which we refer to as the \underline{relative Dehn functions of $G$ with respect 
to $\RH$}. Again, up to linear equivalence, the definition of the Dehn functions is independent of the choice 
of linear splitting.

\begin{definition}
The relative Dehn functions of $G$ with respect to $\mathcal{H}$ are
\underline{$\B$-bounded} if there is type $FF^\infty$ resolution of $\Delta$ over 
$\Z[G]$ such that each splitting $s_n$ is bounded by an element of $\B$.
\end{definition}

\begin{lemma}\label{lem:RelDehnFnctn}
Suppose $\mathcal{L}\preceq \B$.  
If the relative Dehn functions of $G$ with respect to $\mathcal{H}$ are $\B$-bounded
for a particular type 
$FF^\infty$ resolution, then they are $\B$-bounded for all type $FF^\infty$
resolutions. 
\end{lemma}
This is a relative version of the statement that the Dehn functions of a group $G$
does not depend on which 
type $HF^\infty$ classifying space is used in their construction, up to equivalence.

Since $G$ is of relative type $FF^\infty$, it is also of relative type
$\widetilde{FF}^\infty$, with respect to $\mathcal{H}$.
As above, this gives a projective resolution of $\Z$ over $\Z[G]$ of the form
\[ P_m = \bigoplus_{\alpha \in \Lambda} Ind_{H_\alpha}^G (W_m^\alpha ) \oplus Q_m  \]
where for each $\alpha \in \Lambda$, $W_*^\alpha$ is a projective resolution of $\Z$
over
$\Z[H_\alpha]$, and $Q_m$ is a finitely generated free $\Z[G]$-module.
At each level,  $Ind_{H_\alpha}^{G}( W_m^\alpha )$ is a direct summand of $P_m$. 
Taking the quotient by them in each degree yields the sequence
\[ \ldots \to Q_3 \to Q_2 \to Q_1 \twoheadrightarrow Q_0 = 0 \]
which is exact above dimension $1$, and for which the cokernel of
$Q_2 \to Q_1$ is $\Delta$. Thus
\[ \ldots \to Q_3 \to Q_2 \to Q_1  \to \Delta \to 0 \]
is a type $FF^\infty$  resolution of $\Delta$ over $\Z[G]$.

Because the homology of both $P_*$ and $Q_*$ vanish above dimension $1$, one can in that range construct 
a chain contraction $\{s^P_n:P_n \to P_{n+1}\}$ for which the composition
\[
s^Q_n := Q_n \hookrightarrow P_n\overset{s^P_n}{\longrightarrow} P_{n+1}\twoheadrightarrow
Q_{n+1}
\]
yields a chain contraction of $Q_*$ for $* > 1$. This splitting, spliced together with a weight-minimizing 
section $s^Q_0:\Delta\to Q_1$ of the projection $Q_1 \twoheadrightarrow \Delta$, one can define the 
\underline{topological relative Dehn functions of $G$ relative to $\RH$} to be the Dehn functions associated 
to the linear contraction $\{s^Q_n\}_{n\geq 0}$. We say the topological Dehn functions are $\B$-bounded if 
the Dehn functions of $\{s^Q_n\}_{n\geq 0}$ are $\B$-bounded.

\begin{lemma}
Suppose $G$ is of type $FF^\infty$ relative to $\mathcal{H}$.  $G$ has $\B$-bounded algebraic
relative Dehn functions with respect to $\RH$ if and only if it has $\B$-bounded topological
relative Dehn functions with respect to $\RH$ 
\end{lemma}
This follows immediately from Lemma \ref{lem:RelDehnFnctn}.

The term ``topological relative Dehn function'' is justified by the following interpretation 
of them. Assume $G$ is $HF^{\infty}$ relative to $\mathcal{H}$, as before. Start with a 
relative homology cycle in $x\in Z_n(EG(\Lambda), E(\mathcal{H}))$. Then 
$\partial(x)\in C_{n-1}(E({\RH}))$. Each part of $\partial(x)$ lying in a connected component 
of $E({\RH})$ can be coned off, yielding an absolute cycle  $x'\in Z_n(EG(\Delta))$. Choose a 
weight-minimizing $y\in C_{n+1}(EG(\Delta))$ with $\partial(y) = x'$, and take its weight 
relative to the subspace $E(\RH)$; i.e., only total the weights of those $(n+1)$-simplices 
used to construct $z$ which do not lie in $E(\mathcal{H})$. The resulting Dehn function 
computed using this construction agrees (up to the usual equivalence of Dehn functions) with 
the one derived from $\{s^Q_n\}_{n \geq 0}$.

%
%

One can view the nonexistence of a relative Dehn function in a particular degree
as an obstruction to completing the type $FF^\infty$ resolution of $\Delta$,
\[ \ldots \to Q_3 \to Q_2 \to Q_1  \to \Delta \to 0 \]
to yield a type $FF^\infty$ bornological resolution of $\BD$ (defined below).
\[ \ldots \to \B Q_3 \to \B Q_2 \to \B Q_1  \to \BD \to 0 \]
Although there is always a bounded section $\Delta \to Q_1$,
the obstruction to constructing a bounded section
$Q_1 \to Q_2$ is, in general, nontrivial in the unweighted setting.
The relative Dehn functions for $n >1$ do not suffer the same issue.


\subsection{Relative $\B$-bounded cohomology}\label{sect:RelBBC}

We construct a relative $\B$-bounded cohomology theory, which closely mirrors the 
construction of relative group cohomology in \cite{Au, BE2}.
Let $G$ be a discrete group and let $\RH = \{ H_\alpha\, | \, \alpha \in \Lambda \}$
be a finite collection of subgroups of $G$.
Let $G/\RH$ be the disjoint union $\bigsqcup_{\alpha \in \Lambda} G/H_\alpha$.
For a subgroup $H$ of $G$ let $\C[G/H]$ be the $\C$-vector space with basis
the left cosets $G/H$.  
Let $\CGRH = \oplus_{\alpha \in \Lambda} \C[G/H_\alpha]$ which will be identified as finitely supported 
functions $G/\RH \to \C$.  Denote the kernel of the augmentation $\varepsilon : \CGRH \to \C$
by $\Delta$.  

\begin{definition}
The relative cohomology of a discrete group $G$, with respect to a 
collection $\RH$ of subgroups of $G$ with coefficients in a $\C[G]$-module A, 
is given by \[ H^k(G, \RH; A) = \Ext_{\C[G]}^{k-1}(\Delta, A). \]
\end{definition}

Denote by $H^k(\RH;A) = \prod_{\alpha \in \Lambda} H^k( H_\alpha; A)$.  The definition of relative 
cohomology yields the following consequence, proved in \cite{Au} for the case of a single subgroup and in 
\cite{BE2} for many subgroups.

\begin{theorem}[Auslander, Bieri-Eckmann]\label{rellongexact}
Let $G$ and $\RH$ be as above.  For any $\C[G]$-module $A$ there is a long exact sequence:
\[  \ldots  \to H^k(G;A) \to H^k(\RH;A) \to H^{k+1}(G,\RH;A) \to H^{k+1}(G;A) \to  \ldots .\]
\end{theorem}

The length-function $L$ on $G$ induces a weight, $w$, on the cosets $G/H$ given by
$w( gH ) = \min_{h \in H} L(gh)$.  With this weight, define the following
bornological $\BG$-module:  

\[ \BGH = \{ f : G/H \to \C \, | \,  \forall_{\phi \in \B} \sum_{x \in G/H} |f(x)|\phi(w(x)) < \infty \}. \]  
This is a Frechet space in the norms given by
\[ \| f \|_\phi = \sum_{x \in G/H} | f( x ) | \phi(w( x )). \]

Let $\RH = \{ H_\alpha \, | \, \alpha \in \Lambda \}$ be a finite collection of subgroups of $G$, and define  
\[\BGRH = \{ f : G/\RH \to \C \, | \, \forall_{\phi \in \B} \sum_{x \in G/\RH} |f(x)|\phi(w(x)) < \infty \}.\]
As $\RH$ is a finite family of subgroups, $\BGRH = \oplus_{\alpha \in \Lambda} \BGHi$.

The augmentation map $\varepsilon : \BGRH \to \C$ is given by
$\varepsilon( f ) = \sum_{x\in G/\RH} f(x)$.
As $\varepsilon( f ) \leq \| f \|_1$, $\varepsilon$ is a bounded map.
Denote the augmentation kernel by $\BD = \ker \varepsilon$.

\begin{definition}
The \underline{relative $\B$-bounded cohomology} of a discrete group $G$, with respect to a collection 
$\RH$ of subgroups of $G$ with coefficients in an $\BG$-module A, is given by 
${\B H^k}(G, \RH; A) = \Ext_\BG^{k-1}(\BD, A)$, where this $\Ext_\BG^*( \cdot, A)$ functor 
is taken over the category of bornological $\BG$-modules.
\end{definition}

As in the absolute cohomology theory, there is a comparison homomorphism 
$\B H^*( G, \RH; V ) \to H^*( G, \RH; V )$ for any bornological $\BG$-module $V$.
\begin{definition}
Let $G$ be a discrete group with length function $L$, and let $\RH$ be a finite collection of 
subgroups, and let $V$ be a $\BG$-module.  
We say $G$ is relatively $\B$-isocohomological
to $\RH$ with respect to $V$ (abbr. $V$-$\B$RIC) if the relative comparison 
$\B H^*( G, \RH; V ) \to H^*( G, \RH; V )$ is an isomorphism of cohomology groups in all degrees.
Similarly $G$ is relatively $\B$-isocohomological to $\RH$ (abbr. $\B$-RIC) if
it is $\C$-$\B$RIC, and $G$ is strongly relatively $\B$-isocohomological to $\RH$
(abbr. $\B$-SRIC) if it is $V$-$\B$RIC for all $\BG$-modules $V$.
\end{definition}

If $G$ is a group and $H$ is a subgroup, there is an isomorphism: 
\[\Ext_{\C[G]}^*( \C[G/H], \C ) \isom \Ext_{\C[H]}^*( \C, \C).\]  A first step in extending relative
cohomology to the $\B$-bounded framework will be the following analogue.

\begin{lemma}\label{shap}
Let $G$ be a group with length function $L$, $H$ a subgroup of $G$ equipped with the
restricted length function, and $\B$ a multiplicative bounding class.  For any
bornological $\BG$-module $A$, there is an isomorphism:
\[ \Ext_{\BG}^*( \BGH, A ) \isom \Ext_{\BH}^*( \C, A). \]
\end{lemma}

Before proving Lemma \ref{shap} we first turn our attention to a few additional results
which will be necessary.
\begin{lemma}
Let $\B$ be a multiplicative bounding class, and endow the coset space $G/H$ with the weight
$w(gH) = \min_{h \in H} L( gh )$, and $L$ is the length function on $G$.
Then $\BG \isom \BGH \ttensor \BH$ both as bornological 
vector spaces and as right $\BH$-modules.
\end{lemma}
\begin{proof}
Let $R$ be a system of minimal length representatives for left cosets of $H$ in $G$.
For an $r \in R$, the length of $r$ in $G$ is equal to the length of $rH \in G/H$, 
so $\BGH \isom \BR$ as bornological vector spaces.  
For $g \in G$ there is a unique $h_g \in H$ and $r_g \in R$ such that $g = r_g h_g$.

Let $\Phi : \BG \rightarrow \BR \ttensor \BH$ be defined on basis elements by
$\Phi( g ) = ( r_g ) \tensor ( h_g )$ and extended by linearity.
For $\lambda, \mu \in \B$, let $\nu \in \B$ such that $\nu(n) \geq \lambda(2n) \mu(2n)$.
\begin{eqnarray*}
| \Phi( g ) |_{\lambda,\mu} & = & | ( r_g ) \tensor ( h_g ) |_{\lambda,\mu}\\
 & = & \lambda(L(r_g)) \mu(L(h_g))\\
 & \leq & \lambda(L(g)) \mu(L(r_g^{-1}) + L(r_g h_g))\\
 & \leq & \lambda(L(g)) \mu( 2L(g) )\\
 & \leq & \lambda( 2L(g) ) \mu( 2L(g) )\\
 & \leq & \nu( L(g) ).
\end{eqnarray*}
It follows that $\Phi$ is bounded.

Conversely let $\Psi' : \BR \times \BH \rightarrow \BG$ be defined by
$\Psi'( ( r, h ) ) = ( rh )$.  For $\lambda \in \B$, let $\lambda'(n) = \lambda(2n)+1$.
By the properties of bounding classes, $\lambda' \in \B$ as well.  
For $r \in R$ and $h \in H$, set $M_{r,h} = \max \{ L(r), L(h)\}$ and $m_{r,h} = \min\{ L(r), L(h)\}$.
We have:
\begin{eqnarray*}
| \Psi'( r,h ) |_\lambda & = & | (rh) |_\lambda \\
 & = & \lambda( L( rh ) )\\
 & \leq & \lambda( L(r) + L(h) )\\
 & \leq & \lambda( 2M_{r,h} )\\
 & \leq & \lambda'(M_{r,h})\\
 & \leq & \lambda'(M_{r,h}) \lambda'( m_{r,h} )\\
 & = & \lambda'( L(r) ) \lambda'( L(h) ).
\end{eqnarray*}
As $\Psi'$ is bounded, it extends to a bounded
$\Psi : \BR \ttensor \BH \rightarrow \BG$.
These are the required bornological isomorphisms.
\end{proof}

\begin{lemma}
For any bounding class $\B$, $\BG \ttensor_{\BH} \C \isom \BGH$, where $H$ is endowed with the 
restricted length function and $G/H$ is endowed with the weight $w$.
\end{lemma}
\begin{proof}
By \cite{M2}, if $E = H \ttensor A$ then $E \ttensor_{A} F \isom H \ttensor F$.  Appealing to
the previous lemma we obtain the following.
\begin{eqnarray*}
\BG \ttensor_{\BH} \C & \isom & \left(\BGH \ttensor \BH \right) \ttensor_{\BH} \C \\
	& \isom & \BGH \ttensor \C \\
	& \isom & \BGH.
\end{eqnarray*}
\end{proof}

\begin{proof}[Proof of Lemma \ref{shap}]
Consider 
\[ \ldots \rightarrow \BH^{\ttensor 3} \rightarrow \BH \ttensor \BH \rightarrow  \BH \rightarrow \C \rightarrow 0 \]
where the boundary map $d_{n+1} : \BH^{\ttensor n+1} \rightarrow \BH^{\ttensor n}$ is given by
\begin{eqnarray*}
d_{n+1}( h_0, h_1, \ldots, h_n ) & = & ( h_0 h_1, h_2, \ldots, h_n ) - ( h_0, h_1 h_2, h_3, \ldots, h_n )\\
 &  & + \ldots + (-1)^{n-1}( h_0, h_1, \ldots, h_{n-1} h_n ) \\
 &  & + (-1)^{n}( h_0, \ldots, h_{n-1} )
\end{eqnarray*}
and $d_1: \BH \rightarrow \C$, given by $d_1( h_0 ) = 1$, is the augmentation.
There is a bounded contracting homotopy given by $s_{n+1}( h_0, \ldots, h_n ) = (1_G, h_0, \ldots, h_n )$,
where $1_G$ is the identity element of $G$, and $s_0 : \C \rightarrow \BG$ is given by $s_0(z) = z(1_G)$.

$\Ext_{\BH}^*( \C, A )$ is the cohomology of the cochain complex
\begin{multline*}
\bHom_{\BH}( \BH, A ) \rightarrow \bHom_{\BH}( \BH \ttensor \BH, A )\\
	\rightarrow \bHom_{\BH}( \BH^{\ttensor 3}, A ) \rightarrow \ldots .
\end{multline*}
By \cite{M2}, for $B$ a bornological algebra, $E$ any bornological space, and 
$F$ any bornological left $B$-module, $\bHom_{B}( B \ttensor E, F ) \isom \bHom( E, F )$.  Thus
$\Ext_{\BH}^*( \C, A )$ is the cohomology of 
\[ \bHom( \C, A ) \rightarrow \bHom( \BH, A ) \rightarrow \\
	\bHom( \BH^{\ttensor 2}, A ) \rightarrow \ldots . \]
Tensoring each of the left $\BH$ modules $\BH^{\ttensor n}$ by $\BG$ over $\BH$ yields 
\begin{multline*}
\ldots \rightarrow \BG \ttensor_{\BH} (\BH \ttensor \BH) \rightarrow \BG \ttensor_{\BH} \BH  \\
		\rightarrow \BG \ttensor_{\BH} \C \rightarrow 0.
\end{multline*}
As $\BG \ttensor_{\BH} \C \isom \BGH$, this reduces to 
\begin{multline*} 
\ldots \rightarrow \BG \ttensor_{\BH} (\BH \ttensor \BH) \rightarrow \BG \ttensor_{\BH} \BH  \\
		\rightarrow \BGH \rightarrow 0 .
\end{multline*}

The bornological isomorphism $\BG \ttensor_{\BH} \BH^{\ttensor n+1} \isom \BG \ttensor \BH^{\ttensor n}$, 
given by \cite{M2}, shows that this is
\[ \ldots \rightarrow \BG \ttensor \BH^{\ttensor 2} \rightarrow \BG \ttensor \BH \rightarrow \BG \\
		\twoheadrightarrow \BGH \rightarrow 0 . \]

The map $\delta_{n+1} : \BG \ttensor \BH^{\ttensor n} \rightarrow \BG \ttensor \BH^{\ttensor n-1}$
is given by
\begin{eqnarray*}
\delta_{n+1} ( g, h_0, \ldots, h_n ) & = & (g h_0, h_1, \ldots, h_n ) - (g, h_0 h_1, h_2, \ldots, h_n ) \\
 & & + \ldots + (-1)^{n}(g, h_0, \ldots, h_{n-1} h_n )\\
 & & + (-1)^{n+1}(g, h_0, \ldots, h_{n-1})
\end{eqnarray*}
while the map $\delta_1 : \BG \rightarrow \BGH$ is given by
$\delta_1( g ) = ( gH )$.  A bounded contracting homotopy $s'_n$ is constructed as follows.
The map $s'_0 : \BGH \rightarrow \BG$ is given by
$s'_0( gH ) = (r_g)$, where $r_g$ is the fixed minimal length representative of the coset $gH$ in $R$ as above.
The map $s'_1: \BG \rightarrow \BG \ttensor \BH$ is given by
$s'_1( g ) = ( r_g, h_g )$, where $r_g \in R$, $h_g \in H$, and $g = r_g h_g$.  The same 
For $n > 1$, $s'_{n}( g, h_0, \ldots, h_{n-2} ) = ( r_g, h_g, h_0, \ldots, h_{n-2})$.

As this is a projective resolution of $\BGH$ over $\BG$, $\Ext^*_{\BG}( \BGH, A )$ can be
calculated as the cohomology of 
\begin{multline*}
\bHom_{\BG}( \BG, A ) \rightarrow \bHom_{\BG}( \BG \ttensor \BH, A )  \\
		\rightarrow \bHom_{\BG}( \BG \ttensor \BH^{\ttensor 2}, A ) \rightarrow \ldots .
\end{multline*}
As $\bHom_{\BG}( \BG \ttensor \BH^{\ttensor n}, A ) \isom \bHom( \BH^{\ttensor n}, A )$, this is also
the same as the cohomology of
\[ \bHom( \C, A ) \rightarrow \bHom( \BH, A ) \rightarrow \bHom( \BH^{\ttensor 2}, A ) \rightarrow \ldots .\]
\end{proof}

Denote by ${\B H^k}(\RH;A) = \prod_{\alpha \in \Lambda} {\B H^k}( H_\alpha; A)$.
\begin{theorem}\label{prellongexact}
Let $G$ and $\RH$ be as above, and let $A$ be a bornological $\BG$-module.  For any 
multiplicative bounding class $\B$, there is a long exact sequence:
\[  \ldots  \to {\B H^k}(G;A) \to {\B H^k}(\RH;A) \to {\B H^{k+1}}(G,\RH;A) \to {\B H^{k+1}}(G;A) \to  \ldots  \]
where for each $H_\alpha \in \RH$, $H_\alpha$ is given the length function restricted from $G$ and
$G/\RH$ is given the minimal weighting function $w$ as above.
\end{theorem}

\begin{proof}
The following short exact sequence admits a bounded $\C$-splitting.
\[ 0 \to \BD \to \BGRH \stackrel{\varepsilon}{\rightarrow} \C \to 0 \]

Applying the bornological $\Ext_\BG^*( \cdot , A )$ functor , yields the
long exact sequence
\begin{multline*}  
\ldots  \to \Ext_\BG^k( \C, A ) \to \Ext_\BG^k( \BGRH, A ) \to \Ext_\BG^k( \BD, A ) \\
	\to \Ext_\BG^{k+1}(\C, A) \to  \ldots .
\end{multline*}

Making use of the isomorphisms,
\begin{eqnarray*}
 \Ext_\BG^k( \BGRH, A ) & = &  \Ext_\BG^k( \oplus_{\alpha \in \Lambda} \BGHi, A) \\
 & \isom & \prod_{\alpha \in \Lambda} \Ext_\BG^k( \BGHi, A ) \\
 & \isom & \prod_{\alpha \in \Lambda} \Ext_{\BHi}^k( \C, A)
\end{eqnarray*}
one obtains the exact sequence
\[
\ldots  \to \Ext_\BG^k( \C, A ) \to \prod_{\alpha \in \Lambda} \Ext_{\BHi}^k( \C, A ) \to \Ext_\BG^k( \BD, A )
	\to \Ext_\BG^{k+1}(\C, A) \to  \ldots .
\]

By definition $\Ext_\BG^k (\C, A) = \B H^{k+1}(G; A)$, and $\Ext_\BG^k( \BD, A ) = \B H^{k+1}( G, \RH; A)$.
\end{proof}

\begin{corollary}
Let $G$ be a finitely generated group with length function $L$, $\B$ a multiplicative bounding class, 
and $\RH = \{ H_\alpha \, | \, \alpha \in \Lambda\}$ a finite family of subgroups.
Suppose that there is a $\BG$-module $V$ such that each $H_\alpha$ is $V$-$\B$IC, in the length
function restricted from $G$.  If $G$ is $V$-$\B$RIC to $\RH$, then $G$ is $V$-$\B$IC.  
In particular, if each $H_\alpha$ is $\B$-SIC and $G$ is $\B$-SRIC to $\RH$, then
$G$ is $\B$-SIC.
\end{corollary}
\begin{proof}
The comparison map yields a commutative diagram with top row
the long exact sequence from Theorem \ref{prellongexact} and the bottom row
the long exact sequence from Theorem \ref{rellongexact}.  The result
follows from the five-lemma.
\end{proof}

The notion of `niceness' defined above has an obvious extension to
free resolutions of modules other than $\Z$ over $\Z[G]$. The following
generalization of Lemma \ref{lem:niceExt} to the relative setting is straightforward.
\begin{lemma}\label{lem:niceExtRel}
Let $G$ be a group equipped with word-length function $L$, $R_*$ a $k$-nice resolution of $M$ over $\Z[G]$, 
$T_* = R_*\otimes\C$, and $\B$ and $\B'$ bounding classes.  Denote by $\mathcal{B}T_*$ the corresponding Frechet completion of 
$T_*$ with respect to the bounding class $\B$, as defined above. Further suppose that the weighted Dehn functions $\{ d_R^{w,n} \}$ 
are ${\B}'$-bounded in dimensions $n <k$, that $\B$ is a right ${\B}'$-class, and that $\B\succeq {\mathcal L}$ .  Then there 
exists a bounded chain null-homotopy $\{s_{n+1} : \mathcal{B}T_n \to \mathcal{B}T_{n+1}\}_{k > n \geq 0}$, implying ${\B}T_*$ is 
a continuous resolution of $\B M$ over $\BG$ through dimension $k$. Here, $\B M$ denotes the completion of $M \tensor \C$ via the
bounding class $\B$.
\end{lemma}

This yields a suitable complex from which we may calculate bounded relative cohomology of $G$ with respect to $\RH$.  

\begin{theorem}\label{thm:relgeocrit}
Suppose the finitely generated group $G$ is $HF^\infty$ relative to the finite family of 
finitely generated subgroups $\RH$.  The following are equivalent:
\begin{enumerate}
\item[(1)]The relative Dehn functions of $G$ relative to $\RH$ are each $\B$-bounded.
\item[(2)]$G$ is $\B$-SRIC with respect to $\RH$.
\item[(3)]The comparison map $\B H^*(G,\RH;A) \to H^*(G,\RH; A)$ is surjective for all bornological
	$\BG$-modules $A$.
\end{enumerate}
\end{theorem}
\begin{proof}

For (1) implies (2), suppose $A$ is a bornological $\BG$-module, and let 
\[ R_* := \ldots \to R_2 \to R_1 \to R_0 \to R_{-1} = \Delta_\Z \to 0 \]
be a nice type $FF^\infty$ resolution of $\Delta_\Z$, the integral augmentation kernel, 
over $\Z[G]$, and let
\[  T_* := \ldots \to T_2 \to T_1 \to T_0 \to T_{-1} = \Delta \to 0 \]
be given by $T_n = R_n \tensor \C$.  Here $\Delta$ is the complex augmentation kernel.
$T_*$ is a type $FF^\infty$ resolution of $\Delta$ over $\C[G]$.
As the relative Dehn functions are $\B$-bounded, the previous lemma
gives that $\B T_*$ is a bornologically projective resolution of $\BD$ over $\BG$.
Let $V_n$ be the complex vector space with one basis element for each generator of
$T_n$ over $\C[G]$.  There are isomorphisms
$T_n \isom \C[G] \tensor V_n$ and $\B T_n \isom \BG \ttensor V_n$.

Apply $\Hom_{\C[G]}( \cdot, A)$ to the deleted resolution $T_*$ 
yields a cochain complex with terms of the form $\Hom_{\C[G]}( T_n, A )$.
Applying $\bHom_{\BG}( \cdot, A)$ to the deleted resolution $\B T_*$
yields a cochain complex with terms of the form $\bHom_{\BG}( \B T_n, A )$.
\begin{eqnarray*}
\Hom_{\C[G]}( T_n, A ) & \isom & \Hom_{\C[G]}( \C[G] \tensor V_n, A )\\
	& \isom & \Hom( V_n, A )\\
	& \isom & \bHom_{\BG}( \BG \ttensor V_n, A )\\
	& \isom & \bHom_{\BG}( \B T_n, A )
\end{eqnarray*}

As $A$ was arbitrary we obtain that $G$ is $\B$-SRIC with respect to $\RH$.

The implication (2) implies (3) is obvious.
For (3) implies (1), follow the proof the implication ($\B$2) implies ($\B$1) of Theorem \ref{thm:geocrit} with the following
modifications.  Replace the absolute cocycles and boundaries, by the relative cocycles and boundaries.
The argument applies nearly verbatim.
\end{proof}

For the remainder of the section, we suppose that $G$ is a finitely presented group which acts cocompactly without 
inversion on a contractible complex $X$, with finite edge stabilizers and finitely generated vertex stabilizers $G_\sigma$.  
The higher weighted Dehn functions of $X$ bound the topological relative Dehn functions of $G$ with respect 
to the $\{G_\sigma\}$.  Applying Theorem \ref{thm:relgeocrit} we obtain the following.

\begin{theorem}
Suppose all of the higher weighted Dehn functions of $C_*(X)$ are $\B$-bounded.  
Then $G$ is $\B$-SRIC with respect to the $\{G_\sigma\}$.
\end{theorem}

To use this result effectively, we must be able to determine how the restricted
length function on $G_\sigma$ behaves, when compared to the usual word-length
function on $G_\sigma$.

\begin{lemma}\label{lem:Osin}
Suppose the first unweighted geometric Dehn function of $X$ is $\B$-bounded.  Then the standard word-length function, 
$L_{G_\sigma}$, on $G_\sigma$, for every $\sigma$, is 
$\B$-equivalent to the length function on $G_\sigma$, induced by the restriction of $L_G$ to $G_\sigma$.
Specifically, there exists a $\nu \in \B$ such that $L_{G_\sigma}(g) \leq \nu( L_G(g) )$ for all $g \in G_\sigma$
and all $\sigma$.
\end{lemma}
\begin{proof}
The finite edge stabilizers imply that the first relative Dehn function ( in the meaning of Osin \cite{Os} ),
is equivalent to the first Dehn function of $X$, by \cite{BC}.  Thus it is $\B$-bounded.
By Lemma 5.4 of \cite{Os}, the distortion of each $H \in \RH$ is bounded by the relative Dehn function.
Thus each $H$ is at most $\B$-distorted.
\end{proof}

The following is a generalization of a result in \cite{JR1}, which states that if a finitely generated group $G$ is
relatively hyperbolic to a family of finitely generated subgroups $\RH$, and if each $H \in \RH$ is $HF^\infty$
and $\mathcal{P}$-SIC, then $G$ is $\mathcal{P}$-SIC.

\begin{corollary}
Suppose the finitely generated group $G$ is relatively hyperbolic with respect to the family of finitely generated 
subgroups $\RH$.  If $\mathcal{L} \preceq \B$, then $G$ is $\B$-SRIC with respect to $\RH$.
\end{corollary}
\begin{proof}
By Mineyev-Yaman \cite{MY}, there is a contractible hyperbolic complex $X$ on which $G$ acts cocompactly with finite 
edge stabilizers, and vertex stabilizers precisely the $H$ and their translates. Lemma \ref{lem:Osin} gives the result.
\end{proof}

\begin{corollary}\label{cor:RHIC}
Suppose the finitely generated group $G$ is relatively hyperbolic with respect to the
family of finitely generated subgroups $\RH$, and $\B$ is a multiplicative bounding class with $\mathcal{L} \preceq \B$.  
For any bornological $\BG$-module $M$, if each $H \in \RH$ is $M$-$\B$IC, then $G$ is $M$-$\B$IC.  In particular, if
each $H$ is $\B$-SIC then $G$ is $\B$-SIC.
\end{corollary}



\section{Two spectral sequences in $\B$-bounded cohomology}

\subsection{The Hochschild-Serre spectral sequence}\label{sect:HSSS}

We begin with a finiteness result which was first proven for
polynomially bounded cohomology in \cite{M1}.

\begin{theorem} 
Let $(G,L)$ be $V$-$\B$IC with respect to the trivial 
$\BG$-module $V\neq 0$. Assume that $V$ is metrizable, 
with distance function $d_V$. Then for each $n\geq 0$,
\[
 {\B}H^n(G;V) = H^n(G;V)\isom \bigoplus^{k_n} V
\]
\end{theorem}

\begin{proof} 
By contradiction. First, note that the weight function on $C_*(EG)$
induces a weight function on $H_*(BG)$ by $w([x]) = \min\{w(x')\ |\
[x'] = [x]\}$. The statement that $G$ is $V$-$\B$IC
is then equivalent to requiring that for each $n\geq 1$ and for all
$[c]\in H^n(BG)$, there exists a $\phi_n\in{\B}$ such that for
all $[x]\in H_n(BG), d_V([c]([x]),0) \leq \phi_n(w([x]))$.
Next, $H^n(G;V)\isom \Hom(H_n(G),V)$ by the Universal Coefficient Theorem. Suppose $H^n(G;V)$ is not a 
finite sum of copies of $V$. This can only happen if $H_n(G)$ is infinite-dimensional over $\C$. Choose a countably
infinite linearly independent set $[x_1],[x_2],\dots,[x_n],\dots$
of elements in $H_n(G)$; we normalize each element so that
$w([x_i])= 1$ for each $1\leq i$. Then for each $i$, fix a
cohomology class $[c_i]\in H^n(G;V)$ with $d_V([c_i]([x_j]),0) = \delta_{ij}$. The 
set $\{[c_1], [c_2],\dots,[c_m],\dots\}$ is a countably infinite
generating set for a subspace $W = \prod_{1}^{\infty} \mathbb
C\subset H^n(G;V)$. An element of $W$ may be written as
\[C = (n_1,n_2,n_3,\dots,n_m,\dots)
\]
indicating that the $[c_i]$-component of $C$ is $n_i[c_i]$. Define a
cohomology class $[C^f]$ by
\begin{equation*}
[C^f] = (f(1),f(2),\dots,f(n),\dots)\in W
\end{equation*}
By construction,
\[
d_V([C^f]([x_m]),0) = f(m)[c_m]([x_m]) = f(m)
\]
for each $m\geq 1$. Choosing the function $f$ to be unbounded (which we can certainly do) makes $[C^f]$  unbounded on the set
of $n$-dimensional homology classes with weight 1. This contradicts the assumption that $(G,L)$ is 
$\B$-isocohomological with respect to $V$, regardless of the
choice of $\B$.
\end{proof} 

The following theorem was proven in \cite{O1} in the context of p.s. $G$-modules.
It was shown in the bornological case in \cite{Ra1}.

\begin{theorem}[\cite{Ra1,O1}]
Let $0 \to (G_1,L_1) \to (G_2,L_2) \to (G_3,L_3) \to 0$ be an extension of groups with word-length, with $G_3$ type ${\mathcal P} FP^\infty$.  There is a
bornological spectral sequence with $E^{p,q}_2 \isom \mathcal{P}H^p( G_3; \mathcal{P}H^q(G_1) )$ and which converges to $\mathcal{P}H^*(G)$, where $\btensor$ is the bornological completed projective tensor product.
\end{theorem}
(An ``extension of groups with word-length" means $L_1$ and $L_3$ are induced by the length function $L_2$.)

A similar result holds with more general bounding classes and coefficients.

\begin{theorem}[Serre Spectral Sequence in $\B$-bounded cohomology]\label{SSS} 
Let $(G_1,L_1)\rightarrowtail (G_2,L_2)\twoheadrightarrow (G_3,L_3)$ be an
extension of groups with word-length. Suppose that  $V$ is a metrizable (bornological) 
$\BGtwo$-module and $(G_3,L_3)$ is $\B$-SIC.  Then there exists a spectral sequence
\[
E_2^{p,q} = {\B}H^p(G_3; {\B}H^q( G_1; V ))\Rightarrow {\B}H^{p+q}(G_2;V)
\]
Hence if $(G_1,L_1)$ is $\B$-SIC, so is $(G_2,L_2)$.
\end{theorem}

\begin{proof} 
Let $(P_*, d_P)$ be the $\B$-completion of the homogeneous bar resolution of $G_2$ and $T_*$ be the tensor product of $P_*$ by $\C$ over $\BGone$, $T_q \isom \BGthree \btensor \BGtwo^{\btensor q}$.  By hypothesis, there exists a resolution $R_*$ for $\C$ over $\BGthree$, with each $R_p$ free with finite rank.

Let $C^{*,*}$ be the first quadrant double complex given by
\[C^{p,q} = \bHom_{\BGthree}( R_p \btensor T_q, V ) \isom \bHom_{\BGthree}(R_p, \bHom( T_q, V) ).\]
Filter this complex by rows.  For a fixed $q$ we have
\[ \ldots \stackrel{\delta_R}{\to} C^{*-1,q} \stackrel{\delta_R}{\to} 
	C^{*,q} \stackrel{\delta_R}{\to} C^{*+1,q} 
	\stackrel{\delta_R}{\to} \ldots \]

The bounded contraction for the complex $R_*$ induces a contraction on $C^{*,q}$,
so $E^{p,q}_1 = 0$ for $p \geq 1$ and $E^{0,q}_1 = \bHom_{\BGthree}( T_q , V ) \isom \bHom_{\BGtwo}(P_q, V)$.
The $E_2$-term is precisely ${\B}H^*( G_2; V )$, and the spectral sequence collapses here.

Filter $C^{*,*}$ by columns.  For a fixed $p$ we have
\[ \ldots \stackrel{\delta_T}{\to} C^{p,*-1} \stackrel{\delta_T}{\to} 
	C^{p,*} \stackrel{\delta_T}{\to} C^{p,*+1} 
	\stackrel{\delta_T}{\to} \ldots \]

By adjointness, $C^{p,q} \isom \bHom( \overline{R_p}, \bHom( T_q, V ) )$, where $\overline{R_p}$ is finite dimensional with 
$R_p \isom \BGthree \btensor \overline{R_p}$.  Let $d^*_T : \bHom(T_q, V ) \to \bHom(T_{q+1}, V )$ be the map induced by $d_T$.
It is clear that $\ker \delta_T = \bHom( \overline{R_p}, \ker d_T^* )$ and $\im \delta_T \subset \bHom( \overline{R_p}, \im d_T^* )$.
That $\bHom( \overline{R_p}, \im d_T^* ) \subset \im \delta_T$ follows from finite dimensionality of $\overline{R_p}$.

Finite dimensionality also implies 
\[ \bHom\left( \overline{R_p}, \frac{\ker d_T^*}{\im d_T^*} \right) \isom \frac{\bHom(\overline{R_p}, \ker d_T^*)}{\bHom( \overline{R_p}, \im d_T^*) }. \]

Thus this spectral sequence has $E_1^{p,q} \isom \bHom_{\BGthree}( R_p, {\B}H^q(H;V) )$ and $E_2^{p,q} \isom {\B}H^p( G_3; {\B}H^q( G_1; V) )$. By a spectral sequence comparison, if $(G_1,L_1)$  is $V$-$\B$IC, so is $(G_2,L_2)$.  Consequently if $(G_1,L_1)$ is $\B$-SIC, isocohomologicality holds for all $\BGtwo$-modules $V$, implying $(G_2, L_2)$ is $\B$-SIC by Theorem \ref{thm:geocrit}.

\end{proof}



\subsection{The spectral sequence associated to a group acting on a complex}\label{sect:CompSS}

Following Section 1.6 of \cite{S}, suppose that a finitely
generated group $G$ acts cocompactly on an acyclic simplicial complex $X$
without inversion. For a simplex $\sigma$ of $X$, denote the
stabilizer of $\sigma$ by $G_\sigma$. Denote by $\Sigma$ a set of
representatives of simplexes of $X$ modulo the $G$ action, and by
$\Sigma_q$ the $q$-dimensional representatives in $\Sigma$.

Let $C_*(X)$ denote the simplicial chain complex of $X$.  As $X$ is acyclic, there
is an exact sequence
\[  0 \leftarrow \C \leftarrow C_0(X) \leftarrow C_1(X) \leftarrow C_2(X) \leftarrow \ldots \]
There is a direct-sum decomposition $C_q(X) \isom \bigoplus_{\sigma \in \Sigma_q} \C[G/G_{\sigma}]$.
For each $\sigma \in \Sigma$, let $P^{\sigma}_k = \C[  G \times_{G_\sigma} (G_\sigma)^{k+1}]$, the usual
simplicial structure on $G_\sigma$ induced up to a $\C[G]$-module.  In this way,
$P^\sigma_\bullet$ is a projective $\C[G]$ resolution of $\C[G/G_{\sigma}]$.  This yields a double complex


\begin{equation}\label{dia:P}
\xymatrix{
	\vdots \ar[d] & \vdots \ar[d]& \vdots\ar[d]  & \\
	\displaystyle\bigoplus_{\sigma \in \Sigma_0} P^\sigma_2 \ar[d] & \displaystyle\bigoplus_{\sigma \in \Sigma_1} P^\sigma_2 \ar[l] \ar[d] & 
		\displaystyle\bigoplus_{\sigma \in \Sigma_2} P^\sigma_2 \ar[l] \ar[d] & \hdots \ar[l]\\
	\displaystyle\bigoplus_{\sigma \in \Sigma_0} P^\sigma_1 \ar[d] & \displaystyle\bigoplus_{\sigma \in \Sigma_1} P^\sigma_1 \ar[l] \ar[d] & 
		\displaystyle\bigoplus_{\sigma \in \Sigma_2} P^\sigma_1 \ar[l] \ar[d] & \hdots \ar[l]\\
	\displaystyle\bigoplus_{\sigma \in \Sigma_0} P^\sigma_0 & \displaystyle\bigoplus_{\sigma \in \Sigma_1} P^\sigma_0 \ar[l] & 
		\displaystyle\bigoplus_{\sigma \in \Sigma_2} P^\sigma_0 \ar[l] & \hdots \ar[l]
}
\end{equation}

As each $\Sigma_q$ is finite, applying $\Hom_{\C[G]}( \cdot, M )$ yields the following double complex.


\begin{equation}\label{dia:Hom}
\xymatrix{
\vdots  & \vdots  & \vdots  & \\
\displaystyle\bigoplus_{\sigma \in \Sigma_0} \Hom_{\C[G]}( P^\sigma_2, M) \ar[r] \ar[u] & 
	\displaystyle\bigoplus_{\sigma \in \Sigma_1} \Hom_{\C[G]}( P^\sigma_2, M) \ar[r] \ar[u] & 
	\displaystyle\bigoplus_{\sigma \in \Sigma_2} \Hom_{\C[G]}( P^\sigma_2, M) \ar[r] \ar[u] & \hdots \\
\displaystyle\bigoplus_{\sigma \in \Sigma_0} \Hom_{\C[G]}( P^\sigma_1, M) \ar[r] \ar[u] & 
	\displaystyle\bigoplus_{\sigma \in \Sigma_1} \Hom_{\C[G]}( P^\sigma_1, M) \ar[r] \ar[u] &
	\displaystyle\bigoplus_{\sigma \in \Sigma_2} \Hom_{\C[G]}( P^\sigma_1, M) \ar[r] \ar[u] & \hdots \\
\displaystyle\bigoplus_{\sigma \in \Sigma_0} \Hom_{\C[G]}( P^\sigma_0, M) \ar[r] \ar[u] & 
	\displaystyle\bigoplus_{\sigma \in \Sigma_1} \Hom_{\C[G]}( P^\sigma_0, M) \ar[r] \ar[u] &
	\displaystyle\bigoplus_{\sigma \in \Sigma_2} \Hom_{\C[G]}( P^\sigma_0, M) \ar[r] \ar[u] & \hdots
}
\end{equation}

Consider the spectral sequence arising from filtering this double complex by columns.
That $P^\sigma_*$ be a projective $\C[G]$ resolution of 
$\C[G/G_{\sigma}]$ means that $\bigoplus_{\sigma \in \Sigma_0} P^\sigma_\bullet$
is a projective resolution of $\bigoplus_{\sigma \in \Sigma_0} \C[G/G_\sigma]$.
The $E^{p,q}_1$-term of this spectral sequence is then 
\begin{eqnarray*}
\Ext^q_{\C[G]}\left( \bigoplus_{\sigma\in\Sigma_p} \C[G/G_\sigma], M \right) & \isom & 
	\prod_{\sigma\in\Sigma_p} \Ext^q_{\C[G]} \left( \C[G/G_\sigma], M \right) \\
 & \isom & \prod_{\sigma\in\Sigma_p} \Ext^q_{\C[G_\sigma]} \left( \C, M \right) \\
 & \isom & \prod_{\sigma\in\Sigma_p} H^q\left(G_\sigma; M \right) \\
\end{eqnarray*}

On the other hand, the total complex of the double complex in equation \ref{dia:P},
serves as a projective resolution of $\C$ over $\C[G]$, yielding a theorem
of Serre.

\begin{theorem}[Serre]\label{thm:serrecomplexss}
For each $\mathbb{C}G$-module $M$, there is a spectral sequence with 
$E^{p,q}_1 \isom \prod_{\sigma\in\Sigma_p} H^q\left(G_\sigma; M \right)$ and which converges to $H^{p+q}(G;M)$.
\end{theorem}

This extends to the $\B$-bounded case, when the stabilizers are given the length 
function restricted from $G$.

Let $C^{\B}_m(X)$ be defined as in the proof of Theorem \ref{thm:geocrit}.
If the higher weighted Dehn functions of $X$ are $\B$-bounded, $C^{\B}_*(X)$
gives a chain complex of complete bornological $\mathcal{H}_{\B,L}(G)$-modules,
endowed with a bounded $\C$-linear contracting homotopy.  There is a natural quotient length, $w$,
defined on $G/G_\sigma$ induced from the length $L$ on $G$ via 
$w( gG_\sigma ) := \min \{ L( gh ) \, | \, h \in G_\sigma \}$.  Denote by $\mathcal{H}_{\B,w}(G/G_\sigma)$ 
the completion $\C[G/G_\sigma]$ under the following family of seminorms.
\[ 
| \sum_{x \in G/G_\sigma} \alpha_x x |_\lambda := \sum_{x\in G/G_\sigma} |\alpha_x| \lambda( w(x) ) \,\,\, \lambda \in \B
\]
There is a bornological isomorphism $C^{\B}_q(X) \isom \bigoplus_{\sigma \in \Sigma_q}\mathcal{H}_{\B,w}(G/G_\sigma)$.
Similarly, let $\B P^{\sigma}_k$ denote the corresponding completion of $P^{\sigma}_k$.  As above, we obtain a
double complex, but of bornological $\BG$-modules.


\begin{equation}\label{dia:bP}\xymatrix{
\vdots \ar[d] & \vdots \ar[d] & \vdots \ar[d] &  \\
\displaystyle\bigoplus_{\sigma \in \Sigma_0} \B P^\sigma_2 \ar[d] & \displaystyle\bigoplus_{\sigma \in \Sigma_1} \B P^\sigma_2 \ar[l] \ar[d] &
	\displaystyle\bigoplus_{\sigma \in \Sigma_2} \B P^\sigma_2 \ar[l] \ar[d] & \hdots \ar[l] \\
\displaystyle\bigoplus_{\sigma \in \Sigma_0} \B P^\sigma_1 \ar[d] & \displaystyle\bigoplus_{\sigma \in \Sigma_1} \B P^\sigma_1 \ar[l] \ar[d] &
	\displaystyle\bigoplus_{\sigma \in \Sigma_2} \B P^\sigma_1 \ar[l] \ar[d] & \hdots \ar[l]\\
\displaystyle\bigoplus_{\sigma \in \Sigma_0} \B P^\sigma_0 & \displaystyle\bigoplus_{\sigma \in \Sigma_1} \B P^\sigma_0 \ar[l] &
	\displaystyle\bigoplus_{\sigma \in \Sigma_2} \B P^\sigma_0 \ar[l] & \hdots \ar[l]
}\end{equation}

For any $\BG$-module $M$, applying the bounded equivariant homomorphism functor
$\bHom_{\BG}( \cdot, M )$ yields the following.

\begin{equation}\label{dia:bHom}\xymatrix{
\vdots & \vdots &  \\
\displaystyle\bigoplus_{\sigma \in \Sigma_0} \bHom_{\BG}( \B P^\sigma_2, M) \ar[r] \ar[u] &
	\displaystyle\bigoplus_{\sigma \in \Sigma_1} \bHom_{\BG}( \B P^\sigma_2, M) \ar[r] \ar[u] & \hdots\\
\displaystyle\bigoplus_{\sigma \in \Sigma_0} \bHom_{\BG}( \B P^\sigma_1, M) \ar[r] \ar[u] & 
	\displaystyle\bigoplus_{\sigma \in \Sigma_1} \bHom_{\BG}( \B P^\sigma_1, M) \ar[r] \ar[u] & \hdots\\
\displaystyle\bigoplus_{\sigma \in \Sigma_0} \bHom_{\BG}( \B P^\sigma_0, M) \ar[r] \ar[u] & 
	\displaystyle\bigoplus_{\sigma \in \Sigma_1} \bHom_{\BG}( \B P^\sigma_0, M) \ar[r] \ar[u] & \hdots
}\end{equation}

As in the non-bornological case above, when filtering by columns
we obtain a spectral sequence that converges to the cohomology of the total complex.
The choice of $w$ on $G/G_\sigma$ ensures a bornological isomorphism 
\[Ext^*_{\BG}( \mathcal{H}_{\B,w}(G/G_\sigma), M ) \isom
Ext^*_{\mathcal{H}_{\B,L}(G_\sigma)}( \C, M ).\]  As above we find 
$E^{p,q}_1 \isom \prod_{\sigma\in\Sigma_p} \B H^q\left(G_\sigma; M \right)$.
Moreover, the total complex of \ref{dia:bP} gives a projective resolution of $\C$
over $\BG$.  This verifies the following theorem.

\begin{theorem}
Suppose all higher weighted Dehn functions of $C_*(X)$ are $\B$-bounded, when the acyclic complex $X$ is equipped with the $1$-skeleton 
weighting.  For each $\BG$-module $M$, there is a spectral sequence with $E_1$-term the product of the 
${\B}H^*( G_\sigma; M)$ which converges to ${\B}H^*(G;M)$.
\end{theorem}

By comparison with the spectral sequence from Theorem \ref{thm:serrecomplexss},
we immediately obtain the following corollary.
\begin{corollary}\label{cor:ComplexSS}
Suppose the acyclic complex $X$ is equipped with the $1$-skeleton weighting, and all higher weighted Dehn functions of $C_*(X)$ are $\B$-bounded.  
If $M$ is a $\BG$-module for which each $(G_\sigma, L)$ is $M$-$\B$IC, then $(G,L)$ is $M$-$\B$IC.
In particular if each $(G_\sigma,L)$ is $\B$-SIC, so is $(G,L)$.
\end{corollary}



\section{Duality groups and the comparison map}

\subsection{Duality and Poincar\'e Duality Groups}

We recall that $G$ is a duality group of dimension\ $n$ if there exists a $G$-module $D$ such that
\[
H^i(G,M) \isom H_{n-i}(G,D\tensor M)
\]
If this is the case, then $D = H^n(G,\Z [G])$ is the \underline{dualizing module}. When $D = \Z $, 
the group is called a \underline{Poicar\'e Duality group}. It is orientable precisely when the action of $G$ on $D$ 
(induced by the right action of $G$ on $\Z [G]$) is trivial. All known orientable Poincar\'e Duality groups 
occur as the fundamental group of a closed orientable aspherical manifold.


\subsection{Isocohomologicality and the fundamental class}

The question of isocohomologicality for oriented duality groups is answered by the following theorem. All 
homology and cohomology groups are taken with coefficients in $\mathbb{C}$.

\begin{theorem}\label{thm:manifold} 
Let $M$ be a compact, closed, orientable manifold of dim.\ $n$ which is 
aspherical ($\widetilde{M}\homotopic *$). Let $G = \pi_1(M)$, and let $\mu''_G\in H^n(M\times M)$ denote the 
fundamental cohomology class in $H^*(M\times M)$ dual to the diagonal embedding $\Delta(M)\subset M\times M$
\footnote{This class is simply the image, under the restriction map $H^*(M\times M,M\times M - \Delta(M))\to H^*(M\times M)$,  
of the Thom class associated to the normal bundle of the diagonal embedding.}. If $\mu''_G$ is in the image of the comparison 
map $\Phi^*_{\B}:{\B}H^*(G\times G)\to H^*(G\times G)$ with respect to a length function $L$ on $G$, then 
$(G,L)$ is $\B$-isocohomological.
\end{theorem}

\begin{proof}
Assume $L$ fixed, and consider the following diagram:

\centerline{\xymatrix{
H^i(G)\ar@<1ex>[r]^{{}_-\cap \mu_G} & H_{n-i}(G)\ar[d]^{\Phi_*^{\B}}\ar@<1ex>[l]^{\mu_G''/_-}\\
{\B}H^i(G)\ar[u]^{\Phi^*_{\B}}\ar@<1ex>[r]^{{}_-\cap \mu^{\B}_G} & {\B}H_{n-i}(G)\ar@<1ex>@{-->}[l]^{?}}}

Here $\mu_G\in H_n(G) = H_n(M)$ denotes the fundamental homology class of $M$. Now ${}_-\cap \mu_G$ is an 
isomorphism with inverse given by $\mu''_G/_-$. The homology class $\mu^{\B}_G\in {\B}H_n(G)$ 
is simply the image of $\mu_G\in H_n(G)$ under the comparison map $\Psi_*^{\B}$. By section \ref{sect:pairings}, 
\[
{}_-\cap\mu^{\B}_G = \Phi^{\B}_*\circ ({}_-\cap\mu_G)\circ \Phi^*_{\B}
\] 
In fact this identify follows from a similar one that holds on the (co)chain level. If there exists a class 
${}_{\B}\mu_G''\in {\B}H^n(G\times G)$ satisfying $\mu_G'' = \Psi_{\B}^*({}_{\B}\mu_G'')$, 
then taking $? = {}_{\B}\mu_G''/_-$ in the above diagram and appealing again to section \ref{sect:HSSS}, we get the second identity
\[
\mu_G''/_- = \Phi^*_{\B}\circ ({}_{\B}\mu_G''/_-)\circ \Phi^{\B}_*
\]
This implies the diagram, with ``?'' so defined, is commutative. The fact that the maps at the top are 
isomorphisms then implies all of the other maps in the diagram are as well.
\end{proof}

[Note: There is a different way of thinking about this result. By the Duality Theorem (Thm.\ 11.10 of \cite{MS}), 
for any basis $\{b_i\}$ of $H^*(G) = H^*(M)$, taken as a (finite-dimensional) graded vector space over $\C$, 
there exists a ``dual" basis $\{b_j^{\sharp}\}$ with $<b_i\cup b_j^{\sharp},\mu_G> = \delta_{ij}$. In terms of 
these bases, $\mu''_G$ is given by the equation
\[
\mu''_G = \sum_i(-1)^{dim(b_i)}b_i\times b_i^{\sharp}
\]
The condition that this class is $\B$-bounded then forces each of the $b_i$'s (and hence also the $b_j^{\sharp}$'s) to be $\B$-bounded, via linear independence.]

When $BG$ has the homotopy type of an oriented manifold with boundary, we have a similar result.

\begin{theorem} 
Suppose $(G,L)$ is a group with word-length, such that $BG\homotopic M$ an oriented compact $n$-dimensional 
manifold with connected boundary $\partial M$. Assume also that $\partial M$ is aspherical, and incompressibly 
embedded in $M$ (i.e., the induced map on fundamental groups $\pi_1(\partial M)\to \pi_1(M)$ is injective). Let 
$D(M) = M\underset{\partial(M)}\cup M$ denote the double of $M$ along its boundary. If the fundamental cohomology 
classes of both $D(M)$ and  $\partial M$ are both in the image of the comparison map for a bounding class 
$\B$ (in the manner described by theorem \ref{thm:manifold}), then $G$ is $\B$-isocohomological.
\end{theorem}

\begin{proof} 
Let $G_i' = \pi_1(\partial M)$ and $G'' = \pi_1(D(M))$. By Van Kampen's theorem, $G'' \isom G\underset{G'}* G$; 
moreover, the incompressibility of $\partial M$ in $M$ implies $D(M)\homotopic K(G'',1)$ is aspherical. Now consider the diagram
\begin{center}
\centerline{\small{
\xymatrix{
\dots \ar[r] & {\B}H^{j-1}(G')\ar[r]^{\delta}\ar[d]
& {\B}H^{j}(G'')\ar[r]\ar[d] & {\B}H^j(G)\oplus {\B}H^j(G)\ar[r]\ar[d] & {\B}H^j(G')\ar[r]\ar[d] &\dots\\
\dots \ar[r] & H^{j-1}(G')\ar[r]^{\delta}
& H^{j}(G'')\ar[r] & H^j(G)\oplus H^j(G)\ar[r] & H^j(G')\ar[r]^{\delta} &\dots  
}}
}
\end{center}

Both the top and bottom sequences are derived from the collapsing of the spectral sequence associated to a 
group acting on a complex (in this case, a tree with two edges and three vertices, representing the amalgamated 
free product). The vertical maps are induced by the comparison transformation ${\B}H^*(_-)\to H^*(_-)$, implying 
the diagram is commutative. By Theorem \ref{thm:manifold}, the comparison map is an isomorphism for both $G'$ and 
$G''$ (both of whose classifying spaces are represented by compact, oriented finite-dimensional manifolds without 
boundary). The result follows by the five-lemma.
\end{proof}

It is not clear if this is the best possible result when the boundary is non-empty, i.e., whether $\B$-isocohomologicality 
for $G$ could be guaranteed by the $\B$-boundedness of a single cohomology class. It is also not clear what one can say 
in general if either $\partial M$ is not aspherical, or if it is, but not incompressibly embedded in $M$. All of these 
situations would seem to deserve further attention.


\subsection{$\B$-duality groups}

Using the pairing operations of section \ref{sect:pairings}, one has an obvious extension of the definition of a duality group to the $\B$-bounded setting.

\begin{definition} 
Given a bounding class $\B$ and a group with word-length $(G,L)$, we say that $G$ is a 
\underline{$\B$-duality group of dimension $n$} if there exists an $\BG$-module 
$D_{\B}$ and a ``fundamental class" $\mu_{\B}\in {\B}H_n(G,D_{\B})$ with
\begin{equation*}
{\B}H^i(G,V)\overset{_-\cap\mu^{\B}}{\underset{\isom}\longrightarrow} {\B}H_{n-i}(G,D\widehat{\tensor} V)
\end{equation*}
for all $\BG$ modules $V$.
\end{definition}

\begin{theorem} 
Let $\B$ be a bounding class, and $(G,L)$ a $\B$-duality group with duality module $D_{\B}$. 
Suppose $\mu^{\B}$ is in the image of the comparison map $\Psi_*^{\B}$.  Then
\begin{itemize}
\item If $D_{\B}$ is finite-dimensional over $\C$, $(G,L)$ is strongly monocohomological 
	(that is, the comparison map is injective in cohomology for all bornological $\BG$-modules $V$.
\item If $D_{\B}$ is infinite-dimensional over $\C$, $(G,L)$ is monocohomological for all bornological 
	$H_{{\B},L}(G)$-modules $V$ which are finite-dimensional over $\C$.
\end{itemize}
\end{theorem}

\begin{proof} 
Choose $\mu_D\in H_n(G,D)$ with $\Phi_n^{\B}(\mu_D) = \mu^{\B}_D$.  We can consider a diagram analogous to that of Theorem (\ref{thm:manifold}):

\centerline{\xymatrix{
H^i(G,V)\ar@<1ex>[r]^(.4){{}_-\cap \mu_D} & H_{n-i}(G,D_{\B}\tensor V)\ar[d]^{\Phi_*^{\B}}\\
{\B}H^i(G,V)\ar[u]^{\Phi^*_{\B}}\ar@<1ex>[r]^(.4){{}_-\cap \mu^{\B}_D} & {\B}H_{n-i}(G,D_{\B}\widehat{\tensor} V)}}

At issue in this diagram is the difference between $D_{\B}\tensor V$ and $D_{\B}\widehat{\tensor} V$. However, if either $D_{\B}$ or $V$ is finite-dimensional over $\C$, this difference vanishes and the diagram commutes, verifying injectivity of the comparison map in the cases indicated. Note that we do not assume the top horizontal map in the above diagram is an isomorphism.
\end{proof}

{\bf\underline{Remark}} Ideally, one would like to prove the diagram commutes whenever $\mu^{\B}$ is in the image of the comparison map. However, we have not yet been able to show this.


\subsection{Two solvmanifolds}

We construct examples of groups $\pi$, admitting closed oriented compact manifold models for $B\pi$ 
of small dimension, for which the comparison map fails to be surjective.

Let $L_n$ denote the standard word-length function on $\Z^n$, and set
\[
LWL_n(g) = \log(1 + L_n(g))
\]
This is still a length function on $\Z^n$, but it is not $\B$-equivalent to $L_n$ unless $\mathcal{E} \preceq \mathcal{B}$.

\begin{proposition}\label{prop:log} 
Let ${\B}H_{log}^*(\Z^n)$ denote the $\B$-bounded 
cohomology of the group with word-length $(\Z^n,LWL_n)$. Then for all ${\B}\prec {{\mathcal{E}}}$,
\begin{equation*}
0 = \Phi^*_{\B}:{\B}H^*_{log}(\Z^n)\to H^*(\Z^n),\quad *>0
\end{equation*}
\end{proposition}

\begin{proof} 
When $*=1$, elements of ${\B}H_{log}^1(\Z^m)$ correspond bijectively to group 
homomorphisms from $\Z^m $ to $\C$, equipped with its usual norm. The norm of any non-zero 
homomorphism grows linearly with respect to the standard word-length function on $\Z $. This means 
it grows exponentially as a function of $LWL_m$. When ${\B}\prec {\mathcal{E}}$, this is impossible, 
implying ${\B}H_{log}^1(\Z^m) = 0$ for all $m\ge 1$. This verifies the proposition in the case $n=1$.

Suppose now that $n > 1$. There is a commuting diagram of short-exact sequences of groups with word-length

\begin{center}
\centerline{
\xymatrix
{
\mathbb Z_{LWL_{n-1}}^{n-1}\ar[r]\ar[d] & \mathbb Z_{LWL_{n-1}}^{n-1}\times \mathbb Z_{st}\ar[r]\ar[d]& \mathbb Z_{st}\ar[d]\\
\mathbb Z_{st}^{n-1}\ar[r] &\mathbb Z_{st}^n\ar[r] &\mathbb Z_{st}
}}
\end{center}

By induction, we may assume the comparison map ${\B}H^*_{log}(\Z^{n-1})\to H^*(\Z^{n-1})$ is zero for $*>0$. Both sequences satisfy the conditions for the Serre spectral sequence in $\B$-bounded cohomology to exist. A spectral sequence argument then shows the vertical map in the middle $\mathbb Z_{LWL_{n-1}}^{n-1}\times \mathbb Z_{st}\to \mathbb Z^n_{st}$ must be zero in $\B$-bounded cohomology for $* > 1$, implying the same for the composite map
\[
\mathbb Z^n_{LWL_n}\to \mathbb Z^{n-1}_{LWL_{n-1}}\times \mathbb Z_{LWL_1}\to \mathbb Z^{n-1}_{LWL_{n-1}}\times \mathbb Z_{st}\to \mathbb Z^n_{st}
\]
Moreover, for the standard word-length the comparison map induces an isomorphism ${\B}H^*_{st}(\Z^n)\overset\cong{\to} H^*(\Z^n)$. We may then conclude that the comparison map ${\B}H^*_{log}(\Z^n)\to H^*(\Z^n)$ is zero for $* > 1$. As we have already shown it is zero for $*=1$, this completes the induction step.
\end{proof}

\underline{\bf Example 1 } [Gromov] As above, assume ${{\mathcal{L}}}\preceq{\B}\prec {\mathcal{E}}$ and let $G$ be the semi-direct product
$\Z^2\rtimes \Z $, where $\Z $ act on $\Z^2$ by the representation
\[
\begin{pmatrix}
2 & 1\\
1 & 1
\end{pmatrix}
\]
This is a split-extension of $\Z $ by $\Z^2$; moreover, $\Z^2$ has 
exponential distortion in $G$. This is equivalent to saying the induced word-length function on 
$\Z^2$ coming from the embedding in $G$ is (linearly) equivalent to $LWL_2$. For the base 
group $\Z $, the word-length function induced by the projection $G\twoheadrightarrow \Z $ 
is the standard one. Now the Hochschild-Serre spectral sequence in ordinary cohomology for this extension 
satisfies $E_2^{**} = E_{\infty}^{**}$ for dimensional reasons. Embedding $\Z^2\rtimes \Z $ 
in the solvable Lie group $\R^2\rtimes\R$, the action of the base on the fiber 
(over $\R$) is similar to the action given by $r\circ (r_1,r_2) = (e^{\lambda r}r_1,e^{-\lambda r}r_2)$. 
The first exterior power of this representation has no invariant subspaces, while the second exterior power 
is the identity. Hence $E_2^{0,1} = H^0(\Z ;H^1(\Z^2)) = 0$, while  
$E_2^{0,2} = H^0(\Z ;H^2(\Z^2))\isom \C$. On classifying spaces the short-exact 
sequence $\Z^2\rightarrowtail G\twoheadrightarrow \Z $ produces a fibration sequence of 
closed oriented manifolds and orientation-preserving maps. This yields a Poincar\'e Duality map on the 
$E_2^{**}$-term of the spectral sequence for $H^*(G)$. By this duality, we conclude $E_2^{0,1}$ is dual to 
$E_2^{1,1}$ which therefore must also be zero (we already knew 
$H^1(\mathbb{Z}) = E_2^{1,0} \isom E_2^{0,2}= H^0(\Z ;H^2(\Z^2))\isom \C$). This 
gives an isomorphism $H^*(G)\isom H^*(\mathbb{Z})\tensor H^*(\Z^2)$, although there is no homomorphism 
of groups inducing it. If we denote by $t_i\in H^i(G)$ the element corresponding to the generator of 
$H^i(\Z^i),\,\, i = 1,2$ (after fixing a preferred orientation of $BG$), then

\begin{proposition}\label{prop:ex1} 
The cohomology class $t_2\in H^2(G)\isom \C$ cannot 
lie in the image of the comparison map $\Psi_{\B}^2:{\B}H^2(G)\to H^2(G)$ whenever 
${\B}\prec {\mathcal{E}}$.
\end{proposition}

\begin{proof} 
The comparison map is natural with respect to those maps induced by group homomorphisms, implying the existence of a commuting diagram

\begin{center}
\centerline{
\xymatrix
{
{\B}H^2(G)\ar[r]\ar[d]^{\Psi_{\B}^2} & {\B}H^2_{log}(\Z^2)\ar[d]^{\Psi_{\B}^2}\\
H^2(G)\ar[r]\ar[r] & H^2(\Z^2)
}}
\end{center}

For ${\B}\prec {\mathcal{E}}$, the map on the right is trivial by Proposition \ref{prop:log}, 
while the  spectral sequence argument for ordinary cohomology just given shows the lower horizontal 
map sends $t_2$ non-trivially to the generator of $H^2(\Z^2)\isom \C$. 
Thus $t_2$ cannot be in the image of $\Psi_{\B}^2$ (this is in the spirit to Gromov's original argument referenced above).
\end{proof}

With some additional work, one can also show $t_1t_2\in H^3(G)$ is not in the image of $\Psi_{\B}^3$ 
whenever ${\B}\prec {\mathcal{E}}$. Of course, by Theorem \ref{thm:manifold}, The dual fundamental 
class $u'\in H^3(G\times G)$ cannot be in the image of $\Psi_{\B}^3$ for ${\B}\prec {\mathcal{E}}$.

There are some additional consequences of this first example worth noting (with ${{\mathcal{P}}}\preceq{\B}\prec {\mathcal{E}}$).
\begin{itemize}
\item All surface groups are non-positively curved - hence $\B$-IC - so $3$ is the 
	lowest dimension for which there can exist a closed oriented $K(\pi,1)$ manifold with non-$\B$-bounded cohomology.
\item Nilpotent groups are $\B$-IC when ${{\mathcal{P}}}\prec {\B}$ 
	\cite{O1}, \cite{JR1} , so solvable groups are the simplest types of groups which could have non-$\B$-bounded 
	cohomology for ${{\mathcal{P}}}\preceq {\B}$.
\item For finitely-generated groups, all $1$-dimensional cohomology classes exhibit linear growth with respect 
	to the word-length function, so cohomological dimension $2$ is the first dimension in which classes not 
	$\B$-bounded with respect to the word-length function could occur.
\item If the first Dehn function of $G$ were $\B$-bounded, $G$ would have to be strongly 
	$\B$-isocohomological in cohomology dimensions $1$ and $2$. By contradiction, we recover the 
	result of Gersten [Ge1] that the first Dehn function of $G$ must be (at least) exponential.
\end{itemize}

\underline{\bf Example 2} [Arzhantseva-Osin] 
Let $\phi : \Z^2 \to SL_3( \Z )$ be an injection sending the usual generators of $\Z^2$ to to semi-simple matrices
with real spectrum.  Denote by  $H$ be the semi-direct product $\Z^3\rtimes \Z^2$ where $\Z^2$ acts via the 
representation induced by $\phi$.

The classifying space $BH$ is homotopy-equivalent to a $5$-dimensional closed, compact, and oriented solvmanifold 
$M^5$. It is shown in \cite{AO} that $\Z^3$ is exponentially distorted in $H$ in a manner similar to the previous example.

\begin{theorem} 
There exists a cohomology class $t_3\in H^3(H)$ not in the image of $\Psi_{\B}^3$ for any ${\B}\prec {\mathcal{E}}$.
\end{theorem}

\begin{proof} 
On the level of classifying spaces, the short-exact sequence 
$\Z^3\overset{i}{\rightarrowtail} H\overset{p}{\twoheadrightarrow} \Z^2$ corresponds 
to a fibration  sequence of closed oriented compact manifolds, with the maps preserving orientation. Thus 
the top-dimensional cohomology class $\mu_5\in H^5(H)$ satisfies $\mu_5 = \mu_3\mu_2$ where $\mu_3$ maps 
under $i^*$ to $0\ne \mu'_3\in H^3(\Z^3)^{\Z^2}$ (the $\Z^2$-invariant fundamental 
cohomology class of $\Z^3$), and $\mu_2 = p^*(\mu_2')$ where $\mu_2'\in H^2(\Z^2)$ is the 
fundamental cohomology class for $\Z^2$.

As before, there is a commuting diagram

\begin{center}
\centerline{
\xymatrix
{
{\B}H^3(H)\ar[r]\ar[d]^{\Psi_{\B}^3} & {\B}H^3_{log}(\Z^3)\ar[d]^{\Psi_{\B}^3}\\
H^3(G)\ar[r]\ar[r] & H^3(\Z^3)
}}
\end{center}

where the map on the right is zero. The result follows.
\end{proof}

By a more detailed analysis, one can conclude that $\mu_5\in H^5(H)\isom \C$ is not $\B$-bounded 
for any ${\B}\prec {\mathcal{E}}$, and by Theorem \ref{thm:manifold}, we know the same for the dual 
fundamental class in $H^5(H\times H)$. However, this example is important for another reason.

\begin{corollary} The Dehn functions of $H$ are not $\B$-equivalent for any bounding class 
${\B}\prec {\mathcal{E}}$. Precisely, the first Dehn function is quadratic, while the second 
Dehn function is at least simple exponential.
\end{corollary}

\begin{proof} 
The first Dehn function of $H$ was computed in \cite{AO}, where it was shown to be quadratic. 
If the second Dehn function were $\B$-bounded for some 
${\mathcal{L}} \preceq {\B} \prec {\mathcal{E}}$, then by Theorem \ref{thm:geocritRest}, the group $H$ would have to 
be $\B$-isocohomological through dimension $3$ contradicting the previous result. So the second 
Dehn function must be at least simple exponential. 
\end{proof}


\subsection{More on the comparison map}

We have shown the comparison map fails to be surjective in general, at least for bounding classes 
${\B}\prec {\mathcal{E}}$. It is natural to ask whether this map also fails to be injective. 
The next theorem answers this question.

\begin{theorem}\label{thm:surj} 
Let $(G,L_{st})$ be a discrete group with standard word-length function, 
with $BG\homotopic Y$ a finite complex. If $\B$ is a bounding class for which the comparison 
map $\Phi_{\B}^*(G):{\B}H^*(G)\to H^*(G)$ fails to be surjective, then there is another 
group ${\mathfrak F}(G)$, depending functorially on $G$ up to homotopy, for which the comparison map 
$\Phi_{\B}^*$ fails to be injective.
\end{theorem}

\begin{proof} 
As $BG$ is homotopically finite, we may construct a finitely-generated hyperbolic group 
${\mathfrak H}(G)$ and a map $p_G:{\mathfrak H}(G)\to G$ which induces an injection in group cohomology 
with trivial coefficients \cite{Gromov, Charney-Davis, Davis-Janu.}. Also, for any discrete group $G'$, a  
classical construction allow us to embed $G'$ in an acyclic group $A(G')$, where the inclusion $i_{G'}:G'\hookrightarrow A(G')$ 
is a functorial construction in $G'$. If $G'$ is finitely-generated and equipped with the standard word-length 
function, we can arrange for the image of $G'$ in $A(G')$ to be non-distorted. Abbreviate ${\mathfrak H}(G)$ 
as $C$, and let $A_1 = G\times A(C)$, $A_2 = A(C)$. There are inclusions
\begin{gather}
C\hookrightarrow A_1,\,\, g\mapsto (p_G(g),i_C(g)),\\
C\hookrightarrow A_2,\,\, g\mapsto i_C(g)
\end{gather}
Let $A_3 = A_1\underset{C}{*}A_2$. By the spectral sequence of section \ref{sect:CompSS}, there is a commuting diagram of Mayer-Vietoris sequences
\begin{center}
\centerline{\small{
\xymatrix{
\dots \ar[r] & {\B}H^{j-1}(C)\ar[r]^{\delta}\ar[d]^{\isom}
& {\B}H^{j}(A_3)\ar[r]\ar[d] & {\B}H^j(A_1)\oplus {\B}H^j(A_2)\ar[r]\ar[d] & {\B}H^j(C)\ar[r]\ar[d]^{\isom} 
&\dots\\
\dots \ar[r] & H^{j-1}(C)\ar[r]^{\delta}
& H^{j}(A_3)\ar[r] & H^j(A_1)\oplus H^j(A_2)\ar[r] & H^j(C)\ar[r]^{\delta} & \dots
}}
}\end{center}

Because $C$ is finitely-generated hyperbolic, the comparison map for $C$ is an isomorphism for all 
${{\mathcal{L}}}\preceq {\B}$. Moreover, $H^*(A_2) = 0$ for $*>0$, and $F^*(A_1)\isom F^*(G)\tensor F^*(A(C))$ 
for $F^*(_-) = {\B}H^*(_-), H^*(_-)$. Hence the cokernel of the comparison map for $A_1$ is naturally isomorphic 
to the cokernel of the comparison map for $G$. The injectivity of $H^*(G)\to H^*(C)$ implies the map 
$H^j(A_3)\to H^j(A_1)\oplus H^j(A_2)$ is zero for $j > 0$. The result is an injection

\[
	\coker\left(\Phi_{\B}^*:{\B}H^*(A_1)\oplus {\B}H^*(A_2)\to H^*(A_1)\oplus H^*(A_2)\right)
		\hookrightarrow \ker\left(\Phi^{*+1}_{\B}:{\B}H^{*+1}(A_3)\to H^{*+1}(A_3)\right)
\]

Define ${\mathfrak F}(G) = A_3$. If $\coker(\Phi^m_{\B}(G)) \neq 0$, then 
$\ker(\Phi^{m+1}_{\B}({\mathfrak F}(G)))\neq 0$. The acyclic group construction $G\mapsto A(G)$ 
can be done functorially, as can the hyperbolization of the finite complex $Y$. However, this requires 
choosing a finite complex $Y\homotopic BG$, which, on the category of type finitely-presented $FL$ groups, 
is functorial only up to homotopy.
\end{proof}

\begin{corollary} There exist discrete groups equipped with standard word-length function for 
which the comparison map $\Phi_{\B}^3$ fails to be injective for all ${\B}\prec {\mathcal{E}}$.
\end{corollary}

\begin{proof} Let $G$ be the group in Proposition \ref{prop:ex1}. By the previous theorem, 
$\Phi_{\B}^3({\mathfrak F}(G))$ cannot be an injection for any ${\B}\prec {\mathcal{E}}$.
\end{proof}

It should be noted that the groups resulting from the above constructions will typically have large classifying 
spaces, even when $BG$ has the homotopy type of a relatively simple complex. The following alternative construction provides a more geometric model for the acyclic ``envelope'' used above. Again, assume $G$ is type $FL$, so that $BG\simeq Y$ a finite complex. According to recent work of Leary [L], we may construct a diagram
\begin{center}
\centerline{
\xymatrix{
T_Y\ar@{>->}[r]\ar[d] & T_{\widehat{Y}}\ar[d]\\
Y\ar@{>->}[r] & \widehat{Y}
}}
\end{center}
where $\widehat{Y}$ denotes the cone on $Y$ (which can be done so as to be functorial in $Y$ and preserve finiteness), and where $T_X$ denotes the ``metric'' Kan-Thurston space over $X$. By [L], this is a CAT(0)-space (hence aspherical) whose construction is functorial on the category of finite complexes, for which the map $T_X\to X$ is a homology isomorphism. Thus in the above setup, we can replace $C$ by $C_1 := \pi_1(T_Y)$ and $A(C)$ by $C_2 := \pi_1(\widehat{Y})$, and repeat the construction with $A_1 = G\times C_2$, $A_2 = C_2$, ${\mathfrak F}(G) = A_3 = A_1\underset{C_1}{*} A_2$, the difference now being that $A_1$, $A_2$ as well as the amalgamated product $A_3$ are all of type $FL$. Because CAT(0)-groups admit a synchronous linear combing, they are $\B$-SIC for all ${\B}\succeq \mathcal{P}$. Hence

\begin{theorem}\label{thm:surj2} Let $(G,L_{st})$ be a group of type $FL$ with standard word-length function, 
where $BG\homotopic Y$ a finite complex. If $\B\succeq {\mathcal P}$ is a bounding class for which the comparison map $\Psi_{\B}(G):{\B}H^*(G)\to H^*(G)$ fails to be surjective, then there is another
group ${\mathfrak F}(G)$ of type $FL$, depending functorially on $G$ up to homotopy, for which the comparison map $\Phi_{\B}^*$ fails to be injective.
\end{theorem}

\begin{corollary} There exist discrete groups equipped with standard word-length function of type $FL$ for 
which the comparison map $\Phi_{\B}^3$ fails to be injective for all ${\mathcal P}\preceq {\B}\prec {\mathcal{E}}$.
\end{corollary}





\end{document}